\documentclass[a4paper,reqno,12pt]{amsart}
\usepackage[T1]{fontenc}
\usepackage[utf8]{inputenc}
\usepackage[english]{babel}

\usepackage{amsmath}
\usepackage{amsthm}
\usepackage{amssymb}
\usepackage{mathtools}
\usepackage{amsmath}
\usepackage{amsfonts}
\usepackage{esint}
\usepackage{yfonts}
\usepackage{quoting}
\usepackage{graphics}
\usepackage{booktabs}
\usepackage{bm}
\usepackage{bbm}
\usepackage{lmodern}
\usepackage{stmaryrd}
\usepackage{caption}
\usepackage{enumitem}			
\usepackage{tikz}
\usepackage{bookmark}
\usepackage{pgfplots}
\usepgfplotslibrary{fillbetween}
\usepackage{orcidlink}
\usepackage{mathrsfs}
\usepackage[mathscr]{euscript}

\usepackage{hyperref}
\hypersetup{
	colorlinks=true,
	linkcolor=blue,
	citecolor=blue,
	urlcolor=black,
	linktoc=all
}

\usepackage[margin=1.1in]{geometry}

\quotingsetup{font=small}

\theoremstyle{plain}
\newtheorem{proposition}{Proposition}[section]

\theoremstyle{plain}
\newtheorem{theorem}{Theorem}[section]
\numberwithin{equation}{section}	

\theoremstyle{plain}
\newtheorem{lemma}[theorem]{Lemma}

\theoremstyle{plain}
\newtheorem{corollary}{Corollary}[theorem]

\theoremstyle{definition}

\theoremstyle{definition}

\newtheorem{problemA}{Problem}

\DeclarePairedDelimiter{\abs}{\lvert}{\rvert}		
\DeclarePairedDelimiter{\norma}{\lVert}{\rVert}		

\DeclareMathOperator{\defi}{def}
\DeclareMathOperator{\tr}{tr}				
\DeclareMathOperator{\Id}{Id_\textit{n}}	
\DeclareMathOperator{\diver}{div}			

\newcommand{\numberset}{\mathbb}
\newcommand{\N}{\numberset{N}}			
\newcommand{\R}{\numberset{R}}			
\newcommand{\sfera}{\numberset{S}}		
\newcommand{\past}{p^\ast}				
\newcommand{\loc}{{\rm loc}}			
\newcommand{\stressu}{{a\!\left(\nabla u\right)}}	
\newcommand{\stressv}{{a\!\left(\nabla v\right)}}	

\def\Xint#1{\mathchoice
	{\XXint\displaystyle\textstyle{#1}}%
	{\XXint\textstyle\scriptstyle{#1}}%
	{\XXint\scriptstyle\scriptscriptstyle{#1}}%
	{\XXint\scriptscriptstyle\scriptscriptstyle{#1}}%
	\!\int}
\def\XXint#1#2#3{{\setbox0=\hbox{$#1{#2#3}{\int}$ }
		\vcenter{\hbox{$#2#3$ }}\kern-.6\wd0}}

\def\dashint{\Xint-}

\usepackage{enumitem}

\allowdisplaybreaks

\begin{document}
		
		\title[The anisotropic critical~$p$-Laplace equation]{On the anisotropic critical~$p$-Laplace equation: classification, decomposition, and stability results}
		
		\author[Carlo Alberto Antonini]{Carlo Alberto Antonini \orcidlink{0000-0002-7663-1090}}
		\address[]{Carlo Alberto Antonini. Istituto Nazionale di Alta Matematica ‘Francesco Severi’ (INdAM) and Dipartimento di Matematica e Informatica ‘Ulisse Dini’, Università degli Studi di Firenze, Viale Giovanni Battista Morgagni 67/A, 50134 Firenze, Italy}
		\email{antonini@altamatematica.it}
		
		\author[Giulio Ciraolo]{Giulio Ciraolo \orcidlink{0000-0002-9308-0147}}
		\address[]{Giulio Ciraolo. Dipartimento di Matematica ‘Federigo Enriques’, Università degli Studi di Milano, Via Cesare Saldini 50, 20133, Milan, Italy}
		\email{giulio.ciraolo@unimi.it}	
		
		\author[Michele Gatti]{Michele Gatti \orcidlink{0009-0002-6686-9684}}
		\address[]{Michele Gatti. Dipartimento di Matematica ‘Federigo Enriques’, Università degli Studi di Milano, Via Cesare Saldini 50, 20133, Milan, Italy}
		\email{michele.gatti1@unimi.it}
		
		\subjclass[2020]{Primary 35B33, 35B35, 35J62; Secondary 35J20}
		\date{\today}
		\dedicatory{}
		\keywords{Anisotropic critical $p$-Laplace equation, quantitative estimates, quasilinear elliptic equations, Struwe's decomposition, stability}
		
		\begin{abstract}
			We investigate both qualitative and quantitative issues related to the classification of non-negative energy solutions to the anisotropic critical~$p$-Laplace equation in~$\R^n$, for~$1<p<n$.
			
			Specifically, we establish an anisotropic version of Struwe's decomposition, along with the interaction estimate for the family of bubbles in this decomposition. Moreover, we provide a short proof of the classification result as well as a quantitative stability result, proving that every energy solution to a perturbation of the anisotropic critical equation must be closed to a bubble, in the absence of bubbling.
		\end{abstract}
		
		\maketitle
		
		\setcounter{tocdepth}{1}	
		\tableofcontents
	
	
	\section{Introduction}
	\label{sec:intro-stab-anis}
	
	For~$n \in \N$, and~$1<p<n$, the critical anisotropic~$p$-Laplace equation is given by
	\begin{equation}
	\label{eq:critica-anisot}
		\Delta_p^{H} u + u^{\past-1} = 0 \quad\text{in } \R^n,
	\end{equation}
	where~$H:\R^n \to [0,+\infty)$ is a norm,~$\past \coloneqq \frac{np}{n-p}$ is the critical Sobolev exponent, and
	\begin{equation*}
		\Delta_p^{H} u \coloneqq \frac{1}{p} \diver \!\left(\nabla_\xi H^p(\nabla u)\right)
	\end{equation*}
	is the anisotropic~$p$-Laplacian, also known as Finsler~$p$-Laplacian. When~$H$ is the Euclidean norm, this reduces to the classical~$p$-Laplace operator~$\Delta_p u=\diver(\abs*{\nabla u}^{p-2}\nabla u)$.
	
	The critical equation~\eqref{eq:critica-anisot} is closely related to the anisotropic Sobolev inequality
	\begin{equation*}
		S_{p} \,\norma*{u}_{L^{\past}\!(\R^n)} \leq \norma*{H(\nabla u)}_{L^{p}(\R^n)},
	\end{equation*}
	where
	\begin{equation}
	\label{eq:def-SH}
		S_{p} \coloneqq \inf_{u \in \mathcal{D}^{1,p}(\R^n) \setminus \{0\}}\frac{\norma*{H(\nabla u)}_{L^{p}(\R^n)}}{\norma*{u}_{L^{\past}\!(\R^n)}}.
	\end{equation}
	This inequality was proved, employing a mass transportation approach, by Cordero-Erausquin, Nazaret \& Villani~\cite{cor-n-vil}, who also characterized the functions which attain equality. Among these functions, we define the~$(p,H)$-\textit{bubbles} as
	\begin{equation}
	\label{eq:pH-bubb}
		U_p[z,\lambda] (x) \coloneqq \left( \frac{\lambda^\frac{1}{p-1} \, n^\frac{1}{p} \left(\frac{n-p}{p-1}\right)^{\!\!\frac{p-1}{p}}}{\lambda^\frac{p}{p-1}+H_0^\frac{p}{p-1}(x-z)} \right)^{\!\!\frac{n-p}{p}} \!,
	\end{equation}
	where~$\lambda>0$ and~$z \in \R^n$ are the scaling and translation parameters, respectively, and~$H_0$ is the dual norm of~$H$, defined by
	\begin{equation}
	\label{eq:def-H0}
		H_0(x) \coloneqq \sup_{\xi \neq 0} \frac{\left\langle x, \xi \right\rangle}{H(\xi)} \quad\text{for all } x \in \R^n.
	\end{equation}
	
	The~$(p,H)$-bubbles satisfy the critical equation~\eqref{eq:critica-anisot} and, setting~$U = U_p[z,\lambda]$, one has that~$U \in \mathcal{D}^{1,p}(\R^n) \cap L^\infty(\R^n)$ with
	\begin{equation}
	\label{eq:energ-bubb-anisot}
		\norma{U}_{L^{\past}\!(\R^n)}^{\past} = \norma{H(\nabla U)}_{L^{p}(\R^n)}^p = S^n_p.
	\end{equation}
	We recall here that the homogeneous Sobolev space is defined as 
	\begin{equation*}
		\mathcal{D}^{1,p}(\R^n) \coloneqq \{ u \in L^{\past}\!(\R^n) \,\lvert\, \nabla u \in L^p(\R^n) \}.
	\end{equation*}
	
	The Euclidean case has been the object of extensive study, including classification results~\cite{cgs,catino,cl,cg-classif,dm,gnn,ou,sciu,sun-wang,vet}, qualitative properties~\cite{alves,benci-cer,merc-will,struwe}, and quantitative stability~\cite{ccg,cfm,cg-plap,dsw,fg,mg-plap,liu-zhang}, which we will discuss in detail later. Of course, the case $p=2$ has been treated in countless other works due to its special connection with the Yamabe problem -- see, for instance,~\cite{dn} and the references therein. 
	
	The purpose of this manuscript is to prove certain extensions of these results to the case of a general norm~$H$, and it will also offer an overview of the existing theory and simplify some arguments. In particular, we circumvent in a novel way some regularity issues that may arise for a general norm~$H$. Furthermore, some of the results presented here are new even in the Euclidean setting and for the case~$p=2$ in the anisotropic framework.
	
	
	\subsection{The Euclidean setting: a quick review.}
	
	In the Euclidean context, the classification of non-negative solutions to~\eqref{eq:critica-anisot} began with the seminal papers of Gidas, Ni \& Nirenberg~\cite{gnn}, Caffarelli, Gidas \& Spruck~\cite{cgs}, and Chen \& Li~\cite{cl} in the case~$p=2$ and when~$H$ is the Euclidean norm, i.e., for the classical Laplacian, using the method of moving planes. Remarkably, this classification holds for~$C^2$-solutions to~\eqref{eq:critica-anisot} without any energy assumption. Moreover, a beautiful result due to Schoen -- quoted in~\cite[Corollary~1.6]{li-zhang} -- implies that any positive solution~$u \in C^2(\R^n)$ to~\eqref{eq:critica-anisot} with~$p=2$ enjoys the energy-mass estimate
	\begin{equation}
	\label{eq:sch}
		\int_{B_R(0)} \left(\abs*{\nabla u}^2 + u^{2^\ast}\right) dx \leq C_n,
	\end{equation}
	for some dimensional constant~$C_n>0$. More precisely, estimate~\eqref{eq:sch} was first established by Schoen using the classification result in~\cite{cgs}. However, Li \& Zhang~\cite{li-zhang} later provided an alternative proof that does not rely on such a classification. As a consequence, in the case~$p=2$ and for positive functions, being a~$C^2$-smooth solution to~\eqref{eq:critica-anisot} for the Laplacian is equivalent to being a~$\mathcal{D}^{1,2}$-solution to the same equation.
	
	Over the past two decades, the classification of solutions to~\eqref{eq:critica-anisot} has been extended to the full range~$1<p<n$, thus completely settling the problem for solutions in the energy space~$\mathcal{D}^{1,p}(\R^n)$. In particular, by combining the method of moving planes with suitable decay estimates, Damascelli, Merch\'an, Montoro \& Sciunzi~\cite{dm}, V\'etois~\cite{vet}, and Sciunzi~\cite{sciu} proved that  any positive energy solution to~\eqref{eq:critica-anisot} must be of the form~\eqref{eq:pH-bubb}.
	
	More recently, Catino, Monticelli \& Roncoroni~\cite{catino} addressed the classification of positive local weak solutions to~\eqref{eq:critica-anisot}, i.e., positive functions~$u \in W^{1,p}_{\loc}(\R^n) \cap L^\infty_{\loc}(\R^n)$ that satisfy
	\begin{equation*}
		\int_{\R^n} \left\langle \abs*{\nabla u}^{p-2} \,\nabla u, \nabla \psi \right\rangle dx = 	\int_{\R^n} u^{\past-1} \psi \, dx \quad\text{for every } \psi \in C^\infty_c(\R^n).
	\end{equation*}
	Their work was improved by Ou~\cite{ou}, and later refined by V\'etois~\cite{vet-plap} and Sun \& Wang~\cite{sun-wang}. These contributions yield a classification in the regime~$p_n<p<n$ for a threshold~$p_n \sim n/3$ depending on the dimension~$n \geq 3$.
	
	Additionally, Catino, Monticelli \& Roncoroni~\cite{catino} established the classification of local weak solutions in the full range~$1<p<n$, provided that the solution satisfies a suitable growth at infinity, or is bounded and either~$n \leq 6$ or~$n \geq 7$ with~$p> n/3$. For~$n \geq 7$, however, this latter range is essentially covered by the results in~\cite{ou,sun-wang,vet-plap}, and thus the contribution of~\cite{catino} provides new information only for~$p \in (n/3,p_n]$.
	
	Most recently, in order to overcome this limitation, Ciraolo \& Gatti~\cite{cg-classif} proved that any positive local weak solution to~\eqref{eq:critica-anisot}, whose infimum over a ball behaves properly, actually enjoys global higher integrability, namely~$u \in L^{\past-1}(\R^n)$. This yields the classification of bounded solutions in the full range~$1<p<n$, as well as of moderately growing solutions and solutions with additional regularity under this constraint. \newline
	
	Once the classification has been established, a natural problem is the investigation of stability. Informally, this leads to the following question:
	\begin{center}
		is~$u$ \textit{close} to a multiple of a~$p$-bubble, provided that~$u$ \textit{almost} solves~\eqref{eq:critica-anisot}?
	\end{center}
	
	Of course, this informal question needs to be given a rigorous meaning. From a PDE perspective, tackling such quantitative aspects naturally leads to the study of the perturbed equation
	\begin{equation}
		\label{eq:maineq-bubb-eu}
		\Delta_p u + \kappa(x) u^{\past-1} =0 \quad \text{in } \R^n,
	\end{equation}
	for~$\kappa$ close to a constant.
	
	Nevertheless, it is now well known that the answer to the previous question is negative in the following two senses.
	
	First, even restricting ourselves to the classical Laplacian, Ding \& Ni~\cite{dn} showed that the perturbed equation~\eqref{eq:maineq-bubb-eu} admits positive, radial solutions when~$\kappa = 1 +\eta$ and~$\eta \in C^\infty_{c}(\R^n)$ is a small, non-negative cut-off function. Notably, these solutions do not belong to the energy space~$\mathcal{D}^{1,2}(\R^n)$ and therefore cannot be close to an Euclidean~$2$-bubble in any reasonable sense. This phenomenon also motivated the study of stability for non-energy solutions to a class of semilinear equations by Ciraolo, Cozzi \& Gatti~\cite{ccg}, where quasi-symmetry for more general semilinear equations was also addressed. Part of these results was later extended by Gatti~\cite{mg-plap} to energy solutions of~\eqref{eq:critica-anisot} for~$p>2$.
	
	In the pursuit of a positive answer, it thus seems necessary to restrict our attention to solutions in the energy space~$\mathcal{D}^{1,p}(\R^n)$. However, even under this assumption, the situation remains subtle. As shown in a seminal work by Struwe~\cite{struwe} and its subsequent generalizations~\cite{alves,benci-cer,merc-will} the answer is negative. Indeed, one cannot, in general, assert that the quasi-solution~$u$ must be close to a single Euclidean~$p$-bubble. Instead,~$u$ may only be close to a \textit{sum} of bubbles. This phenomenon is called~\textit{bubbling}. The number of bubbles can be controlled by an a priori energy bound on~$u$. Moreover,~$u$ must be close to a family of \textit{weakly interacting} bubbles in the decomposition, meaning that the parameters of the bubbles, i.e.,  their centers and scaling factors, satisfy certain asymptotic behavior -- see Section~\ref{sec:interaction} for details. To the best of our knowledge, these interaction estimates are well understand only for~$p=2$ thanks to a result of Bahri \& Coron~\cite{bahri-coron}. \newline
	
	The stability problem for~\eqref{eq:maineq-bubb-eu} in a quantitative form has been completely settled in the Euclidean case for~$p=2$, even in the broader setting of critical points of the Sobolev inequality, thanks to the works of Ciraolo, Figalli \& Maggi~\cite{cfm}, Figalli \& Glaudo~\cite{fg}, and Deng, Sun \& Wei~\cite{dsw}.
	
	In particular, Ciraolo, Figalli \& Maggi~\cite{cfm} studied solutions~$u \in \mathcal{D}^{1,2}(\R^n)$ to~\eqref{eq:maineq-bubb-eu} under the a priori energy assumption
	\begin{equation}
		\label{eq:ener-est-euclid}
		\frac{1}{2} S^n \leq \int_{\R^n} \,\abs{\nabla u}^2 \, dx \leq \frac{3}{2} S^n,
	\end{equation}
	which ensures that the energy of~$u$ is close to that of a single Euclidean~$2$-bubble. Here~$S$ denotes the best Sobolev constant for the Euclidean norm, i.e.,~\eqref{eq:def-SH} for~$H=\abs*{\cdot}$. They also defined the deficit
	\begin{equation}
		\label{eq:def-cfm-anis}
		\defi(u,\kappa) \coloneqq \norma*{\left(\kappa-\kappa_0\right)u^{2^\ast-1}}_{L^{(2^\ast)'}\!(\R^n)},
	\end{equation}
	with the reference constant
	\begin{equation*}
		\kappa_0=\kappa_0(u) \coloneqq \frac{\int_{\R^n} \kappa(x) u^{2^\ast} dx}{\int_{\R^n} u^{\past} dx} = \frac{\int_{\R^n} \,\abs{\nabla u}^2 \, dx}{\int_{\R^n} u^{2^\ast} dx}.
	\end{equation*}
	The deficit~\eqref{eq:def-cfm-anis} measures the deviation of~$\kappa$ from the reference constant~$\kappa_0$ and clearly vanishes when~$u$ is a multiple of an Euclidean~$2$-bubble.
	
	Under the above assumptions and when~$\kappa_0(u)=1$, the main result in~\cite{cfm} states that there exists a dimensional constant~$c_n>0$ such that
	\begin{equation}
		\label{eq:clos-euclid}
		\norma*{u-U_2[z,\lambda]}_{\mathcal{D}^{1,2}(\R^n)} \leq c_n \defi(u,\kappa),
	\end{equation}
	for an Euclidean~$2$-bubble~$U_2[z,\lambda]$.
	
	This quantitative stability result was later extended in~\cite{dsw,fg} allowing for bubbling. It is worth noticing that, in these works, the interaction in Struwe's decomposition plays a crucial role in obtaining quantitative estimates.
	
	All the results in~\cite{cfm,dsw,fg} strongly rely on the Hilbert structure of~$\mathcal{D}^{1,2}(\R^n)$ and on the explicit description of the family of solutions to~\eqref{eq:critica-anisot} given by~\eqref{eq:pH-bubb}. Consequently, the method developed in these works do not seem suitable for the general case.
	
	For a general~$1<p<n$, to the best of our knowledge, quantitative stability has been recently investigated by Ciraolo \& Gatti~\cite{cg-plap} for solutions to~\eqref{eq:maineq-bubb-eu} and by~Liu \& Zhang~\cite{liu-zhang} in the critical point framework, independently and concurrently. Assuming the counterpart of~\eqref{eq:ener-est-euclid}, both the main results in~\cite{cg-plap,liu-zhang} provide an estimate of the same form as~\eqref{eq:clos-euclid} in the energy space~$\mathcal{D}^{1,p}(\R^n)$.
	
	
	\subsection{The anisotropic setting.}
	
	We now present our main assumption on the norm~$H$ and introduce some notations that will be used throughout the manuscript.
	
	For our purposes, we require that the anisotropy~$H$ is of class~$C^{3,\beta}(\R^n \setminus \{0\})$ for some~$\beta\in (0,1)$, and its anisotropic unit ball
	\begin{equation}
	\label{eq:unif-convex}
		B^H_1 \coloneqq \left\{\xi \in \R^n \mid H(\xi)<1\right\} \quad\text{is uniformly convex}.
	\end{equation}
	This means that all the principal curvatures of its boundary are bounded away from zero -- see, for instance,~\cite{cozzi-grad,cozzi-monot}. Additionally, we assume the ellipticity condition
	\begin{equation}
	\label{eq:ellip-norm}
		\lambda_H  \,\abs*{\eta}^2 \leq \frac{1}{2} \left\langle \nabla^2_\xi H^2\!\left(\xi\right) \eta , \eta \right\rangle \leq \Lambda_H \,\abs*{\eta}^2 \quad\text{for all } \xi \in \R^n \setminus \{0\} \text{ and } \eta \in \R^n,
	\end{equation}
    for some constants~$0<\lambda_H\leq \Lambda_H$.
	Note that the left-most inequality in~\eqref{eq:ellip-norm} in equivalent to the geometric assumption~\eqref{eq:unif-convex} -- see, for instance,~\cite{cozzi-monot}.

    Moreover, by Lemma~3.2 in~\cite{carlos+big4}, we have
    \begin{equation}
    \label{eq:nabla2Hp:twosides}
    c_{p,H}\,H^{p-2}(\xi)\Id\leq \nabla^2_\xi H^p(\xi)\leq C_{p,H}\,H^{p-2}(\xi)\Id\quad\text{for all } \xi\in \R^n\setminus\{0\}
    \end{equation}
    for some constants~$0<c_{p,H}\leq C_{p,H}$, depending on~$p$,~$\lambda_H$,~and $\Lambda_H$, and where~$\Id$ is the identity matrix.
    
	We denote by~$B^H_r(\xi_0)$ the anisotropic ball of radius~$r>0$ centered at~$\xi_0$, that is
	\begin{equation*}
		B^H_r(\xi_0) \coloneqq \left\{\xi \in \R^n \mid H(\xi-\xi_0)<r\right\} 
	\end{equation*}
	and abbreviate~$B^H_r \coloneqq B^H_r(0)$. We also set
	\begin{equation}
	\label{eq:defV-anisot}
		V(\xi) \coloneqq \frac{H^p(\xi)}{p}
	\end{equation}
	and define the so-called stress field
	\begin{equation}
	\label{eq:def-a}
		a(\xi) \coloneqq \nabla V(\xi).
	\end{equation}
	Thanks to the regularity of~$H$, it follows that~$V \in C^1(\R^n) \cap C^{3}(\R^n \setminus \{0\})$, and we have that
	\begin{equation}
	\label{eq:def-a-2}
		a(\xi) \coloneqq
		\begin{cases}
			\begin{aligned}
				& H^{p-1}(\xi) \nabla H(\xi)	&& \text{if } \xi \neq 0, \\
				& 0								&& \text{if } \xi = 0.
			\end{aligned}
		\end{cases}
	\end{equation}
	
	Since~$H$ is a norm on~$\R^n$, it is equivalent to the Euclidean one. Thus, there exists a couple of constants~$c_H,C_H>0$ such that
	\begin{equation}
	\label{eq:equiv-norma}
		c_{H} \,\abs*{\xi} \leq H(\xi) \leq C_{H} \,\abs*{\xi} \quad\text{for every } \xi \in \R^n.
	\end{equation}
	We also observe that, by the homogeneity of~$H^2$,~\eqref{eq:ellip-norm} entails that
	\begin{equation}
	\label{eq:H2}
		\lambda_H  \,\abs*{\xi}^2 \leq H^2(\xi) \leq \Lambda_H \,\abs*{\xi}^2 \quad\text{for all } \xi \in \R^n,
	\end{equation}
	so that~\eqref{eq:equiv-norma} holds, for instance, with~$c_H = \sqrt{\lambda_H}$ and~$C_H = \sqrt{\Lambda_H}$.
	
	Moreover, from the definition of the dual norm in~\eqref{eq:def-H0} and~\eqref{eq:H2}, we also deduce that
	\begin{equation}
	\label{eq:equiv-H0}
		\frac{1}{\Lambda_H} \,\abs*{\xi}^2 \leq H_0^2(\xi) \leq \frac{1}{\lambda_H} \,\abs*{\xi}^2 \quad\text{for all } \xi \in \R^n.
	\end{equation}
	Thanks to Lemma 3.1 in~\cite{cianchi-sal}, we know that
	\begin{equation}\label{H0:nablH}
		H_0(\nabla H(\xi)) = 1 \quad\text{for all } \xi \in \R^n \setminus \{0\},
	\end{equation}
	which, combined with~\eqref{eq:equiv-H0}, gives that
	\begin{equation*}
		\sqrt{\lambda_H} \leq \abs*{\nabla H (\xi)} \leq \sqrt{\Lambda_H} \quad\text{for all } \xi \in \R^n \setminus \{0\}.
	\end{equation*}
	
	
	\subsection{Main results.}
	
	In the anisotropic setting, Ciraolo, Figalli \& Roncoroni~\cite{cfr} established the classification of positive solutions to~\eqref{eq:critica-anisot} in convex cones and for uniformly convex norms -- see condition~\eqref{eq:ellip-norm}. In particular, in the case of the entire space, they proved the following result.
	
	\begin{theorem}
		\label{th:calssif-anisot}
		Let~$n \in \N$ and~$H:\R^n \to [0,+\infty)$ be a norm of class~$C^{3,\beta}(\R^n \setminus \{0\})$ for some~$\beta\in (0,1)$, satisfying the ellipticity condition~\eqref{eq:ellip-norm}.
		
		Moreover, let~$1<p<n$ and~$u \in \mathcal{D}^{1,p}(\R^n)$ be a positive weak solution to~\eqref{eq:critica-anisot}. Then, there exist~$\lambda>0$ and~$z \in \R^n$ such that~$u = U_p[z,\lambda]$, i.e.,~$u$ is of the form~\eqref{eq:pH-bubb}.
	\end{theorem}
	
	Although this theorem has already been established through integral estimates, we provide in Section~\ref{sec:anis-classif-energ} a concise proof to highlight the strength of our approach.
	
	Theorem~\ref{th:calssif-anisot} shows that, under mild assumptions on the norm~$H$, the~$(p,H)$-bubbles are the only positive energy solutions to the critical equation~\eqref{eq:critica-anisot}. As in the Euclidean setting, once the classification is proven, it is natural to seek a quantitative counterpart. \newline
	
	Thus, as a first goal of the present work, we establish Struwe's decomposition in the anisotropic setting. In addition, we derive an interaction estimate for the family of~$(p,H)$-bubbles in the decomposition. To the best of our knowledge, this interaction estimate is new even in the Euclidean case for general~$1<p<n$, and it is expected to play a role in the quantitative analysis of bubbling.
	
	To be precise, we define on the space~$\mathcal{D}^{1,p}(\R^n)$ the functional
	\begin{equation*}
		\mathcal{J}(u) \coloneqq \int_{\R^n} \frac{1}{p} \, H^p(\nabla u) - \frac{1}{\past} \,\abs*{u}^{\past} dx,
	\end{equation*}
	and denote its Fr\'echet derivative by~$\mathcal{J}'$ -- see Section~\ref{sec:anis-struwe} for a definition.
	
	With this notation at hand, the our first result reads as follows.
	
	\begin{theorem}
	\label{th:struwe-intro}
		For~$n \in \N$ and~$1<p<n$, let~$\left\{u_m\right\}_m \subseteq \mathcal{D}^{1,p}(\R^n)$ be a sequence of non-negative functions satisfying
		\begin{equation*}
			\mathcal{J}'(u_m) \to 0 \quad\text{in } \mathcal{D}^{-1,p'}(\R^n).
		\end{equation*}
		Here,~$\mathcal{D}^{-1,p'}(\R^n)$ denotes the dual space of~$\mathcal{D}^{1,p}(\R^n)$. Moreover, suppose that
		\begin{equation*}
			\left(k-\frac{1}{2}\right) S_p^n \leq \int_{\R^n} H^p(\nabla u_m) \, dx \leq \left(k+\frac{1}{2}\right) S_p^n
		\end{equation*}
		for every~$m \in \N$ and for some~$k \in \N$ with~$k \geq 1$.
		
		Then, possibly passing to a subsequence, there exist positive solutions~$v_1,\dots,v_k \in \mathcal{D}^{1,p}(\R^n)$ to~\eqref{eq:critica-anisot}, i.e.,~$(p,H)$-bubbles of the form~\eqref{eq:pH-bubb}, and~$k$ sequences~$\left\{y^i_m\right\}_m \subseteq \R^n$ and~$\left\{\lambda^i_m\right\}_m \subseteq (0,+\infty)$ such that
		\begin{gather*}
			\notag
			\norma*{u_m - \sum_{i=1}^{k} \left(\lambda^i_m\right)^{\!\frac{p-n}{p}} v_i\left(\frac{\cdot - y^i_m}{\lambda^i_m}\right)}_{\mathcal{D}^{1,p}(\R^n)} \to 0, \\
			\notag
			\norma*{u_m}_{\mathcal{D}^{1,p}(\R^n)}^p \to \sum_{i=1}^{k} \,\norma*{v_i}_{\mathcal{D}^{1,p}(\R^n)}^p \quad\text{and}\quad \norma*{H(\nabla u_m)}_{L^{p}(\R^n)}^p \to \sum_{i=1}^{k} \,\norma*{H(\nabla v_i)}_{L^{p}(\R^n)}^p, \\
			\notag
			\mathcal{J}(u_m) \to c_\star=
			\sum_{i=1}^{k} \mathcal{J}(v_i),
		\end{gather*}
		for some $c_\star\in \R$.
		Moreover, if~$y^i_m \to y^i$ as~$m \to +\infty$, then either~$\lambda^i_m \to 0$ or~$\lambda^i_m \to +\infty$. In addition, it follows that
		\begin{equation*}
			\label{eq:interaction-intro}
			\max\left\{\frac{\lambda^i_m}{\lambda^j_m}, \frac{\lambda^j_m}{\lambda^i_m}, \frac{\abs*{y^i_m-y^j_m}^2}{\lambda_m^i \lambda_m^j} \right\} \to +\infty \quad\text{as } m \to+\infty \text{ for all } i \neq j.
		\end{equation*}
	\end{theorem}
	
	We emphasize that Theorem~\ref{th:struwe-intro} recovers the known interaction result in the Euclidean case for~$p=2$, as we shall discuss in Section~\ref{sec:interaction}. \newline
	
	Concerning quantitative aspect, as mentioned above, in the Euclidean setting only the case of absence of bubbling has been recently treated in~\cite{cg-plap,liu-zhang}.
	
	Although the estimate in~\cite{liu-zhang} is optimal, we adopt here the method developed in~\cite{cg-plap}, which seems more flexible and can be extended to the anisotropic case, as we illustrate below.
	
	In the present manuscript, building on the ideas in~\cite{cg-plap} and establishing the appropriate integral estimates, we provide a quantitative closeness result to a~$(p,H)$-bubble for solutions to~\eqref{eq:maineq-bubb-anisot} in the absence of bubbling. To the best of our knowledge, this quantitative analysis in new even in the case~$p=2$ for general norms~$H$.
	
	In the anisotropic setting, to obtain quantitative stability results, one has to consider solutions to
	\begin{equation}
		\label{eq:maineq-bubb-anisot}
		\Delta_p^H u + \kappa(x) u^{\past-1} =0 \quad \text{in } \R^n,
	\end{equation}
	with the a priori energy assumption 	
	\begin{equation}
	\label{eq:ipotesi-energ-anisot}
		\frac{1}{2} S^n_p \leq \int_{\R^n} H^{p}\!\left(\nabla u\right) dx \leq \frac{3}{2} S^n_p.
	\end{equation}
	where~$S_p$ is given by~\eqref{eq:def-SH}.
	In particular, when~$\kappa \equiv \kappa_0$ is a constant, any solution~$u \in \mathcal{D}^{1,p}(\R^n)$ to~\eqref{eq:maineq-bubb-anisot} must be a multiple of a~$(p,H)$-bubble and, by testing the equation with~$u$, one finds
	\begin{equation}
		\label{eq:defk0-anisot}
		\kappa_0=\kappa_0(u) \coloneqq \frac{\int_{\R^n} \kappa(x) u^{\past} dx}{\int_{\R^n} u^{\past} dx} = \frac{\int_{\R^n} H^{p}\!\left(\nabla u\right) dx}{\int_{\R^n} u^{\past} dx}.
	\end{equation}
	It is thus natural to measure the proximity of~$\kappa$ to~$\kappa_0$ using the deficit
	\begin{equation}
		\label{eq:def-cfm-anisot}
		\defi(u,\kappa) \coloneqq \norma*{\left(\kappa-\kappa_0\right)u^{\past-1}}_{L^{(\past)'}\!(\R^n)}.
	\end{equation}
    Here, as usual,~$q'=\frac{q}{q-1}$ denotes the H\"older conjugate of~$q>1$.

	With this notation in place, we can state our main stability result.
	
	\begin{theorem}
	\label{th:main-th-anisot-stab}
		Let~$n \in \N$ and~$H:\R^n \to [0,+\infty)$ be a  norm of class~$C^{3,\beta}(\R^n \setminus \{0\})$ for some~$\beta\in (0,1)$ satisfying the ellipticity condition~\eqref{eq:ellip-norm}.
		
		Let~$1<p<n$, and let~$\kappa \in L^\infty(\R^n) \cap C^{1,1}_{\loc}(\R^n)$ be a positive function. Suppose~$u \in \mathcal{D}^{1,p}(\R^n)$ is a positive weak solution to~\eqref{eq:maineq-bubb-anisot} satisfying the energy bound~\eqref{eq:ipotesi-energ-anisot}. Finally, assume that
		\begin{equation*}
			\kappa_0(u)=1,
		\end{equation*}
		where~$\kappa_0(u)$ is defined in~\eqref{eq:defk0-anisot}.
		
		Then, there exist constants~$C \geq 1$ and~$\vartheta \in (0,1)$, and a~$(p,H)$-bubble~$U_p[z,\lambda]$, of the form~\eqref{eq:pH-bubb}, such that
		\begin{equation*}
			\label{eq:close-toapHbub}
			\norma*{u-U_p[z,\lambda]}_{\mathcal{D}^{1,p}(\R^n)} \leq C \defi(u,\kappa)^{\vartheta}.
		\end{equation*}
		The constants~$C$ and~$\vartheta$ depend only on~$n$,~$p$,~$H$, and~$\norma*{\kappa}_{L^\infty(\R^n)}$.
	\end{theorem}
	
	As observed in~\cite{cfm,cg-plap}, the general case follows from Theorem~\ref{th:main-th-anisot-stab} by scaling. Thus, we have the following corollary, whose proof is omitted.
	
	\begin{corollary}
	\label{cor:main-bubbles-anisot}
		Let~$n \in \N$ and~$H:\R^n \to [0,+\infty)$ be a  norm of class~$C^{3,\beta}(\R^n \setminus \{0\})$ for some~$\beta\in (0,1)$ satisfying the ellipticity condition~\eqref{eq:ellip-norm}.
		
		Let~$1<p<n$, and let~$\kappa \in L^\infty(\R^n) \cap C^{1,1}_{\loc}(\R^n)$ be a positive function. Suppose~$u \in \mathcal{D}^{1,p}(\R^n)$ is a positive weak solution to~\eqref{eq:maineq-bubb-anisot} satisfying the energy bound
		\begin{equation*}
			\frac{1}{2} \,\kappa_0(u)^{\frac{p}{p-\past}} S^n_p \leq \int_{\R^n} H^p(\nabla u) \, dx \leq \frac{3}{2} \,\kappa_0(u)^{\frac{p}{p-\past}} S^n_p,
		\end{equation*}
		where~$\kappa_0(u)$ is defined in~\eqref{eq:defk0-anisot}.
		
		Then, there exist constants~$C \geq 1$ and~$\vartheta \in (0,1)$, and a function~$\mathcal{U} \in \mathcal{D}^{1,p}(\R^n)$ of the form
		\begin{equation*}
			\mathcal{U} = \kappa_0(u)^{\frac{1}{p-\past}} \,U_p[z,\lambda],
		\end{equation*}
		for some~$(p,H)$-bubble~$U_p[z,\lambda]$ of the form~\eqref{eq:pH-bubb}, such that
		\begin{equation*}
			\norma*{u-\mathcal{U}}_{\mathcal{D}^{1,p}(\R^n)} \leq C \defi(u,\kappa)^{\vartheta}.
		\end{equation*}
		The constant~$\vartheta$ depends only on~$n$,~$p$,~$H$, and~$\norma*{\kappa}_{L^\infty(\R^n)}$, while the constant~$C$ depends on~$\kappa_0(u)$ as well.
	\end{corollary}
	
	As mentioned above, the proof of Theorem~\ref{th:main-th-anisot-stab} adapts to the anisotropic setting the scheme introduced in~\cite{cg-plap}. The main idea is to consider the auxiliary function~$v \coloneqq u^{-\frac{p}{n-p}}$ and the so-called~$P$-function, defined in the Euclidean setting as
	\begin{equation*}
		P \coloneqq n \,\frac{p-1}{p} v^{-1} \,\abs*{\nabla v}^p + \left(\frac{p}{n-p}\right)^{\! p-1} v^{-1}. 
	\end{equation*}
	In addition, we consider the matrix~$W \coloneqq \nabla a(\nabla v)$ and its traceless version~$\mathring{W}$. The key point is to derive and integral inequality involving~$P$ and~$\mathring{W}$ -- namely, the content of Proposition~3.3 in~\cite{cg-plap}. Building on this, one can construct a function that approximates~$v$ in an integral sense and, then, revert back to~$u$.
	
	This strategy carries over to our general framework with appropriate modifications due to the presence of the anisotropy~$H$. Once the correct integral estimates for the~$P$-function have been established -- see Proposition~\ref{prop:fund-ineq-anisot} below -- the remainder of proof follows a similar construction.
	
	The main difficulties stem from the fact that, in the present setting, we work with a general norm~$H$. Nevertheless, we provide a direct proof of the main integral inequality -- namely, the content of  Proposition~\ref{prop:fund-ineq-anisot} below --together with the~$P$-function estimate from which it follows -- that is Proposition~\ref{prop:fund-ident-anisot-tbp} below. This approach is more straightforward and significantly simplifies the regularization procedure used in~\cite{cg-plap}.
	
	Finally, the integral identity obtained in Proposition~\ref{prop:fund-ident-anisot-tbp} may serve to extend the classification of~\eqref{eq:critica-anisot} to local weak solutions, thus generalizing the works~\cite{catino,cg-classif,ou,vet-plap} to a broader context. We refrain from pursuing this direction here, preferring instead to focus on energy solutions.
	
	We conclude by emphasizing that the regularity assumptions on the norm,~$H \in C^{3,\beta}(\R^n \setminus \{0\})$, is purely technical. In fact, the main results of the present manuscript could be established under the weaker assumption that~$H \in C^{2}(\R^n \setminus \{0\})$, satisfying the ellipticity condition~\eqref{eq:ellip-norm}. This would, however, require an additional regularization procedure in Proposition~\ref{prop:fund-ident-anisot-tbp}, which we refrain from including here in order not to overburden the presentation.
	
	
	\subsection{Open problems.}
	
	At present, there are at least a couple of questions concerning~\eqref{eq:critica-anisot}, that remain open even in the Euclidean framework. We will briefly discuss them below.
	
	The first one concerns the classification of local weak solutions. As mentioned above, this result is known only in a certain dimensional range for~$p$ or under additional assumptions, such as boundedness.
	\begin{problemA}
	\label{pb:1}
		For~$1<p \leq p_n$, does there exist a non-negative solution~$u \in W^{1,p}_{\loc}(\R^n) \cap L^\infty_{\loc}(\R^n)$ to~\eqref{eq:critica-anisot} which is not a~$(p,H)$-bubble of the form~\eqref{eq:pH-bubb}, even in the Euclidean case?
	\end{problemA}
	For the precise value of~$p_n$, see, for instance, Theorem~I in~\cite{cg-classif}. By the previous discussion, any such solution -- if it exists -- must fail to satisfy certain regularity properties. Moreover, if a counterexample does exist -- possibly for a different threshold~$p_n$ -- and therefore Problem~\ref{pb:1} has a positive answer, a philosophical question arise:
	\begin{center}
		which is the meaning of this dimensional threshold?
	\end{center}
	
	The second problem concerns stability issues. For~$p=2$, the stability of~\eqref{eq:critica-anisot} in the Euclidean case has been completely settled, allowing for bubbling.
	\begin{problemA}
	\label{pb:2}
		Is it possible to obtain a quantitative stability result, in the direction of Theorem~\ref{th:main-th-anisot-stab} -- or those in~\cite{cg-plap,liu-zhang} -- that allows for bubbling, even in the Euclidean case?
	\end{problemA}
	Indeed, Theorem~\ref{th:struwe-intro} and the results in~\cite{alves,merc-will} show that perturbations of~\eqref{eq:critica-anisot} may give rise to bubbling phenomena. However, since the case~$p=2$ is fully understood, we strongly believe that Problem~\ref{pb:2} has a positive answer. At present, pursuing such a result seems to be obstructed only by technical difficulties stemming from the non-Hilbertian structure of the problem. Of course, the anisotropic case, even for~$p=2$, may present additional challenges.
	
	
	\subsection{Structure of the paper.}
	
	In Section~\ref{sec:symm-anisot}, we introduce a family of operators that act nicely on solution to~\eqref{eq:maineq-bubb-anisot} and will be useful throughout the paper. Sections~\ref{sec:anis-struwe} and~\ref{sec:interaction} are devoted to proving Struwe's decomposition and establishing the interaction estimate for the family of bubbles appearing in the decomposition and, in particular, to proving Theorem~\ref{th:struwe-intro}. In Section~\ref{sec:ineq-Pfunct}, we introduce the so-called~$P$-function and prove some fundamental integral inequalities involving it. Exploiting these estimates, we finally prove Theorem~\ref{th:calssif-anisot} in Section~\ref{sec:anis-classif-energ} and Theorem~\ref{th:main-th-anisot-stab} in Section~\ref{sec:anis-stab}.

	
	\section{Symmetries of the problem}
	\label{sec:symm-anisot}

	Given~$\lambda >0$ and~$z \in \R^n$, let~$T_{z,\lambda}: C^\infty_c(\R^n) \to C^\infty_c(\R^n)$ be the operator defined by
	\begin{equation*}
		T_{z,\lambda}(\phi)(x) \coloneqq \lambda^{\frac{n-p}{p}} \phi\left(\lambda(x-z)\right).
	\end{equation*}
	Clearly, this operator can be extended to functions which are non smooth with compact support in the same fashion.
	
	The following lemma provides some fundamental properties of the family of transformations~$T_{z,\lambda}$ which will be useful in the following. This is the counterpart of Lemma~2.1 in~\cite{cg-plap} -- see also Section~2.1 in~\cite{fg}.
	
	\begin{lemma}
		\label{lem:sym-anisot}
		The operator~$T_{z,\lambda}$ enjoys the subsequent properties.
		\begin{enumerate}[leftmargin=*,label=$(\arabic*)$]
			\item \label{it:cons-normapas-anisot} For any~$\phi \in C^\infty_c(\R^n)$ and any norm~$H$, it holds that
			\begin{equation*}
				\int_{\R^n} H^{\past}\!\!\left(T_{z,\lambda}(\phi)\right) dx = \int_{\R^n} H^{\past}\!\!\left(\phi\right) dx.
			\end{equation*}
			\item \label{it:cons-normap-anisot} For any~$\phi \in C^\infty_c(\R^n)$ and any norm~$H$, it holds that
			\begin{equation*}
				\int_{\R^n} H^{p}\!\left(\nabla T_{z,\lambda}(\phi)\right) dx = \int_{\R^n} H^{p}\!\left(\nabla \phi\right) dx.
			\end{equation*}
			\item \label{it:inver-anisot} For any~$\phi \in C^\infty_c(\R^n)$, it holds that
			\begin{equation*}
				T_{z,\lambda}\!\left(T_{-z\lambda,1/\lambda}(\phi)\right) = T_{-z\lambda,1/\lambda}\!\left(T_{z,\lambda}(\phi)\right) = \phi.
			\end{equation*}
			\item \label{it:stillbub-anisot} For any~$(p,H)$-bubble~$U$,~$T_{z,\lambda} \!\left(U\right)$ is still a~$(p,H)$-bubble.
			\item \label{it:transf-unitbub-anisot} The~$(p,H)$-bubbles obey
			\begin{equation*}
				U_p[z,\lambda] = T_{z,1/\lambda} \!\left(U_p[0,1]\right) \quad\text{and}\quad U_p[0,1] = T_{-z/\lambda,\lambda} \!\left(U_p[z,\lambda]\right).
			\end{equation*}
		\end{enumerate}
	\end{lemma}
	Obviously all the properties of Lemma~\ref{lem:sym-anisot} hold also if the functions are non smooth with compact support, provided that the involved integrals are finite. \newline
	
	We also observe that both the quantities~$\kappa_0(u)$ and~$\defi(u,\kappa)$, given in~\eqref{eq:defk0-anisot} and~\eqref{eq:def-cfm-anisot}, respectively, are invariant under the action of the operators~$T_{z,\lambda}$.
	
	More precisely, if~$u \in \mathcal{D}^{1,p}(\R^n)$ is a weak solution to~\eqref{eq:maineq-bubb-anisot} and~$v \coloneqq T_{z,\lambda}(u)$, then~$v \in \mathcal{D}^{1,p}(\R^n)$ is a weak solution to
	\begin{equation*}
		\Delta_p^H v + \widehat{\kappa}(x) v^{\past-1} =0 \quad \text{in } \R^n
	\end{equation*}
	where~$\widehat{\kappa}(x) = \kappa(\lambda(x-z))$, moreover
	\begin{equation}
		\label{eq:inv-norma-anisot}
		\widehat{\kappa} \in L^\infty(\R^n) \quad\text{with}\quad \norma*{\widehat{\kappa}}_{L^\infty(\R^n)}=\norma*{\kappa}_{L^\infty(\R^n)},
	\end{equation}
	and also
	\begin{equation}
		\label{eq:inv-k-anisot}
		\kappa_0(u) = \kappa_0(v) \quad\text{and}\quad \defi(u,\kappa) = \defi(v,\widehat{\kappa}).
	\end{equation}
	
	
	\section{Anisotropic Struwe's decomposition}
	\label{sec:anis-struwe}
	
	This section is devoted to establishing a variant of the Struwe's decomposition for the anisotropic~$p$-Laplacian, which will be instrumental in proving the stability result.
	
	Let~$H\in C^2(\R^n\setminus \{0\})$ be a norm satisfying~\eqref{eq:ellip-norm}. On the energy space~$\mathcal{D}^{1,p}(\R^n)$, we define the functional
	\begin{equation*}
		\mathcal{J}(u) \coloneqq \int_{\R^n} V(\nabla u) - \frac{1}{\past} \,\abs*{u}^{\past} dx = \int_{\R^n} \frac{1}{p} \, H^p(\nabla u) - \frac{1}{\past} \,\abs*{u}^{\past} dx,
	\end{equation*}
	and recall that its Fr\'echet derivative satisfies
	\begin{equation*}
		\left\langle \mathcal{J}'(u), \phi \right\rangle = \int_{\R^n}  \left\langle \stressu, \nabla \phi \right\rangle - \abs*{u}^{\past-2} u \phi \, dx,
	\end{equation*}
	where the vector field~$a$ is given by~\eqref{eq:def-a}.

    We also denote by~$\|u\|_{\mathcal{D}^{1,p}(\R^n)} \coloneqq \|\nabla u\|_{L^p(\R^n)}$ the norm on the space~$\mathcal{D}^{1,p}(\R^n)$, and by~$\mathcal{D}^{-1,p'}(\R^n)$ the dual space of~$\mathcal{D}^{1,p}(\R^n)$.

    The main result of this section is the following.

	\begin{theorem}
	\label{th:struwe-def}
		For~$n \in \N$ and~$1<p<n$, let~$\left\{u_m\right\}_m \subseteq \mathcal{D}^{1,p}(\R^n)$ be a sequence such that
		\begin{equation}
		\label{eq:hyp-sruwe-anisot}
			\mathcal{J}(u_m) \to c_\star \quad\text{and}\quad \mathcal{J}'(u_m) \to 0 \quad\text{in } \mathcal{D}^{-1,p'}(\R^n),
		\end{equation}
		for some $c_\star\in \R$. Then, perhaps passing to a subsequence, there exists a possibly trivial solution~$v_0 \in \mathcal{D}^{1,p}(\R^n)$ to
		\begin{equation}
		\label{eq:crit-anis-struwe}
			\Delta_p^{H} v + \abs{v}^{\past-2} v = 0 \quad\text{in } \R^n,
		\end{equation}
		a natural number~$k \in \N$, non-trivial solutions~$v_1,\dots,v_k \in \mathcal{D}^{1,p}(\R^n)$ to~\eqref{eq:crit-anis-struwe}, and~$k$ sequences~$\left\{y^i_m\right\}_{m\in \N} \subseteq \R^n$ and~$\left\{\lambda^i_m\right\} \subseteq (0,+\infty)$ for $i=1,\dots,k$, such that
		\begin{gather*}
			\norma*{u_m - v_0 - \sum_{i=1}^{k} \left(\lambda^i_m\right)^{\!\frac{p-n}{p}} v_i\left(\frac{\cdot - y^i_m}{\lambda^i_m}\right)}_{\mathcal{D}^{1,p}(\R^n)} \to 0, \\
			\norma*{u_m}_{\mathcal{D}^{1,p}(\R^n)}^p \to \sum_{i=0}^{k} \,\norma*{v_i}_{\mathcal{D}^{1,p}(\R^n)}^p \quad\text{and}\quad \norma*{H(\nabla u_m)}_{L^{p}(\R^n)}^p \to \sum_{i=0}^{k} \,\norma*{H(\nabla v_i)}_{L^{p}(\R^n)}^p, \\
			\sum_{i=0}^{k} \mathcal{J}(v_i) = c_\star.
		\end{gather*}
		Moreover, if~$y^i_m \to y^i$ as~$m \to +\infty$, then either~$\lambda^i_m \to 0$ or~$\lambda^i_m \to +\infty$.
		
		Finally, if~$k \geq 1$ and~$u_m \geq 0$ a.e.\ in~$\R^n$ for all~$m \in \N$, then~$v_i>0$ in~$\R^n$ for all~$i=1,\dots,k$.
	\end{theorem}

	We will occasionally say that a sequence~$\left\{u_m\right\}_m \subseteq \mathcal{D}^{1,p}(\R^n)$ satisfying~\eqref{eq:hyp-sruwe-anisot} is a Palais-Smale sequence for the energy functional~$\mathcal{J}$.
	
	In the isotropic setting, Theorem~\ref{th:struwe-def} was first established by Benci \& Cerami~\cite{benci-cer} for~$p=2$, building on the original theorem due to Struwe~\cite{struwe}. It was later extended by Alves~\cite{alves} for~$p>2$ and by Mercuri \& Willem~\cite{merc-will} for~$1<p<2$.
	
	Since we could not find an exact reference for the anisotropic case in the literature, we provide a proof of Theorem~\ref{th:struwe-def} below, following the approach in~\cite{merc-will}.
	
	
	\subsection{Preliminary results.} 
	
	In this subsection, we gather some technical results needed for the proof of the main Theorem~\ref{th:struwe-def}. Although a few are simple variations of existing results, we include their proofs here for completeness.
	
	We start by recalling a well-known result due to Brezis \& Lieb~\cite{brezis-lieb}.
	
	\begin{lemma}[Brezis-Lieb Lemma]
	\label{lem:brez-lieb}
		Let~$j: \R^n \to \R$ be a continuous function such that~$j(0)=0$. Moreover, suppose that for every~$\varepsilon \in (0,1)$ there exist two continuous functions~$\varphi_\varepsilon, \psi_\varepsilon: \R^n \to [0,+\infty]$ such that
		\begin{equation}
		\label{eq:toverify-brezis-lieb}
			\abs*{j(x+y)-j(x)} \leq \varepsilon \varphi_\varepsilon(x) + \psi_\varepsilon(y) \quad\text{for all } x,y \in \R^n.
		\end{equation}
		Let~$\left(\Omega,\Sigma,\mu\right)$ be a measure space and define
		\begin{equation*}
			\mathrm{meas}\!\left(\Omega,\R^n\right) \coloneqq \left\{f: \Omega \to \R^n \mid f \text{ is } \Sigma \text{-measurable} \right\}\!.
		\end{equation*}
		Let~$f \in \mathrm{meas}\!\left(\Omega,\R^n\right)$ and~$\left\{g_m\right\}_m \subseteq \mathrm{meas}\!\left(\Omega,\R^n\right)$ satisfy
		\begin{equation}
		\label{eq:assum-br-lieb}
			g_m \to 0 \quad\mu\text{--a.e.\ in } \Omega \quad\text{and}\quad j(f) \in L^1\!\left(\Omega,\Sigma,\mu\right),
		\end{equation}
		and suppose that
		\begin{equation}
		\label{eq:toverify-brezis-lieb-2}
			\int_{\Omega} \varphi_\varepsilon(g_m) \, d\mu \leq C \quad\text{and}\quad \int_{\Omega} \psi_\varepsilon(f) \, d\mu < +\infty \quad\text{for all } \varepsilon \in (0,1),
		\end{equation}
		where~$C>0$ is a constant independent of~$m$ and~$\varepsilon$. Then, it follows that
		\begin{equation*}
			\lim_{m \to +\infty} \int_{\Omega} \,\abs*{j(g_m+f) - j(g_m) - j(f)} \, d\mu = 0.
		\end{equation*}
		Finally, if~$j$ is a norm,~\eqref{eq:toverify-brezis-lieb} holds with~$\varphi_\varepsilon = 0$ and~$\psi_\varepsilon = j$, moreover~\eqref{eq:toverify-brezis-lieb-2} is trivially satisfied under the sole assumption~\eqref{eq:assum-br-lieb}.
	\end{lemma}
	
	The next lemma is a technical bound involving the stress field.
	
	\begin{lemma}
	\label{lem:simil-diff-quot}
		Let~$1<p\leq 2$. Then, we have
		\begin{equation*}
			\mathsf{c} \coloneqq \sup_{\substack{\eta \in \R^n \setminus \{0\} \\ \xi \in \R^n}} \,\abs*{\frac{a \!\left(\xi+\eta\right) - a \!\left(\xi\right)}{H^{p-1}(\eta)}} < +\infty,
		\end{equation*} 
		where~$a(\xi)=\{a_k(\xi)\}_{k=1,\dots,n}$ is given by~\eqref{eq:def-a}.	
	\end{lemma}
	\begin{proof}
		Define
		\begin{equation*}
			F(\xi,\eta) \coloneqq \abs*{\frac{a \!\left(\xi+\eta\right) - a \!\left(\xi\right)}{H^{p-1}(\eta)}},
		\end{equation*}
		so that, by homogeneity, it is easy to verify that
		\begin{equation*}
			F \!\left(\xi,t \eta\right) = F \!\left(\frac{\xi}{t},\eta\right) \quad \text{for every } t \in \R \setminus \left\{0\right\}\!.
		\end{equation*}
		As a consequence, we obtain
		\begin{equation*}
			\mathsf{c} = \sup_{\substack{H(\eta) = 1 \\ \xi \in \R^n}} F \!\left(\xi, \eta\right).
		\end{equation*}
		Since by continuity
		\begin{equation*}
			\mathsf{c}_1 \coloneqq \sup_{\substack{H(\eta) = 1 \\ H(\xi) \leq 2}} F \!\left(\xi, \eta\right) < +\infty,
		\end{equation*}
		it suffices to prove that
		\begin{equation}
		\label{eq:c2-bound-toprove}
			\mathsf{c}_2 \coloneqq \sup_{\substack{H(\eta) = 1 \\ H(\xi) > 2}} F \!\left(\xi, \eta\right) < +\infty.
		\end{equation}
		To this end, assume that~$H(\eta)=1$ and~$H(\xi)>2$. Then, we have
		\begin{equation}
		\label{eq:Hxith}
			H \!\left(\xi+t \eta\right) \geq H(\xi) - H(\eta) > 1 \quad\text{for every } t \in [0,1].
		\end{equation}
		Moreover, the fundamental theorem of calculus entails that
		\begin{equation*}
			\abs*{a_k \!\left(\xi+\eta\right) - a_k \!\left(\xi\right)} = \abs*{\int_{0}^{1} \frac{d}{dt} a_k \!\left(\xi+t\eta\right) dt} = \abs*{\sum_{j=1}^{n} \int_{0}^{1} \frac{\partial a_k}{\partial \xi_j} \!\left(\xi+t\eta\right) dt \, \eta_j}.
		\end{equation*}
		We now observe that, by Lemma~3.2 in~\cite{carlos+big4} and~\eqref{eq:Hxith}, it holds 
		\begin{equation*}
			\abs*{\frac{\partial a_k}{\partial \xi_j} \!\left(\xi+t\eta\right)} \leq C H^{p-2} \!\left(\xi+t\eta\right) \leq C,
		\end{equation*}
		for a constant~$C>0$ depending only on~$p$ and~$\Lambda_H$. Hence, using~\eqref{eq:equiv-norma}, we conclude that
		\begin{equation*}
			\abs*{a_k \!\left(\xi+\eta\right) - a_k \!\left(\xi\right)} \leq C \,\abs*{\eta} \leq C \, c_H^{-1}.
		\end{equation*}
		This proves~\eqref{eq:c2-bound-toprove} and concludes the proof.
	\end{proof}
	
	The following result is a variant of Lemma~\ref{lem:brez-lieb}.
	
	\begin{lemma}
	\label{lem:brez-lieb-variant}
		Let~$p>1$ and~$\left\{\eta_m\right\}_m \subseteq L^p(\R^n,\R^n)$ be such that~$\eta_m \to 0$ a.e.\ in~$\R^n$ and~$\sup_{m \in \N} \,\norma*{\eta_m}_{L^p(\R^n)} < +\infty$. Then, for any fixed~$w \in L^p(\R^n,\R^n)$, we have
		\begin{equation*}
			\lim_{m \to +\infty} \int_{\R^n} \,\abs*{a \!\left(\eta_m+w\right)- a \!\left(\eta_m\right) - a \!\left(w\right)}^{\frac{p}{p-1}} \, dx = 0,
		\end{equation*}
		where~$a$ is given by~\eqref{eq:def-a}.
	\end{lemma}
	\begin{proof}
		For~$1<p\leq 2$, the result follows from Lemma~\ref{lem:simil-diff-quot} and an application of the dominated convergence theorem.
			
		In case~$p > 2$, we follow the proof of Lemma~3 in~\cite{alves}. In the following calculations~$C,C'>0$ will denote some constants independent of~$m$, which may vary from line to line. For~$w=0$ the conclusion is trivial, so we may assume that~$w \neq 0$.
			
		Arguing as in the proof of Lemma~\ref{lem:simil-diff-quot}, we deduce that
		\begin{equation*}
			\abs*{a_k \!\left(\eta_m+w\right) - a_k \!\left(\eta_m\right)} \leq C \int_{0}^{1} H^{p-2}\!\left(\eta_m+tw\right) dx \,\abs*{w}.
		\end{equation*}
		Exploiting~\eqref{eq:equiv-norma}, this implies
		\begin{equation*}
			\abs*{a \!\left(\eta_m+w\right) - a \!\left(\eta_m\right)} \leq C \left[\abs*{\eta_m}+\abs*{w}\right]^{p-2} \abs*{w} \leq C \,\abs*{\eta_m}^{p-2} \,\abs*{w} + C' \,\abs*{w}^{p-1}.
		\end{equation*}
		Thus, by Young's inequality, for each fixed~$\varepsilon>0$, we obtain
		\begin{equation}
		\label{eq:for-Gmep}
			\abs*{a \!\left(\eta_m+w\right) - a \!\left(\eta_m\right)} \leq C_\varepsilon \,\abs*{w}^{p-1} + \varepsilon \,\abs*{\eta_m}^{p-1},
		\end{equation}
		where the constant~$C_\varepsilon>0$ depends on~$\varepsilon$ and $p> 2$, but not on~$m$. For~$\varepsilon>0$ we define the function
		\begin{equation}
		\label{eq:defGmep}
			G_{m,\varepsilon} \coloneqq \max\left\{\abs*{a \!\left(\eta_m+w\right)- a \!\left(\eta_m\right) - a \!\left(w\right)} - \varepsilon \,\abs*{\eta_m}^{p-1}, 0\right\}\!,
		\end{equation}
		and observe that, by our assumptions,~$G_{m,\varepsilon} \to 0$ a.e.\ in~$\R^n$ as~$m \to +\infty$. Moreover, by~\eqref{eq:equiv-norma},~\eqref{eq:for-Gmep}, and the definition of~$a$, we have
		\begin{equation*}
			0 \leq G_{m,\varepsilon} \leq C_\varepsilon \,\abs*{w}^{p-1} \in L^{\frac{p}{p-1}} (\R^n),
		\end{equation*}
		where~$C_\varepsilon>0$ depends also on~$\varepsilon$. Therefore, for each fixed~$\varepsilon>0$, the dominated convergence theorem implies that
		\begin{equation}
		\label{eq:limGmep}
		 	\lim_{m \to +\infty} \int_{\R^n} \,\abs*{G_{m,\varepsilon}}^\frac{p}{p-1} \, dx = 0.
		\end{equation}
	 	On the other hand, by~\eqref{eq:defGmep}, we also have
	 	\begin{equation*}
	 		\abs*{a \!\left(\eta_m+w\right)- a \!\left(\eta_m\right) - a \!\left(w\right)} \leq G_{m,\varepsilon} + \varepsilon \,\abs*{\eta_m}^{p-1},
	 	\end{equation*}
 		hence, for any~$\varepsilon \in (0,1)$,
 		\begin{equation*}
 			\abs*{a \!\left(\eta_m+w\right)- a \!\left(\eta_m\right) - a \!\left(w\right)}^\frac{p}{p-1} \leq 2^\frac{p}{p-1} G_{m,\varepsilon}^\frac{p}{p-1} + 2^\frac{p}{p-1} \varepsilon \,\abs*{\eta_m}^{p}.
 		\end{equation*}
 		Consequently, from~\eqref{eq:limGmep}, we deduce that for every~$\varepsilon \in (0,1)$
 		\begin{align*}
 			\limsup_{m \to +\infty} \int_{\R^n} \,\abs*{a \!\left(\eta_m+w\right)- a \!\left(\eta_m\right) - a \!\left(w\right)}^{\frac{p}{p-1}} \, dx &\leq 2^\frac{p}{p-1} \limsup_{m \to +\infty} \,\norma*{\eta_m}_{L^{p}(\R^n)}^p \,\varepsilon \\
 			&\leq 2^\frac{p}{p-1} \sup_{m \in \N} \,\norma*{\eta_m}_{L^{p}(\R^n)}^p \,\varepsilon \leq C \varepsilon,
 		\end{align*}
 		where~$C>0$ is independent of~$\varepsilon$. From the last estimate the conclusion follows at once.
	\end{proof}

	Let us now define the truncation map as
	\begin{equation}
	\label{eq:def-mapT}
		\mathcal{T}(t) \coloneqq
		\begin{cases}
			\begin{aligned}
				& \;\, t				&& \text{if } \abs*{t} \leq 1, \\
				& \frac{t}{\abs*{t}}	&& \text{if } \abs*{t} > 1.
			\end{aligned}
		\end{cases}
	\end{equation}
	With this notation at hand, we can state the following result.

	\begin{proposition}
	\label{prop:exhaust}
		Let~$\left\{\Omega_k\right\}_k \subseteq \R^n$ be an increasing sequence of open, bounded sets such that
		\begin{equation*}
			\bigcup_{k=1}^{+\infty} \Omega_k = \R^n.
		\end{equation*}
		For~$p>1$, let~$\left\{u_m\right\}_m \subseteq \mathcal{D}^{1,p}(\R^n)$ be a sequence such that~$u_m \to u$ weakly in~$\mathcal{D}^{1,p}(\R^n)$ and
		\begin{equation}
		\label{eq:assump-lim-foreveryk}
			\lim_{m \to +\infty} \int_{\Omega_k} \left\langle a \!\left(\nabla u_m\right) - a \!\left(\nabla u\right), \nabla \mathcal{T} \!\left(u_m-u\right) \right\rangle dx = 0 \quad\text{for every } k \in \N.
		\end{equation}
		Then, passing if necessary to a subsequence, we have the following limits
		\begin{gather}
		\label{eq:exhaust-convgrad}
			\nabla u_m \to \nabla u \quad\text{a.e.\ in } \R^n, \\
		\label{eq:exhaust-convnormgrad-euclid}
			\lim_{m \to +\infty} \left[\norma*{\nabla u_m}_{L^p(\R^n)}^p - \norma*{\nabla (u_m - u)}_{L^p(\R^n)}^p\right] = \norma*{\nabla u}_{L^p(\R^n)}^p, \\
		\label{eq:exhaust-convnormgrad}
		\lim_{m \to +\infty} \left[\norma*{H(\nabla u_m)}_{L^p(\R^n)}^p - \norma*{H(\nabla (u_m - u))}_{L^p(\R^n)}^p\right] = \norma*{H(\nabla u)}_{L^p(\R^n)}^p, \\
		\label{eq:exhaust-stress}
			a \!\left(\nabla u_m\right) - a \!\left(\nabla u_m - u\right) \to a \!\left(\nabla u\right) \quad\text{in } L^{\frac{p}{p-1}}(\R^n).
		\end{gather}
	\end{proposition}
	\begin{proof}
		We immediately observe that, once~\eqref{eq:exhaust-convgrad} has been established,~\eqref{eq:exhaust-convnormgrad-euclid}--\eqref{eq:exhaust-stress} follow from Lemma~\ref{lem:brez-lieb} and Lemma~\ref{lem:brez-lieb-variant}.
		
		Fix~$k \in \N$. We claim that, possibly up to a subsequence,
		\begin{equation}
		\label{eq:claim-conv-grad}
			\nabla u_m \to \nabla u \quad\text{a.e.\ in } \Omega_k.
		\end{equation}
		We begin by noticing that, since~$\Omega_k$ is bounded, by the weak convergence in~$\mathcal{D}^{1,p}(\R^n)$ and the Rellich–Kondrachov theorem, it follows that
		\begin{equation}
		\label{eq:convae-um-tou}
			u_m \to u \quad\text{a.e.\ in } \Omega_k,
		\end{equation}
		up to a subsequence, as~$\mathcal{D}^{1,p}(\R^n) $ compactly embeds into~$L^q(\Omega_k)$ for every~$q \in [1,\past)$.
			
		We now argue as in the proof of Theorem~1.1 in~\cite{deV-willem} to establish~\eqref{eq:claim-conv-grad}. Define
		\begin{equation*}
			E_{m} \coloneqq \left\{x \in \Omega_k \,\lvert\, \abs*{u_m(x)-u(x)} \leq 1\right\} \quad\text{and}\quad e_m \coloneqq \left\langle a \!\left(\nabla u_m\right) - a \!\left(\nabla u\right), \nabla (u_m-u) \right\rangle\!.
		\end{equation*}
		From~\eqref{eq:assump-lim-foreveryk}, we infer that
		\begin{equation*}
			\int_{E_m} e_m \, dx = \int_{\Omega_k} \left\langle a \!\left(\nabla u_m\right) - a \!\left(\nabla u\right), \nabla \mathcal{T} \!\left(u_m-u\right) \right\rangle dx \to 0 \quad\text{as } m \to +\infty.
		\end{equation*}
        
		Since~$e_m\geq 0$ by the convexity of~$H^p$, extracting a subsequence we may assume that~$e_m \chi_{E_m} \to 0$ a.e.\ in~$\Omega_k$. Since~\eqref{eq:convae-um-tou} implies that~$\chi_{E_m} \to 1$ a.e.\ in~$\Omega_k$, we necessarily have
		\begin{equation}
		\label{eq:conv-em}
			e_m \to 0 \quad \text{a.e.\ in } \Omega_k.
		\end{equation}
		Thanks to Lemma~3.1 in~\cite{carlos+big4}, we may use Lemma~2.1 in~\cite{dam-comp} and deduce that 
		\begin{equation}
		\label{eq:coerc-1}
			\left\langle a \!\left(\nabla u_m\right) - a \!\left(\nabla u\right), \nabla (u_m-u) \right\rangle \geq \gamma \,\abs*{\nabla (u_m-u)}^p \quad\text{for } p \geq 2,
		\end{equation}
		for a constant~$\gamma>0$, and
		\begin{multline}
		\label{eq:coerc-2}
			\left\langle a \!\left(\nabla u_m\right) - a \!\left(\nabla u\right), \nabla (u_m-u) \right\rangle \\
			\geq \gamma \left(\abs*{\nabla u_m}+\abs*{\nabla u}\right)^{p-2} \abs*{\nabla (u_m-u)}^2 \quad\text{for } 1<p<2,
		\end{multline}
		provided that~$\abs*{\nabla u_m}+\abs*{\nabla u}>0$. From~\eqref{eq:conv-em}--\eqref{eq:coerc-2} and recalling the definition of~$e_m$, we infer the validity of~\eqref{eq:claim-conv-grad}.
		
		Finally,~\eqref{eq:exhaust-convgrad} follows from~\eqref{eq:claim-conv-grad} via a standard Cantor diagonal argument.
		\end{proof}
		
		The following result provides some convergence properties for a Palais–Smale sequence for the energy functional~$\mathcal{J}$.
	
		\begin{lemma}
		\label{lem:conv-D1p-ener-Frediff}
			Let~$1<p<n$ and~$\left\{u_m\right\}_m \subseteq \mathcal{D}^{1,p}(\R^n)$ be a sequence such that~$u_m \to u$ weakly in~$\mathcal{D}^{1,p}(\R^n)$,~$u_m \to u$ a.e.\ in~$\R^n$, and
			\begin{equation*}
				\mathcal{J}(u_m) \to c_\star \quad\text{and}\quad \mathcal{J}'(u_m) \to 0 \quad\text{in } \mathcal{D}^{-1,p'}(\R^n).
			\end{equation*}
			Then, possibly passing to a subsequence, we have~$\nabla u_m \to \nabla u$ a.e.\ in~$\R^n$ and~$\mathcal{J}'(u)=0$. Moreover, defining~$v_m \coloneqq u_m - u$, it follows that
			\begin{gather}
			\label{eq:lemconv-D1p}
				\lim_{m \to +\infty} \left[\norma*{u_m}_{\mathcal{D}^{1,p}(\R^n)}^p - \norma*{v_m}_{\mathcal{D}^{1,p}(\R^n)}^p\right] = \norma*{u}_{\mathcal{D}^{1,p}(\R^n)}^p, \\
			\label{eq:lemconv-LpH}
				\lim_{m \to +\infty} \left[\norma*{H(\nabla u_m)}_{L^{p}(\R^n)}^p - \norma*{H(\nabla v_m)}_{L^{p}(\R^n)}^p\right] = \norma*{H(\nabla u)}_{L^{p}(\R^n)}^p, \\
			\label{eq:lemconv-energ}
				\mathcal{J}(v_m) \to c_\star - \mathcal{J}(u), \\
			\label{eq:lemconv-Fredif}
				\mathcal{J}'(v_m) \to 0 \quad\text{in } \mathcal{D}^{-1,p'}(\R^n).
			\end{gather}
		\end{lemma}
		\begin{proof}
			We begin by noticing that the weak convergence in~$\mathcal{D}^{1,p}(\R^n)$ and the Sobolev inequality imply that
			\begin{equation}
			\label{eq:unif-bound-p-past-1}
				\norma*{\nabla u_m}_{L^p(\R^n)} \leq C \quad\text{and}\quad \norma*{u_m}_{L^{\past}\!(\R^n)} \leq C,
			\end{equation}
			for some constant~$C>0$ independent of $m\in \N$. Moreover, by the definition of~$\mathcal{T}$ in~\eqref{eq:def-mapT}, we have
			\begin{equation}
			\label{eq:weak-conv-T}
				\mathcal{T}(u_m - u) \to 0 \quad\text{weakly in } \mathcal{D}^{1,p}(\R^n) \;\;\text{and a.e.\ in }\R^n.
			\end{equation}
			For any~$k \in \N$, define~$\Omega_k \coloneqq B_k(0)$. Now, fix a~$k \in \N$ and choose~$\eta \in C^\infty_c(\R^n)$ such that~$0 \leq \eta \leq 1$,~$\eta = 1$ in~$\Omega_k$,~$\eta = 0$ in~$\R^n \setminus \Omega_{k+1}$, and~$\abs*{\nabla \eta} \leq 1$ in~$\Omega_{k+1} \setminus \Omega_k$. Using the Poincaré inequality, it is easy to see that the sequence~$\{\mathcal{T} \!\left(u_m - u\right) \eta\}_m \subseteq \mathcal{D}^{1,p}(\R^n)$ is uniformly bounded. Hence, passing to a subsequence if necessary,~\eqref{eq:weak-conv-T} implies that 
			\begin{equation}
			\label{eq:weak-conv-Teta}
				\mathcal{T} \!\left(u_m - u\right) \eta \to 0 \quad\text{weakly in } \mathcal{D}^{1,p}(\R^n).
			\end{equation}
			Define~$\mathscr{A}_m \coloneqq a \!\left(\nabla u_m\right) - a \!\left(\nabla u\right)$. Then, we can write
			\begin{align*}
				\int_{\Omega_k} &\left\langle \mathscr{A}_m, \nabla \mathcal{T} \!\left(u_m-u\right) \right\rangle \eta \, dx \\
				&= \int_{\R^n} \left\langle \mathscr{A}_m, \nabla \!\left(\mathcal{T} \!\left(u_m-u\right) \eta\right) \right\rangle dx - \int_{\R^n} \left\langle \mathscr{A}_m, \nabla \eta \right\rangle \mathcal{T} \!\left(u_m-u\right) dx,
			\end{align*}
			and observe that
			\begin{align}
			\label{eq:pass-limit-anablaT-2}
				\int_{\R^n} &\left\langle \mathscr{A}_m, \nabla \!\left(\mathcal{T} \!\left(u_m-u\right) \eta\right) \right\rangle dx = \left\langle \mathcal{J}'(u_m), \mathcal{T} \!\left(u_m-u\right) \eta \right\rangle \\
			\notag
				&\quad + \int_{\R^n} \,\abs{u_m}^{\past-2} u_m \, \mathcal{T} \!\left(u_m-u\right) \eta \, dx  - \int_{\R^n} \left\langle a \!\left(\nabla u\right), \nabla \!\left(\mathcal{T} \!\left(u_m-u\right) \eta\right) \right\rangle dx.
			\end{align}
			Using our assumptions and the weak convergence in~\eqref{eq:weak-conv-Teta}, we obtain
			\begin{equation}
			\label{eq:pass-limit-anablaT-3}
				\left\langle \mathcal{J}'(u_m), \mathcal{T} \!\left(u_m-u\right) \eta \right\rangle \to 0 \quad\text{and}\quad \int_{\R^n} \left\langle a \!\left(\nabla u\right), \nabla \!\left(\mathcal{T} \!\left(u_m-u\right) \eta\right) \right\rangle dx \to 0.
			\end{equation}
			Now, we prove that, up to a subsequence,
			\begin{equation}
			\label{eq:pass-limit-anablaT-4}
				\int_{\R^n} \,\abs{u_m}^{\past-2} u_m \, \mathcal{T} \!\left(u_m-u\right) \eta \, dx \to 0.
			\end{equation}
			Indeed, since~$\mathcal{T}$ is bounded and~$\eta \in C^\infty_c(\R^n)$, using H\"older inequality,~\eqref{eq:unif-bound-p-past-1}, the fact that~$\mathcal{T}(u_m-u) \to 0$ a.e.\ in~$\R^n$, and the dominated convergence theorem, we deduce that
			\begin{equation*}
				\abs*{\int_{\R^n} \,\abs{u_m}^{\past-2} u_m \, \mathcal{T} \!\left(u_m-u\right) \eta \, dx} \leq C \left(\int_{\mathrm{supp} \,\eta} \,\abs*{\mathcal{T}(u_m-u)}^{\past} dx \right)^{\!\frac{1}{\past}} \to 0.
			\end{equation*}
			Thus, combining~\eqref{eq:pass-limit-anablaT-2},~\eqref{eq:pass-limit-anablaT-3}, and~\eqref{eq:pass-limit-anablaT-4}, we get
			\begin{equation}
			\label{eq:passata-limit-anablaT-1}
				\int_{\R^n} \left\langle \mathscr{A}_m, \nabla \!\left(\mathcal{T} \!\left(u_m-u\right) \eta\right) \right\rangle dx \to 0.
			\end{equation}
			Finally, we show that
			\begin{equation}
			\label{eq:passata-limit-anablaT-2}
				\int_{\R^n} \left\langle \mathscr{A}_m, \nabla \eta \right\rangle \mathcal{T} \!\left(u_m-u\right) dx \to 0.
			\end{equation}
			To this end, using~\eqref{eq:def-a-2} and~\eqref{eq:equiv-norma}, note that
			\begin{equation*}
				\abs*{\mathscr{A}_m} \leq C \left[\abs*{\nabla u_m}^{p-1} + \abs*{\nabla u}^{p-1}\right] \!,
			\end{equation*}
			for some~$C>0$ independent of~$m$. Therefore, arguing as in~\eqref{eq:pass-limit-anablaT-4}, we have
			\begin{equation*}
				\abs*{\int_{\R^n} \left\langle \mathscr{A}_m, \nabla \eta \right\rangle \mathcal{T} \!\left(u_m-u\right) dx} \leq C \left(\int_{\mathrm{supp} \,\eta} \,\abs*{\mathcal{T}(u_m-u)}^{p} \, dx \right)^{\!\frac{1}{p}} \to 0.
			\end{equation*}
			Since~$\eta \equiv 1$ on~$\Omega_k$, combining~\eqref{eq:passata-limit-anablaT-1} with~\eqref{eq:passata-limit-anablaT-2}, we conclude that, for an arbitrary~$k \in \N$,
			\begin{equation*}
				\int_{\Omega_k} \left\langle a \!\left(\nabla u_m\right) - a \!\left(\nabla u\right), \nabla \mathcal{T} \!\left(u_m-u\right) \right\rangle dx \to 0.
			\end{equation*}
		
			Applying Proposition~\ref{prop:exhaust}, we deduce that, up to a subsequence,~$\nabla u_m \to \nabla u$ a.e.\ in~$\R^n$ and~\eqref{eq:lemconv-D1p}--\eqref{eq:lemconv-LpH} directly follow from~\eqref{eq:exhaust-convnormgrad-euclid}--\eqref{eq:exhaust-convnormgrad}. In particular,~\eqref{eq:lemconv-LpH} reads as
			\begin{equation}
			\label{eq:nablaH-o(1)}
				\norma*{H(\nabla v_m)}_{L^p(\R^n)}^p = \norma*{H(\nabla u_m)}_{L^p(\R^n)}^p - \norma*{H(\nabla u)}_{L^p(\R^n)}^p + o(1).
			\end{equation}
			As~$u_m \to u$ a.e.\ in~$\R^n$, Lemma~\ref{lem:brez-lieb} entails that
			\begin{equation}
			\label{eq:past-o(1)}
				\norma*{v_m}_{L^{\past}\!(\R^n)}^{\past} = \norma*{u_m}_{L^{\past}\!(\R^n)}^{\past} - \norma*{u}_{L^{\past}\!(\R^n)}^{\past} + o(1).
			\end{equation}
			Therefore,~\eqref{eq:nablaH-o(1)} and~\eqref{eq:past-o(1)} yield
			\begin{equation*}
				\mathcal{J}(v_m) = \frac{1}{p} \,\norma*{H(\nabla v_m)}_{L^p(\R^n)}^p - \frac{1}{\past} \,\norma*{v_m}_{L^{\past}\!(\R^n)}^{\past} = \mathcal{J}(u_m) - \mathcal{J}(u) + o(1) = c_\star - \mathcal{J}(u) + o(1),
			\end{equation*}
			which establishes the validity of~\eqref{eq:lemconv-energ}.
			
			We are left to prove~\eqref{eq:lemconv-Fredif}. By assumption, we have~$u_m \to u$ weakly in~$\mathcal{D}^{1,p}(\R^n)$ and~$\mathcal{J}'(u_m) \to 0$ in~$\mathcal{D}^{-1,p'}(\R^n)$, hence it follows that~$\mathcal{J}'(u)=0$.
			
			Finally, for any fixed~$\phi \in C^\infty_c(\R^n)$ with~$\norma*{\phi}_{\mathcal{D}^{1,p}(\R^n)} = 1$, we write
			\begin{equation}
			\label{eq:rewrite-frediff}
				\left\langle \mathcal{J}'(v_m) , \phi \right\rangle = \left\langle \mathcal{J}'(u_m) - \mathcal{J}'(u) , \phi \right\rangle + \mathscr{R}_m^1 + \mathscr{R}_m^2,
			\end{equation}
			where
			\begin{align*}
				\mathscr{R}_m^1 &\coloneqq \int_{\R^n} \left\langle a \!\left(\nabla v_m\right) - a \!\left(\nabla u_m\right) + a \!\left(\nabla u\right) , \nabla \phi \right\rangle dx, \\
				\mathscr{R}_m^2 &\coloneqq \int_{\R^n} \left[\abs*{u_m}^{\past-2} u_m - \abs*{v_m}^{\past-2} v_m - \abs*{u}^{\past-2} u \right] \phi \, dx.
			\end{align*}
			We now note that, using~\eqref{eq:exhaust-stress}, we have
			\begin{equation*}
				\abs{\mathscr{R}_m^1} \leq \norma*{a \!\left(\nabla v_m\right) - a \!\left(\nabla u_m\right) + a \!\left(\nabla u\right)}_{L^{p'}\!(\R^n)} \to 0.
			\end{equation*}
			We now claim that
			\begin{equation}
			\label{eq:convR2-toprove}
				\mathscr{R}_m^2 \to 0 \quad\text{as } m \to +\infty.
			\end{equation}
			Thus, from~\eqref{eq:rewrite-frediff} and our assumptions, we infer
			\begin{equation*}
				\mathcal{J}'(v_m) = \mathcal{J}'(u_m) - \mathcal{J}'(u) + o(1) = o(1),
			\end{equation*}
			which implies the validity of~\eqref{eq:lemconv-Fredif}.
			
			Thus, to complete the proof it remains to establish~\eqref{eq:convR2-toprove}. To this end, we follow the argument from Step~1.3 on pages~305-306 of~\cite{saint}. Define
			\begin{align*}
				\Phi_m &\coloneqq \abs*{u_m}^{\past-2} u_m - \abs*{v_m}^{\past-2} v_m - \abs*{u}^{\past-2} u \\
				&= \abs*{v_m+u}^{\past-2} (v_m+u) - \abs*{v_m}^{\past-2} v_m - \abs*{u}^{\past-2} u.
			\end{align*}
			Recalling that for~$q>1$ and~$\theta>0$ sufficiently small, depending only on~$q$, there exists a constant~$C>0$ such that
			\begin{multline*}
				\abs*{\abs*{a+b}^{q-2} (a+b) - \abs*{a}^{q-2} a - \abs*{b}^{q-2} b} \\ \leq C \left[\abs*{a}^{q-1-\theta} \,\abs*{b}^\theta + \abs*{b}^{q-1-\theta} \,\abs*{a}^\theta \right] \quad\text{for all } a,b \in \R,
			\end{multline*}
			so that
			\begin{equation*}
				\abs*{\Phi_m} \leq C \left[\abs*{v_m}^{\past-1-\theta} \,\abs*{u}^\theta + \abs*{u}^{\past-1-\theta} \,\abs*{v_m}^\theta \right]\!.
			\end{equation*}
			Applying H\"older inequality, we obtain
			\begin{equation}
			\label{eq:norm-vanish-toprove}
				\abs{\mathscr{R}_m^2} \leq C \left[\norma*{\abs*{v_m}^{\past-1-\theta} \,\abs*{u}^\theta}_{L^{(\past)'}\!(\R^n)} + \norma*{\abs*{u}^{\past-1-\theta} \,\abs*{v_m}^\theta}_{L^{(\past)'}\!(\R^n)} \right] \!,
			\end{equation}
			where~$C>0$ is independent of~$m$. Next, we show that, up to a subsequence,  the norms on the right-hand side of~\eqref{eq:norm-vanish-toprove} vanish as~$m \to +\infty$. We start by defining the H\"older conjugate pair
			\begin{equation*}
				\sigma_1 \coloneqq \frac{\past-1}{\theta} \quad\text{and}\quad \sigma_2 \coloneqq \frac{\past-1}{\past-1-\theta}.
			\end{equation*}
			Thus, we have
			\begin{gather*}
				\mathscr{N}_1 \coloneqq \norma*{\abs*{v_m}^{\past-1-\theta} \,\abs*{u}^\theta}_{L^{(\past)'}\!(\R^n)}^{(\past)'} = \int_{\R^n} \,\abs{v_m}^{\frac{\past}{\sigma_2}} \,\abs*{u}^{\frac{\past}{\sigma_1}} \, dx, \\
				\mathscr{N}_2 \coloneqq \norma*{\abs*{u}^{\past-1-\theta} \,\abs*{v_m}^\theta}_{L^{(\past)'}\!(\R^n)}^{(\past)'} = \int_{\R^n} \,\abs{v_m}^{\frac{\past}{\sigma_1}} \,\abs*{u}^{\frac{\past}{\sigma_2}} \, dx.
			\end{gather*}
			Using~\eqref{eq:unif-bound-p-past-1} and since~$v_m \to 0$ a.e.\ in~$\R^n$, up to subsequences, we obtain
			\begin{equation*}
				\abs{v_m}^{\past/\sigma_j} \to 0 \quad\text{weakly in } L^{\sigma_j}(\R^n) \quad\text{for } j=1,2.
			\end{equation*}
			Since~$u \in L^{\past}\!(\R^n)$, the weak convergence above implies that
			\begin{equation*}
				\mathscr{N}_j \to 0 \quad\text{for } j=1,2.
			\end{equation*}
			This, in turn, establishes~\eqref{eq:norm-vanish-toprove}, thereby concluding the proof.
		\end{proof}
		
		The following lemma is necessary in the iterative step of the main proof.
		
		\begin{lemma}
		\label{lem:struwe-iter}
			Consider the sequences~$\left\{y_m\right\}_m \subseteq \R^n$ and~$\left\{\lambda_m\right\}_m \subseteq (0,+\infty)$. For a sequence~$\{u_m\}_m \subseteq \mathcal{D}^{1,p}(\R^n)$, define
			\begin{equation*}
				v_m(x) \coloneqq \lambda_m^{\frac{n-p}{p}} u_m \!\left(\lambda_m x + y_m\right) \quad\text{for all } x \in \R^n.
			\end{equation*}
			Additionally, suppose that
			\begin{gather}
			\label{eq:lemit-convm}
				v_m \to v \quad\text{weakly in } \mathcal{D}^{1,p}(\R^n) \quad\text{and}\quad v_m \to v \quad\text{a.e.\ in } \R^n, \\
			\label{eq:lemit-energ-conv}
				\mathcal{J}(u_m) \to c_\star, \\
			\label{eq:lemit-Frediff-conv}
				\mathcal{J}'(u_m) \to 0 \quad\text{in } \mathcal{D}^{-1,p'}(\R^n).
			\end{gather}
			Then, possibly passing to a subsequence, we have~$\nabla v_m \to \nabla v$ a.e.\ in~$\R^n$ and~$\mathcal{J}'(v)=0$. Moreover, defining
			\begin{equation*}
				w_m(x) \coloneqq u_m(x) - \lambda_m^{\frac{p-n}{p}} v \!\left( \frac{x - y_m}{\lambda_m}\right) \quad\text{for all } x \in \R^n,
			\end{equation*}
			it follows that
			\begin{gather}
			\label{eq:lemiter-D1p-w}
				\lim_{m \to +\infty} \left[\norma*{u_m}_{\mathcal{D}^{1,p}(\R^n)}^p - \norma*{w_m}_{\mathcal{D}^{1,p}(\R^n)}^p\right] = \norma*{v}_{\mathcal{D}^{1,p}(\R^n)}^p, \\
			\label{eq:lemiter-LpH-w}
				\lim_{m \to +\infty} \left[\norma*{H(\nabla u_m)}_{L^{p}(\R^n)}^p - \norma*{H(\nabla w_m)}_{L^{p}(\R^n)}^p\right] = \norma*{H(\nabla v)}_{L^{p}(\R^n)}^p, \\
			\label{eq:lemiter-energ-w}
				\mathcal{J}(w_m) \to c_\star - \mathcal{J}(v), \\
			\label{eq:lemiter-Fredif-w}
				\mathcal{J}'(w_m) \to 0 \quad\text{in } \mathcal{D}^{-1,p'}(\R^n).
			\end{gather}
		\end{lemma}
		\begin{proof}
			We first prove the existence of a subsequence along which
			\begin{equation}
			\label{eq:convgrad-vm}
				\nabla v_m \to \nabla v \quad\text{a.e.\ in } \R^n.
			\end{equation}
			For any~$k \in \N$, define~$\Omega_k \coloneqq B_k(0)$. Moreover, for any~$\phi \in C^\infty_c(\R^n)$ set
			\begin{equation*}
			\phi_m(x) \coloneqq \lambda_m^{\frac{p-n}{p}} \phi \!\left( \frac{x - y_m}{\lambda_m}\right) = T_{y_m,1/\lambda_m}(\phi)(x) \quad\text{for all } x \in \R^n.
			\end{equation*}
			Clearly,~$\phi_m \in C^\infty_c(\R^n)$, and point~\ref{it:cons-normap-anisot} in Lemma~\ref{lem:sym-anisot} and a change of variables give that
			\begin{equation*}
				\norma*{\phi_m}_{\mathcal{D}^{1,p}(\R^n)} = \norma*{\phi}_{\mathcal{D}^{1,p}(\R^n)} \quad\text{and}\quad \left\langle \mathcal{J}'(v_m), \phi \right\rangle = \left\langle \mathcal{J}'(u_m), \phi_m \right\rangle \!.
			\end{equation*}
			Therefore, we deduce
			\begin{equation}
			\label{eq:est-for-convFrediff}
				\begin{split}
					\abs*{\left\langle \mathcal{J}'(v_m), \phi \right\rangle} = \abs*{\left\langle \mathcal{J}'(u_m), \phi_m \right\rangle} &\leq \norma*{\mathcal{J}'(u_m)}_{\mathcal{D}^{-1,p'}(\R^n)} \,\norma*{\phi_m}_{\mathcal{D}^{1,p}(\R^n)} \\
					&= \norma*{\mathcal{J}'(u_m)}_{\mathcal{D}^{-1,p'}(\R^n)} \,\norma*{\phi}_{\mathcal{D}^{1,p}(\R^n)}.
				\end{split}
			\end{equation}
			Combining~\eqref{eq:est-for-convFrediff} with~\eqref{eq:lemit-Frediff-conv}, we infer that~$\mathcal{J}'(v_m) \to 0$ in~$\mathcal{D}^{-1,p'}(\R^n)$. Together with~\eqref{eq:lemit-convm}, this allows us to repeat the argument in the proof of Lemma~\ref{lem:conv-D1p-ener-Frediff} to get
			\begin{equation*}
				\int_{\Omega_k} \left\langle a \!\left(\nabla v_m\right) - a \!\left(\nabla v\right), \nabla \mathcal{T} \!\left(v_m-v\right) \right\rangle dx \to 0,
			\end{equation*}
			for any~$k \in \N$. Thus,~\eqref{eq:convgrad-vm} follows from Proposition~\ref{prop:exhaust}.
			
			By applying Proposition~\ref{prop:exhaust} again, we obtain
			\begin{gather*}
				\lim_{m \to +\infty} \left[\norma*{v_m}_{\mathcal{D}^{1,p}(\R^n)}^p - \norma*{v_m-v}_{\mathcal{D}^{1,p}(\R^n)}^p\right] = \norma*{v}_{\mathcal{D}^{1,p}(\R^n)}^p, \\
				\lim_{m \to +\infty} \left[\norma*{H(\nabla v_m)}_{L^{p}(\R^n)}^p - \norma*{H(\nabla v_m-v)}_{L^{p}(\R^n)}^p\right] = \norma*{H(\nabla v)}_{L^{p}(\R^n)}^p,
			\end{gather*}
			which are equivalent to~\eqref{eq:lemiter-D1p-w}--\eqref{eq:lemiter-LpH-w} via point~\ref{it:cons-normap-anisot} in Lemma~\ref{lem:sym-anisot}.
			
			From~\eqref{eq:convgrad-vm}, Lemma~\ref{lem:sym-anisot}, Lemma~\ref{lem:brez-lieb}, and~\eqref{eq:lemit-energ-conv}, it follows that
			\begin{equation*}
				\mathcal{J}(w_m) = \mathcal{J}(v_m-v) = \mathcal{J}(v_m) - \mathcal{J}(v) + o(1) = \mathcal{J}(u_m) - \mathcal{J}(v) + o(1) = c_\star - \mathcal{J}(v) + o(1),
			\end{equation*}
			which is precisely~\eqref{eq:lemiter-energ-w}.
			
			Finally, from the weak convergence in~\eqref{eq:lemit-convm} and the fact that~$\mathcal{J}'(v_m) \to 0$ in $\mathcal{D}^{-1,p'}(\R^n)$, we deduce that~$\mathcal{J}'(v) = 0$. Now, let~$\psi \in C^\infty_c(\R^n)$ be such that~$\norma*{\psi}_{\mathcal{D}^{1,p}(\R^n)} = 1$ and set
			\begin{equation*}
				\psi_m(x) \coloneqq \lambda_m^{\frac{n-p}{p}} \psi \!\left(\lambda_m x + y_m\right) \quad\text{for all } x \in \R^n.
			\end{equation*}
			Changing variables and arguing as for~\eqref{eq:rewrite-frediff}, we deduce that
			\begin{align*}
				\left\langle \mathcal{J}'(w_m) , \psi \right\rangle &= \left\langle \mathcal{J}'(v_m-v) , \psi_m \right\rangle = \left\langle \mathcal{J}'(v_m) , \psi_m \right\rangle - \left\langle \mathcal{J}'(v) , \psi_m \right\rangle + o(1) \\
				&= \left\langle \mathcal{J}'(u_m) , \psi \right\rangle + o(1) = o(1),
			\end{align*}
			where the last identity follows from~\eqref{eq:lemit-Frediff-conv}. This yields the validity of~\eqref{eq:lemiter-Fredif-w}.
		\end{proof}
		
		The subsequent lemma provides some boundedness and compactness properties of a Palais–Smale sequence for the energy functional~$\mathcal{J}$.
		
		\begin{lemma}
		\label{lem:bounded-seq}
			Let~$\{u_m\}_m \subseteq \mathcal{D}^{1,p}(\R^n)$ be a sequence satisfying
			\begin{equation}
				\label{eq:J-J'-conv}
				\mathcal{J}(u_m) \to c_\star \quad\text{and}\quad \mathcal{J}'(u_m) \to 0 \quad\text{in } \mathcal{D}^{-1,p'}(\R^n),
			\end{equation}
			then~$\{u_m\}_m \subseteq \mathcal{D}^{1,p}(\R^n)$ is uniformly bounded. Moreover, if in addition we have
			\begin{equation}
			\label{eq:threshold-J}
				c_\star < \frac{S_p^n}{n},
			\end{equation}
			then~$u_m \to 0$ strongly in~$\mathcal{D}^{1,p}(\R^n)$.
		\end{lemma}
		\begin{proof}
			By~\eqref{eq:J-J'-conv}, we have
			\begin{equation*}
				c_\star + o(1) = \mathcal{J}(u_m) = \int_{\R^n} \frac{1}{p} \, H^p(\nabla u_m) - \frac{1}{\past} \,\abs*{u_m}^{\past} dx,
			\end{equation*}
			and
			\begin{equation}
				\label{eq:J'-um}
				\begin{split}
					\left\langle \mathcal{J}'(u_m), u_m \right\rangle &= \int_{\R^n}  H^{p-1}(\nabla u_m) \left\langle \nabla H(\nabla u_m), \nabla u_m \right\rangle - \abs*{u_m}^{\past} dx \\
					&= \int_{\R^n}  H^{p}(\nabla u_m) - \abs*{u_m}^{\past} dx,
				\end{split}
			\end{equation}
			where we used the homogeneity of~$H$ in the last equality. Combining these two identities with~\eqref{eq:equiv-norma} and~\eqref{eq:J-J'-conv} we deduce that, for~$m$ sufficiently large, we have
			\begin{align*}
				c_\star + o(1) \norma*{u_m}_{\mathcal{D}^{1,p}(\R^n)} &\geq \mathcal{J}(u_m) - \frac{1}{\past} \left\langle \mathcal{J}'(u_m), u_m \right\rangle = \left(\frac{1}{p}-\frac{1}{\past}\right) \int_{\R^n} H^p(\nabla u_m) \, dx \\
				&= \frac{1}{n} \int_{\R^n} H^p(\nabla u_m) \, dx \geq \frac{c_H^p}{n} \norma*{u_m}_{\mathcal{D}^{1,p}(\R^n)}^p,
			\end{align*}
			thereby proving the first claim.
			
			Now assume that~\eqref{eq:threshold-J} holds. Since the sequence is uniformly bounded, from~\eqref{eq:J-J'-conv} it follows that
			\begin{equation}
				\label{eq:J'-wm}
				\left\langle \mathcal{J}'(u_m), u_m \right\rangle = o(1)
			\end{equation}
			and, for~$m$ sufficiently large, there holds
			\begin{equation*}
				\frac{1}{n} \int_{\R^n} H^p(\nabla u_m) \, dx \leq c_\star + o(1) < \frac{S_p^n}{n}.
			\end{equation*}
			Hence, for~$m$ large enough, we obtain
			\begin{equation}
				\label{eq:boundH-SnH}
				\norma*{H(\nabla u_m)}_{L^p(\R^n)}^p < S_p^n.
			\end{equation}
			Using~\eqref{eq:J'-um},~\eqref{eq:J'-wm}, and the anisotropic Sobolev inequality, we deduce that
			\begin{align*}
				\norma*{H(\nabla u_m)}_{L^p(\R^n)}^p &\left(1-S_p^{-\past}\norma*{H(\nabla u_m)}_{L^p(\R^n)}^{\past-p}\right) \\
				&= \norma*{H(\nabla u_m)}_{L^p(\R^n)}^p - S_p^{-\past}\norma*{H(\nabla u_m)}_{L^p(\R^n)}^{\past} \\
				&\leq \int_{\R^n}  H^{p}(\nabla u_m) - \abs*{u_m}^{\past} dx = o(1).
			\end{align*}
			The latter inequality, together~with \eqref{eq:boundH-SnH}, implies that~$u_m \to 0$ strongly in~$\mathcal{D}^{1,p}(\R^n)$.
		\end{proof}
	
	
	\subsection{Proof of Theorem~\ref{th:struwe-def}.}
	
	For clarity, we divide the proof into several steps, following the scheme of~\cite{merc-will}.
	
	\renewcommand\thesubsubsection{\itshape Step \arabic{subsubsection}}
	
	
	\subsubsection{Weak convergence of the sequence.}
	\label{step:proofstruwe-boundedness}
	
	We start by noticing that, thanks to Lemma~\ref{lem:bounded-seq}, the sequence~$\{u_m\}_m \subseteq \mathcal{D}^{1,p}(\R^n)$ is uniformly bounded. As a result, we can extract a subsequence such that
	\begin{equation}
	\label{eq:convum-prrofstruwe}
		u_m \to v_0 \quad\text{weakly in } \mathcal{D}^{1,p}(\R^n) \quad\text{and}\quad u_m \to v_0 \quad\text{a.e.\ in } \R^n.
	\end{equation}
	By Lemma~\ref{lem:conv-D1p-ener-Frediff}, it follows that~$\mathcal{J}'(v_0)=0$ and that
	\begin{equation*}
		u_m^1 \coloneqq u_m - v_0
	\end{equation*}
	satisfies
	\begin{gather}
	\label{eq:conv-D1p-u1m}
		\norma{u_m^1}_{\mathcal{D}^{1,p}(\R^n)}^p = \norma*{u_m}_{\mathcal{D}^{1,p}(\R^n)}^p - \norma*{v_0}_{\mathcal{D}^{1,p}(\R^n)}^p + o(1), \\
	\label{eq:conv-LpH-u1m}
	\norma{H(\nabla u_m^1)}_{L^{p}(\R^n)}^p = \norma*{H(\nabla u_m)}_{L^{p}(\R^n)}^p - \norma*{H(\nabla v_0)}_{L^{p}(\R^n)}^p + o(1), \\
	\label{eq:conv-energ-u1m}
		\mathcal{J}(u_m^1) \to c_\star - \mathcal{J}(v_0), \\
	\label{eq:conv-Fredif-u1m}
		\mathcal{J}'(u_m^1) \to 0 \quad\text{in } \mathcal{D}^{-1,p'}(\R^n).
	\end{gather}
	
	
	\subsubsection{Introduction of the Lévy concentration function.}
	\label{step:proofstruwe-concfunct}
	
	If~$u_m^1 \to 0$ in~$L^{\past}\!(\R^n)$, then testing with~$u_m^1$,~\eqref{eq:conv-Fredif-u1m} implies that~$u_m^1 \to 0$ in~$\mathcal{D}^{1,p}(\R^n)$, completing the proof. Otherwise, we can assume that, up to a subsequence,
	\begin{equation*}
		\int_{\R^n} \,\abs{u_m^1}^{\past} dx > \delta,
	\end{equation*}
	for some~$\delta \in (0, 2^{-n/p} S_p^n)$.
	
	We now introduce the Lévy concentration function as
	\begin{equation*}
		Q_m(r) \coloneqq \sup_{y \in \R^n} \int_{B_r(y)} \,\abs{u_m^1}^{\past} dx.
	\end{equation*}
	Since~$Q_m(0)=0$ and~$Q_m(\infty)>\delta$, there exist a sequence of points~$\{y^1_m\}_m \subseteq \R^n$ and a sequence~$\{\lambda^1_m\}_m \subseteq (0,+\infty)$ such that
	\begin{equation}
	\label{eq:delta-rescale}
		\delta = \sup_{y \in \R^n} \int_{B_{\lambda^1_m}(y)} \,\abs{u_m^1}^{\past} dx = \int_{B_{\lambda^1_m}(y^1_m)} \,\abs{u_m^1}^{\past} dx.
	\end{equation}
	We define
	\begin{equation}
	\label{eq:v^1m-def}
		v^1_m(x) \coloneqq \left(\lambda^1_m\right)^{\!\frac{n-p}{p}} u^1_m \!\left(\lambda^1_m x + y^1_m\right) = T_{-y^1_m /\lambda^1_m, \lambda^1_m}\!\left(u^1_m\right)\!(x) \quad\text{for all } x \in \R^n.
	\end{equation}
	
	Therefore, point~\ref{it:cons-normap-anisot} in Lemma~\ref{lem:sym-anisot}, together with~\ref{step:proofstruwe-boundedness}, enables us to assume that
	\begin{equation}
	\label{eq:conv-v1m}
		v^1_m \to v_1 \quad\text{weakly in } \mathcal{D}^{1,p}(\R^n) \quad\text{and}\quad v^1_m \to v_1 \quad\text{a.e.\ in } \R^n.
	\end{equation}
	Furthermore, a change of variables, together with~\eqref{eq:delta-rescale}, gives that
	\begin{equation}
	\label{eq:deltaV1m}
		\delta = \sup_{y \in \R^n} \int_{B_{1}(y)} \,\abs{v_m^1}^{\past} dx = \int_{B_{1}(0)} \,\abs{v_m^1}^{\past} dx.
	\end{equation}

	
	\subsubsection{Non triviality of $v_1$.}
	\label{step:proofstruwe-nontriv}
	
	We now claim that~$v_1 \neq 0$.
	
	We begin by noticing that, thanks to the definition of~$v_m^1$ in~\eqref{eq:v^1m-def} and~\eqref{eq:conv-Fredif-u1m}, and by arguing as in~\eqref{eq:est-for-convFrediff}, we deduce that
	\begin{equation}
	\label{eq:Frediff-v1m-to0}
		\mathcal{J}'(v_m^1) \to 0 \quad\text{in } \mathcal{D}^{-1,p'}(\R^n).
	\end{equation}
	
	Suppose, by contradiction, that~$v_1 = 0$. Thanks to~\eqref{eq:conv-v1m} and the Rellich–Kondrachov theorem, we may assume that, up to a subsequence,
	\begin{equation}
	\label{eq:conv-Lploc-Vm1}
		v_m^1 \to 0 \quad\text{in } L^p_{\loc}(\R^n).
	\end{equation}
	Let~$\phi \in C^\infty_c(\R^n)$ be such that~$\mathrm{supp} \,\phi \subseteq \overline{B_1(0)}$, and note that, by H\"older inequality,
	\begin{equation}
	\label{eq:tovomb-contr-1}
	\begin{split}
		\int_{\R^n} \,\abs*{\phi}^p \,\abs{v_m^1}^{\past} dx &= \int_{\R^n} \,\abs{v_m^1}^{\past-p} \,\abs{\phi v_m^1}^p dx \\
		&\leq \left(\int_{\mathrm{supp} \,\phi} \,\abs{v_m^1}^{\past} dx\right)^{\!\frac{p}{n}} \left(\int_{\R^n} \,\abs{\phi v_m^1}^{\past} dx\right)^{\!\frac{n-p}{n}}.
	\end{split}
	\end{equation}
	Moreover, by~\eqref{eq:deltaV1m} and our choice of~$\phi$, we see that
	\begin{equation}
	\label{eq:tovomb-contr-2-TENERE}
		 \left(\int_{\mathrm{supp} \,\phi} \,\abs{v_m^1}^{\past} dx\right)^{\!\frac{p}{n}} \leq \left(\int_{B_1(0)} \,\abs{v_m^1}^{\past} dx\right)^{\!\frac{p}{n}} = \delta^{\frac{p}{n}}.
	\end{equation}
	Finally, the anisotropic Sobolev inequality yields
	\begin{equation}
	\label{eq:tovomb-contr-3}
		\left(\int_{\R^n} \,\abs{\phi v_m^1}^{\past} dx\right)^{\!\frac{n-p}{n}} = \norma{\phi v_m^1}_{L^{\past}\!(\R^n)}^p \leq S_p^{-p} \norma{H(\nabla (\phi v_m^1))}_{L^{p}(\R^n)}^p.
	\end{equation}
	Combining~\eqref{eq:tovomb-contr-1}--\eqref{eq:tovomb-contr-3} with the property~$\delta\in (0,2^{-n/p}S_p^n)$, we conclude that
	\begin{equation}
	\label{eq:topiec-contr-1}
		\int_{\R^n} \,\abs*{\phi}^p \,\abs{v_m^1}^{\past} dx \leq \frac{1}{2} \,\norma{H(\nabla (\phi v_m^1))}_{L^{p}(\R^n)}^p.
	\end{equation}

	The strong convergence in~\eqref{eq:conv-Lploc-Vm1} implies that
	\begin{equation}
	\label{eq:topiec-contr-2}
		\int_{\R^n} H^p(\nabla (\phi v_m^1)) \, dx = \int_{\R^n} H^p(\phi \nabla v_m^1) \, dx + o(1) = \int_{\R^n} \,\abs*{\phi}^p H^p(\nabla v_m^1) \, dx + o(1).
	\end{equation}
	Then, by the homogeneity of~$H$ and~\eqref{eq:conv-Lploc-Vm1}, it follows that
	\begin{equation}
	\label{eq:equality-integrals}
		\left\langle \mathcal{J}'(v_m^1), \abs*{\phi}^p v_m^1 \right\rangle = \int_{\R^n} \,\abs*{\phi}^p H^p(\nabla v_m^1) - \abs*{\phi}^p \,\abs{v_m^1}^{\past} dx + o(1).		
	\end{equation}
	We now claim that
	\begin{equation}
	\label{eq:crochet-o(1)}
		\left\langle \mathcal{J}'(v_m^1), \abs*{\phi}^p v_m^1 \right\rangle = o(1).
	\end{equation}
	Indeed, by the weak convergence in~\eqref{eq:conv-v1m} and the strong local convergence in~\eqref{eq:conv-Lploc-Vm1}, the sequence~$\{\abs*{\phi}^p v_m^1\}_m \subseteq \mathcal{D}^{1,p}(\R^n)$ is uniformly bounded, for~$m$ sufficiently large. This piece of information coupled with~\eqref{eq:Frediff-v1m-to0} implies~\eqref{eq:crochet-o(1)}.
	
	Now, from~\eqref{eq:equality-integrals} and~\eqref{eq:crochet-o(1)}, we infer that
	\begin{equation}
	\label{eq:topiec-contr-3}
		 \int_{\R^n} \,\abs*{\phi}^p H^p(\nabla v_m^1) \, dx = \int_{\R^n} \, \abs*{\phi}^p \,\abs{v_m^1}^{\past} dx + o(1).		
	\end{equation}
	Therefore, piecing together~\eqref{eq:topiec-contr-1},~\eqref{eq:topiec-contr-2}, and~\eqref{eq:topiec-contr-3}, we deduce
	\begin{equation*}
		\int_{\R^n} H^p(\nabla (\phi v_m^1)) \, dx \leq \frac{1}{2} \int_{\R^n} H^p(\nabla (\phi v_m^1)) \, dx + o(1).
	\end{equation*}
	The latter, in turn, implies~$\nabla v_m^1 \to 0$ in~$L^{p}_{\loc}(\R^n)$ and~$v_m^1 \to 0$ in~$L^{\past}_{\loc}(\R^n)$, by the Sobolev inequality. This contradicts~\eqref{eq:deltaV1m}, thereby proving that~$v_1$ must be non-trivial.
	
	
	\subsubsection{Towards the iteration.}
	\label{step:proofstruwe-towiter}
	
	We notice that, if~$y^1_m \to y^1$ and~$\lambda^1_m \to \lambda^1 \in (0,+\infty)$, we obtain a contradiction. Indeed, from the weak convergences in~\eqref{eq:convum-prrofstruwe} and~\eqref{eq:conv-v1m} and recalling the definition of~$v_m^1$ in~\eqref{eq:v^1m-def}, we would have~$v_m^1 \to 0 = v_1$ weakly in~$\mathcal{D}^{1,p}(\R^n)$, which contradicts~\ref{step:proofstruwe-nontriv}.
	
	We now define the set of indices
	\begin{equation*}
		M_1 \coloneqq \left\{m \in \N \bigm| \lambda_m^1 \leq 1\right\} \quad\text{and}\quad M_k \coloneqq \left\{m \in \N \bigm| k-1<\lambda_m^1 \leq k\right\} \quad\text{for } k \geq 2.
	\end{equation*}
	Thus, if~$\mathrm{card}(M_k) < +\infty$ for all~$k \geq 1$, then~$\lambda^1_m \to +\infty$. Otherwise, in light of the previous observation, the only possibility is that~$\mathrm{card}(M_1) = +\infty$ and~$\lambda^1_m \to 0$, up to passing to a subsequence if necessary.
	
	Next, from the definition of~$v_m^1$ in~\eqref{eq:v^1m-def}, along with~\eqref{eq:conv-energ-u1m},~\eqref{eq:conv-Fredif-u1m},~\eqref{eq:conv-v1m}, and Lemma~\ref{lem:struwe-iter}, we deduce that~$\mathcal{J}'(v_1)=0$.
	
	Then, we define
	\begin{equation*}
		u^2_m(x) \coloneqq u^1_m(x) - \left(\lambda^1_m\right)^{\!\frac{p-n}{p}} v_1 \!\left(\frac{x-y^1_m}{\lambda^1_m}\right) \quad\text{for all } x \in \R^n.
	\end{equation*}
	From Lemma~\ref{lem:struwe-iter} and~\eqref{eq:conv-D1p-u1m}, we have
	\begin{align*}
		\norma{u_m^2}_{\mathcal{D}^{1,p}(\R^n)}^p &= \norma{u_m^1}_{\mathcal{D}^{1,p}(\R^n)}^p - \norma*{v_1}_{\mathcal{D}^{1,p}(\R^n)}^p + o(1) \\
		&= \norma{u_m}_{\mathcal{D}^{1,p}(\R^n)}^p - \norma*{v_0}_{\mathcal{D}^{1,p}(\R^n)}^p - \norma*{v_1}_{\mathcal{D}^{1,p}(\R^n)}^p + o(1),
	\end{align*}
	and using~\eqref{eq:conv-LpH-u1m} we also get
	\begin{align*}
		\norma{H(\nabla u_m^2)}_{L^{p}(\R^n)}^p &= \norma{H(\nabla u_m^1)}_{L^{p}(\R^n)}^p - \norma*{H(\nabla v_1)}_{L^{p}(\R^n)}^p + o(1) \\
		&= \norma{H(\nabla u_m)}_{L^{p}(\R^n)}^p - \norma*{H(\nabla v_0)}_{L^{p}(\R^n)}^p - \norma*{H(\nabla v_1)}_{L^{p}(\R^n)}^p + o(1).
	\end{align*}
	Similarly, from Lemma~\ref{lem:struwe-iter} and~\eqref{eq:conv-energ-u1m}, we obtain
	\begin{equation}
	\label{eq:J-to-iter}
		\mathcal{J}(u_m^2) = \mathcal{J}(u_m^1) - \mathcal{J}(v_1) + o(1) =  c_\star - \mathcal{J}(v_0) - \mathcal{J}(v_1) + o(1).
	\end{equation}
	Moreover, applying Lemma~\ref{lem:struwe-iter} once more, we deduce that~$\mathcal{J}'(u_m^2) \to 0$ in~$\mathcal{D}^{-1,p'}(\R^n)$.
	
	
	\subsubsection{Iterating the procedure.}
	\label{step:proofstruwe-iter}
	
	Let~$u \in \mathcal{D}^{1,p}(\R^n)$ be a non-trivial critical point of the functional~$\mathcal{J}$. A straightforward calculation, using the homogeneity of~$H$, reveals that
	\begin{equation}
	\label{eq:eqint-critpoint}
		0 = \left\langle \mathcal{J}'(u), u \right\rangle = \int_{\R^n} H^p(\nabla u) - \abs*{u}^{\past} dx,
	\end{equation}
	which implies
	\begin{equation*}
		\mathcal{J}(u) = \frac{1}{n} \,\norma*{u}_{L^{\past}\!(\R^n)}^{\past}.
	\end{equation*}
	Combining~\eqref{eq:eqint-critpoint} with the anisotropic Sobolev inequality, we also obtain
	\begin{equation*}
		\norma*{u}_{L^{\past}\!(\R^n)}^p \leq S^{-p}_p \norma*{H(\nabla u)}_{L^{p}(\R^n)}^p = S^{-p}_p \norma*{u}_{L^{\past}\!(\R^n)}^{\past},
	\end{equation*}
	from which, since~$\past>p$, it follows that
	\begin{equation}
	\label{eq:boundbelowenerg}
		\mathcal{J}(u) \geq \mathcal{J}_0 \coloneqq \frac{S_p^n}{n}.
	\end{equation}
	In view of~\eqref{eq:boundbelowenerg}, the procedure of the previous steps iterates only a finite number of times, constructing the sequences~$\left\{v_i\right\}_i$,~$\left\{y_m^i\right\}_m$,~$\left\{\lambda_m^i\right\}_m$, and~$\left\{v_m^i\right\}_m$.
	
	Indeed, from~\ref{step:proofstruwe-towiter}, we deduce that~$\mathcal{J}'(v_i)=0$ for all~$i \geq 1$, therefore each~$v_i$ enjoys the lower bound~\eqref{eq:boundbelowenerg}, i.e.,~$\mathcal{J}(v_i) \geq \mathcal{J}_0$ for all~$i \geq 1$. Therefore, iterating~\eqref{eq:J-to-iter}, the threshold in~\eqref{eq:threshold-J} will eventually be crossed and, in light of Lemma~\ref{lem:bounded-seq} and~\ref{step:proofstruwe-concfunct}, the proof is complete.
	
	
	\subsubsection{Positivity.}
	\label{step:proofstruwe-posit}
	
	If~$k \geq 1$ and~$u_m \geq 0$ a.e.\ in~$\R^n$ for all~$m \in \N$, it suffices to consider the functional
	\begin{equation*}
		\mathcal{J}_+(u) \coloneqq \int_{\R^n} \frac{1}{p} \, H^p(\nabla u) - \frac{1}{\past} \,u_{+}^{\past}  dx
	\end{equation*}
	and apply the strong maximum principle of Theorem~2.5.1 in~\cite{ps} to deduce that~$v_i > 0$ for all~$i = 1,\dots,k$.

	
	\section{Interaction estimate}
	\label{sec:interaction}
	
	Following the work of Brezis \& Coron~\cite{brezis-coron}, Bahri \& Coron~\cite{bahri-coron} observed that, in the Euclidean setting and for Struwe's decomposition in the case~$p=2$, one can also assert that
	\begin{equation}
		\label{eq:interaction-p=2}
		\frac{\lambda_m^i}{\lambda_m^j} + \frac{\lambda_m^j}{\lambda_m^i} +  \frac{\abs*{y^i_m-y^j_m}^2}{\lambda_m^i \lambda_m^j} \to +\infty \quad\text{as } m \to+\infty \text{ for all } i \neq j.
	\end{equation}
	This condition is closely related to the interaction of a family of Euclidean~$2$-bubbles. Indeed, for~$U = U_2[z_1, \lambda_1]$ and~$V = U_2[z_2, \lambda_2]$, one has
	\begin{align*}
		\int_{\R^n} \left\langle \nabla U, \nabla V \right\rangle dx &= \int_{\R^n} U^{2^\ast-1} V \, dx = \int_{\R^n} U V^{2^\ast-1} \, dx \\
		&= c_n \min\left\{\frac{\lambda_1}{\lambda_2}, \frac{\lambda_2}{\lambda_1}, \frac{\lambda_1 \lambda_2}{\abs*{z_1-z_2}^2} \right\}^{\!\frac{n-2}{2}}\!,
	\end{align*}
	for some dimensional constant~$c_n>0$.
	
	To the best of our knowledge, no analogous interaction result is known for~$p\neq 2$, even in the Euclidean case. Since understanding these conditions is fundamental also for quantitative purposes -- see, e.g.,~\cite{dsw,fg} -- we provide in this section the counterpart of~\eqref{eq:interaction-p=2}, also in the anisotropic case, as stated in the following result, which is Theorem~\ref{th:struwe-intro}.
	
	\begin{theorem}
		\label{th:struwe-with-inter}
		For~$n \in \N$ and~$1<p<n$, let~$\left\{u_m\right\}_m \subseteq \mathcal{D}^{1,p}(\R^n)$ be a sequence of non-negative functions satisfying
		\begin{equation}
		\label{eq:hyp-sruwe-anisot-2}
			\mathcal{J}'(u_m) \to 0 \quad\text{in } \mathcal{D}^{-1,p'}(\R^n).
		\end{equation}
		Moreover, suppose that
		\begin{equation}
			\label{eq:energ-cond-struwe}
			\left(k-\frac{1}{2}\right) S_p^n \leq \int_{\R^n} H^p(\nabla u_m) \, dx \leq \left(k+\frac{1}{2}\right) S_p^n
		\end{equation}
		for every~$m \in \N$ and for some~$k \in \N$ with~$k \geq 1$.
		
		Then, possibly passing to a subsequence, there exist positive solutions~$v_1,\dots,v_k \in \mathcal{D}^{1,p}(\R^n)$ to~\eqref{eq:crit-anis-struwe}, i.e.,~$(p,H)$-bubbles of the form~\eqref{eq:pH-bubb}, and~$k$ sequences~$\left\{y^i_m\right\}_m \subseteq \R^n$ and~$\left\{\lambda^i_m\right\}_m \subseteq (0,+\infty)$ such that
		\begin{gather}
			\notag
			\norma*{u_m - \sum_{i=1}^{k} \left(\lambda^i_m\right)^{\!\frac{p-n}{p}} v_i\left(\frac{\cdot - y^i_m}{\lambda^i_m}\right)}_{\mathcal{D}^{1,p}(\R^n)} \to 0, \\
			\label{eq:struwe-plap-2}
			\norma*{u_m}_{\mathcal{D}^{1,p}(\R^n)}^p \to \sum_{i=1}^{k} \,\norma*{v_i}_{\mathcal{D}^{1,p}(\R^n)}^p \quad\text{and}\quad \norma*{H(\nabla u_m)}_{L^{p}(\R^n)}^p \to \sum_{i=1}^{k} \,\norma*{H(\nabla v_i)}_{L^{p}(\R^n)}^p, \\
            \notag
			\mathcal{J}(u_m) \to c_\star=
			\sum_{i=1}^{k} \mathcal{J}(v_i),
		\end{gather}
        for some $c_\star\in \R$.
		Moreover, if~$y^i_m \to y^i$ as~$m \to +\infty$, then either~$\lambda^i_m \to 0$ or~$\lambda^i_m \to +\infty$. In addition, it follows that
		\begin{equation}
			\label{eq:interaction-plap}
			\max\left\{\frac{\lambda^i_m}{\lambda^j_m}, \frac{\lambda^j_m}{\lambda^i_m}, \frac{\abs*{y^i_m-y^j_m}^2}{\lambda_m^i \lambda_m^j} \right\} \to +\infty \quad\text{as } m \to+\infty \text{ for all } i \neq j.
		\end{equation}
	\end{theorem}
	
	Observe that~\eqref{eq:interaction-p=2} and~\eqref{eq:interaction-plap} are equivalent. Therefore, we recover the known result for the Laplacian.
	
	We will need the following algebraic lemma, whose proof is deferred to Appendix~\ref{app:lem}.
	
	\begin{lemma}
		\label{lem:xi-p}
		Let~$k \in \N$ be a natural number, assume that~$x_1,\dots,x_k \in \R^n$ and~$p>1$, then there exists a constant~$\mathsf{C}_{p,k}>0$, depending only on~$p$ and~$k$, such that
		\begin{equation}
		\label{eq:xi-p-stat}
			\abs*{\abs*{\sum_{i=1}^{k} x_i}^p - \sum_{i=1}^{k} \,\abs*{x_i}^p} \leq \mathsf{C}_{p,k} \sum_{i=1}^{k} \sum_{\substack{j=1 \\ j \neq i}}^{k} \,\abs*{x_i}^{p-1} \abs*{x_j}.
		\end{equation}
	\end{lemma}
	
	The following simple result is the counterpart of Lemma~5 in~\cite{brezis-coron}.
	
	\begin{lemma}
		\label{lem:sum-bub}
		Let~$k \in \N$ be a natural number. Let~$u,u_1,\dots,u_k \in \mathcal{D}^{1,p}(\R^n)$ be a finite number of non-negative solutions to~\eqref{eq:critica-anisot} and assume that
		\begin{equation}
			\label{eq:sum-u-ui}
			u = \sum_{i=1}^{k} u_i.
		\end{equation}
		Then, it follows that
		\begin{equation}
			\label{eq:est-energ-sum}
			\int_{\R^n} \,\abs*{\nabla u}^p \, dx \leq \sum_{i=1}^{k} \int_{\R^n} \,\abs*{\nabla u_i}^p \, dx.
		\end{equation}
		Moreover, if equality holds in~\eqref{eq:est-energ-sum}, then each~$u_i$ is trivial, with the possible exception of one of them.
	\end{lemma}
	\begin{proof}
		Although the result is rather elementary, we include the details of the proof.
		
		We first note that since all~$u_i$ are non-negative, if each~$u_i$ is trivial, then~$u$ is also trivial by~\eqref{eq:sum-u-ui} and equality holds in~\eqref{eq:est-energ-sum}. Therefore, we may assume that at least one of the~$u_i$ is non-trivial, say~$u_1$.
		
		In addition, by~\eqref{eq:energ-bubb-anisot}, we know that for any non-trivial solution~$U \in \mathcal{D}^{1,p}(\R^n)$ to~\eqref{eq:critica-anisot} we have
		\begin{equation}
			\label{eq:enerHgradU}
			\int_{\R^n} H^p(\nabla U) \, dx = S_p^n.
		\end{equation}
		Now, we observe that thanks to~\eqref{eq:equiv-norma} and~\eqref{eq:enerHgradU}, if follows that
		\begin{equation*}
			\int_{\R^n} \,\abs*{\nabla U_p[0,1]}^p \, dx = \sigma >0,
		\end{equation*}
		moreover, by points~\ref{it:cons-normapas-anisot} and~\ref{it:transf-unitbub-anisot} of Lemma~\ref{lem:sym-anisot}, for any non-trivial solution~$U \in \mathcal{D}^{1,p}(\R^n)$ to~\eqref{eq:critica-anisot} we have
		\begin{equation}
			\label{eq:enerEugradU}
			\int_{\R^n} \,\abs*{\nabla U}^p \, dx = \sigma.
		\end{equation}
		Since all~$u_i$ are non-negative, by~\eqref{eq:sum-u-ui} also~$u$ is non-trivial and, from~\eqref{eq:enerEugradU}, we get
		\begin{equation*}
			\int_{\R^n} \,\abs*{\nabla u}^p \, dx = \sigma = \int_{\R^n} \,\abs*{\nabla u_1}^p \, dx \leq \sum_{i=1}^{k} \int_{\R^n} \,\abs*{\nabla u_i}^p \, dx.
		\end{equation*}
		
		Finally, suppose that the equality holds in~\eqref{eq:est-energ-sum}. If~$k=1$, we have nothing to prove, so we may assume~$k \geq 2$. Suppose by contradiction that, up to reordering,~$u_1,\dots,u_j$ are non-trivial for some~$2 \leq j \leq k$. By~\eqref{eq:sum-u-ui} and since all~$u_i$ are non-negative,~$u$ is again non-trivial and, from~\eqref{eq:enerHgradU}, we deduce
		\begin{equation*}
			\sigma = \int_{\R^n} \,\abs*{\nabla u}^p \, dx = \sum_{i=1}^{j} \int_{\R^n} \,\abs*{\nabla u_i}^p \, dx = j \sigma.
		\end{equation*}
		Since~$\sigma>0$, this implies that~$1=j \geq 2$, a contradiction.
	\end{proof}
	
	We are now in a position to prove the main result of this section.
	
	\begin{proof}[Proof of Theorem~\ref{th:struwe-with-inter}]
		We start by claiming that the sequence~$\{u_m\}_m \subseteq \mathcal{D}^{1,p}(\R^n)$ is Palais-Smale for~$\mathcal{J}$, i.e.,~\eqref{eq:hyp-sruwe-anisot} holds true.
		
		To this end we note that, by~\eqref{eq:equiv-norma} and~\eqref{eq:energ-cond-struwe}, the sequence~$\{u_m\}_m \subseteq \mathcal{D}^{1,p}(\R^n)$ is uniformly bounded. Hence, from~\eqref{eq:hyp-sruwe-anisot-2}, we deduce that
		\begin{equation*}
			\left\langle \mathcal{J}'(u_m), u_m \right\rangle = o(1).
		\end{equation*}
		This, together with~\eqref{eq:J'-um}, yields
		\begin{equation*}
			\mathcal{J}(u_m) = \frac{1}{n} \int_{\R^n} H^p(\nabla u_m) \, dx + o(1).
		\end{equation*}
		As a consequence,~\eqref{eq:hyp-sruwe-anisot-2} implies that~$\{\mathcal{J}(u_m)\}_m \subseteq \R$ is uniformly bounded. Thus, possibly passing to a subsequence, there exists~$c_\star \in \R$ such that
		\begin{equation*}
			\mathcal{J}(u_m) \to c_\star,
		\end{equation*}
		thereby proving the assertion.
		
		Since~\eqref{eq:hyp-sruwe-anisot} holds, Theorem~\ref{th:struwe-def} applies and, considering also the energy bound~\eqref{eq:energ-cond-struwe}, we infer the validity of~\eqref{eq:struwe-plap-2}. Thus, it remains to prove~\eqref{eq:interaction-plap}.
		
		For this purpose, we follow the scheme of the proof of Theorem~2 in~\cite{brezis-coron}. Let us introduce the equivalence relation on the integers $i,j \in \left\{1,\dots,k\right\}$ defined by
		\begin{equation*}
			i \sim j \text{ if and only if } \max\left\{\frac{\lambda^i_m}{\lambda^j_m}, \frac{\lambda^j_m}{\lambda^i_m}, \frac{\abs*{y^i_m-y^j_m}^2}{\lambda_m^i \lambda_m^j} \right\} \text{ remains bounded as } m \to +\infty,
		\end{equation*}
		and denote by~$I_1,\dots,I_h$ the equivalence classes. We shall prove that
		each equivalence class contains precisely one element, which implies the validity of~\eqref{eq:interaction-plap}.
		
		For each~$i=1,\dots,k$, we define
		\begin{equation}
		\label{eq:def-vjm}
			v_m^i(x) \coloneqq \left(\lambda^i_m\right)^{\!\frac{p-n}{p}} v_i \!\left(\frac{x - y^i_m}{\lambda^i_m}\right) = T_{y^i_m,1/\lambda^i_m}(v_i)(x),
		\end{equation}
		according to the notation of Section~\ref{sec:symm-anisot}. We first claim that
		\begin{equation}
			\label{eq:ij-non-equiv}
			\int_{\R^n} \,\abs*{\nabla v_m^i}^{p-1} \abs*{\nabla v_m^j} \, dx = o(1) \quad\text{as } m \to +\infty \text{ for } i \not\sim j.
		\end{equation}
		Changing variables, we get
		\begin{equation}
			\label{eq:change-var-vim}
			\int_{\R^n} \,\abs*{\nabla v_m^i}^{p-1} \abs*{\nabla v_m^j} \, dx = \int_{\R^n} \phi_m^i \,\abs*{\nabla v^j} \, dx = \int_{\R^n} \,\abs*{\nabla v_i}^{p-1} \phi_m^j \, dx,
		\end{equation}
		where we set
		\begin{gather*}
			\phi_m^i(x) \coloneqq \left(\frac{\lambda^j_m}{\lambda^i_m}\right)^{\!\!\frac{n(p-1)}{p}}\abs*{\nabla v_i \!\left(\frac{\lambda^j_m x + y^j_m - y^i_m}{\lambda^i_m}\right)}^{p-1}, \\
			\phi_m^j(x) \coloneqq \left(\frac{\lambda^i_m}{\lambda^j_m}\right)^{\!\!\frac{n}{p}}\abs*{\nabla v_j \!\left(\frac{\lambda^i_m x + y^i_m - y^j_m}{\lambda^j_m}\right)}.
		\end{gather*}
		Since~$i$ and~$j$ are not is the same equivalence class, we may assume, up to subsequences, that the following trichotomy holds
		\begin{multline}
			\label{eq:lij-1}
			\text{either }\;\frac{\lambda^i_m}{\lambda^j_m} \to 0 \;\text{ or } \;\frac{\lambda^j_m}{\lambda^i_m} \to 0 \\
			\text{or } \;\frac{\lambda^i_m}{\lambda^j_m} \to \ell \;\text{ and } \;\frac{\abs*{y^i_m-y^j_m}^2}{\lambda_m^i \lambda_m^j} \to +\infty \quad\text{as } m \to +\infty,
		\end{multline}
		for some~$\ell \in (0,+\infty)$. In particular, it is easy to see that the last possibility in~\eqref{eq:lij-1} entails 
		\begin{equation}
			\label{eq:lij-2}
			\frac{\lambda^i_m}{\lambda^j_m} \to \ell \quad\text{and}\quad  \frac{\abs*{y^i_m-y^j_m}}{\lambda_m^j} \to +\infty \quad\text{as } m \to +\infty.
		\end{equation}
		Changing variables and using~\eqref{eq:energ-bubb-anisot} and~\eqref{eq:equiv-norma}, we have
		\begin{gather*}
			\int_{\R^n} \,\abs*{\phi_m^i}^{p'} dx = \int_{\R^n} \,\abs*{\nabla v_i}^p \, dx \leq C_H^p S_p^n, \\
			\int_{\R^n} \,\abs*{\phi_m^j}^p \, dx = \int_{\R^n} \,\abs*{\nabla v_j}^p \, dx \leq C_H^p S_p^n,
		\end{gather*}
		hence the sequence~$\{\phi_m^i\}_m$ is bounded in~$L^{p'}\!(\R^n)$ and~$\{\phi_m^j\}_m$ is bounded in~$L^p(\R^n)$. Moreover,it is readily seen that~\eqref{eq:lij-1}--\eqref{eq:lij-2} imply that~$\phi_m^i \to 0$ and~$\phi_m^j \to 0$ a.e.\ in~$\R^n$, whence
		\begin{equation*}
			\phi_m^i \to 0 \quad\text{weakly in } L^{p'}\!(\R^n) \quad\text{and}\quad \phi_m^j \to 0 \quad\text{weakly in } L^p(\R^n).
		\end{equation*}
		Since~$\abs*{\nabla v_j} \in L^{p}(\R^n)$ and~$\abs*{\nabla v_i}^{p-1} \in L^{p'}\!(\R^n)$, these convergences, together with~\eqref{eq:change-var-vim}, yield the validity of~\eqref{eq:ij-non-equiv}.
		
		For~$i \sim j$, we note that
		\begin{equation*}
			\frac{\abs*{y^i_m-y^j_m}^2}{(\lambda_m^j)^2} = \frac{\lambda^i_m}{\lambda^j_m} \frac{\abs*{y^i_m-y^j_m}^2}{\lambda_m^i \lambda_m^j} \quad\text{remains bounded as } m \to +\infty,
		\end{equation*}
		so we are led to introduce the quantities
		\begin{equation*}
			\ell_{ij} \coloneqq \lim_{m \to +\infty} \frac{\lambda^i_m}{\lambda^j_m} \quad\text{and}\quad p_{ij} \coloneqq \lim_{m \to +\infty} \frac{y^i_m-y^j_m}{\lambda_m^j}.
		\end{equation*}
		For each equivalence class~$I$, we fix an index~$i \in I$ and set
		\begin{equation}\label{def:vI}
			v_I(x) \coloneqq \sum_{j \in I} \ell_{ij}^{\frac{n-p}{p}} v_j(\ell_{ij} x + p_{ij}) = v_i(x) + \sum_{j \in I \setminus \{i\}} \ell_{ij}^{\frac{n-p}{p}} v_j(\ell_{ij} x + p_{ij}),
		\end{equation}
		which is non-negative.
		
		We claim that, for each equivalence class~$I$, there holds
		\begin{equation}
			\label{eq:ener-sum-I}
			\int_{\R^n} \,\abs*{\sum_{j \in I} \nabla v_m^j}^p \, dx = \int_{\R^n} \,\abs*{\nabla v_I}^p \, dx + o(1) \quad\text{as } m \to +\infty.
		\end{equation}
		Indeed, changing variables and recalling the definition of~$v_m^j$ in~\eqref{eq:def-vjm}, we get
		\begin{align*}
			\int_{\R^n} \,\abs*{\sum_{j \in I} \nabla v_m^j(x)}^p \, dx &= \int_{\R^n} \,\abs*{\nabla v_m^i(x) + \sum_{j \in I \setminus \{i\}} \nabla v_m^j(x)}^p \, dx \\
			&= \int_{\R^n} \,\abs*{\nabla v_i(x) + \sum_{j \in I \setminus \{i\}} \left(\frac{\lambda^i_m}{\lambda^j_m}\right)^{\!\!\frac{n}{p}} \nabla v_j \!\left(\frac{\lambda^i_m x + y^i_m - y^j_m}{\lambda^j_m}\right) }^p \, dx.
		\end{align*}
		Moreover, passing to a subsequence if necessary, we have
		\begin{equation*}
			\left(\frac{\lambda^i_m}{\lambda^j_m}\right)^{\!\!\frac{n}{p}} \nabla v_j \!\left(\frac{\lambda^i_m x + y^i_m - y^j_m}{\lambda^j_m}\right) \to \ell_{ij}^{\frac{n}{p}} \,\nabla v_j(\ell_{ij} x + p_{ij}) \quad\text{strongly in } L^p(\R^n).
		\end{equation*}
		As a result,~\eqref{eq:ener-sum-I} immediately follows.
		
		Now we prove that, for each equivalence class~$I$, the function~$v_I$ is a solution of~\eqref{eq:critica-anisot}. To this end, we define
		\begin{equation*}
			\Theta_m \coloneqq u_m - \sum_{i=1}^{k} v_m^i,
		\end{equation*}
		so that \eqref{eq:struwe-plap-2} and Sobolev embedding imply
		\begin{equation}
			\label{eq:Thetam-o1}
			\int_{\R^n} \,\abs*{\nabla \Theta_m}^p \, dx = o(1) \quad\text{and}\quad \norma*{\Theta_m}_{L^{\past}\!(\R^n)} = o(1) \quad\text{as } m \to +\infty.
		\end{equation}
		For~$i \in I$ fixed as above, we set
		\begin{equation}
			\label{eq:def-um-T}
			\begin{split}
				\mathsf{u}_m(x) &\coloneqq \left(\lambda^i_m\right)^{\!\frac{n-p}{p}} u_m \!\left(\lambda^i_m x + y^i_m\right) = T_{-y^i_m/\lambda^i_m,\lambda^i_m}(u_m)(x) \\
				&= v_i(x) + \sum_{j \neq i} \left(\frac{\lambda^i_m}{\lambda^j_m}\right)^{\!\!\frac{n-p}{p}} v_j \!\left(\frac{\lambda^i_m x + y^i_m - y^j_m}{\lambda^j_m}\right) + \widehat{\Theta}_m(x),
			\end{split}
		\end{equation}
		where~$\widehat{\Theta}_m \coloneqq T_{-y^i_m/\lambda^i_m,\lambda^i_m}(\Theta)$. Thus, using~\eqref{eq:energ-cond-struwe} and recalling also point~\ref{it:cons-normap-anisot} in Lemma~\ref{lem:sym-anisot}, the sequence~$\{\mathsf{u}_m\}_m$ is bounded in~$\mathcal{D}^{1,p}(\R^n)$. Therefore, up to subsequences, we have
		\begin{equation}
		\label{eq:conv-um}
			\mathsf{u}_m  \to v \;\,\text{weakly in } \mathcal{D}^{1,p}(\R^n) \quad\text{and}\quad \mathsf{u}_m \to v \;\,\text{a.e.\ in } \R^n,\\
		\end{equation}
		for some~$v \in \mathcal{D}^{1,p}(\R^n)$. In light of the definition of~$\mathsf{u}_m$ in~\eqref{eq:def-um-T} and~\eqref{eq:hyp-sruwe-anisot}, we can apply Lemma~\ref{lem:struwe-iter} to deduce that~$\mathcal{J}'(v)=0$. As a result,~$v$ is a non-negative solution to~\eqref{eq:critica-anisot}.
		
		On the other hand, by~\eqref{eq:Thetam-o1} and point~\ref{it:cons-normapas-anisot} of Lemma~\ref{lem:sym-anisot}, we deduce that
		\begin{equation}
			\label{eq:conv-Thm}
			\widehat{\Theta}_m \to 0 \quad\text{strongly in } L^{\past}\!(\R^n) \quad\text{and}\quad \widehat{\Theta}_m \to 0 \quad\text{a.e.\ in } \R^n,
		\end{equation}
		possibly up to a subsequence. Moreover, for~$j \not \in I$, we claim that
		\begin{equation}
			\label{eq:conv-vjlam}
			\left(\frac{\lambda^i_m}{\lambda^j_m}\right)^{\!\!\frac{n-p}{p}} v_j \!\left(\frac{\lambda^i_m x + y^i_m - y^j_m}{\lambda^j_m}\right) \to 0 \quad\text{for a.e.\ } x \in \R^n.
		\end{equation}
		Exploiting~\eqref{eq:conv-um}--\eqref{eq:conv-vjlam} and taking the limit for~$m \to +\infty$ in~\eqref{eq:def-um-T}, we deduce that~$v=v_I$, where $v_I$ was given by \eqref{def:vI}. Since~$v$ is a  non-negative solution to~\eqref{eq:critica-anisot}, so is~$v_I$.
		
		To justify~\eqref{eq:conv-vjlam}, observe that if the first or third case in~\eqref{eq:lij-1} occurs, the claim follows from the fact that~$v_i \in \mathcal{D}^{1,p}(\R^n) \cap L^\infty(\R^n)$ and vanishes at infinity, since~$v_i$ is a~$(p,H)$-bubble. If the second case in~\eqref{eq:lij-1} occurs, the claim follows using the explicit expression of~$v_i$ given in~\eqref{eq:pH-bubb}.
		
		Since, for every equivalence class,~$v_I$ is a non-negative solution to~\eqref{eq:critica-anisot}, Lemma~\ref{lem:sum-bub} implies that
		\begin{equation}
			\label{eq:energ-tochain-1}
			\int_{\R^n} \,\abs*{\nabla v_I}^p \, dx \leq \sum_{j \in I} \int_{\R^n} \,\abs*{\nabla v_j}^p \, dx
		\end{equation}
		with equality if and only if~$I$ consists of a single element, since each~$v_j$ is non-trivial.
		
		Now we claim that
		\begin{equation}
		\label{eq:ener-est-p-toprove}
			\norma*{\nabla u_m}_{L^{p}(\R^n)}^p = 
			\norma*{\sum_{i=1}^{k} \nabla v^i_m}_{L^{p}(\R^n)}^p + o(1).
		\end{equation}
		By the mean value theorem, for~$t,s>0$ there holds~$|t^p-s^p|\leq p\,\max\{t,s\}^{p-1}|t-s|$. We thus have that
		\begin{align*}
			&\abs*{\norma*{\nabla u_m}_{L^{p}(\R^n)}^p - \norma*{\sum_{i=1}^{k} \nabla v^i_m}_{L^{p}(\R^n)}^p} \\
			&\quad\leq p\, \max\left\{\norma*{\nabla u_m}_{L^{p}(\R^n)}, 
			\norma*{\sum_{i=1}^{k} \nabla v^i_m}_{L^{p}(\R^n)}\right\}^{\!p-1} \abs*{\norma*{\nabla u_m}_{L^{p}(\R^n)} - \norma*{\sum_{i=1}^{k} \nabla v^i_m}_{L^{p}(\R^n)}}.
		\end{align*}
		On the other hand, using~\eqref{eq:equiv-norma}, the the energy bounds~\eqref{eq:energ-bubb-anisot} and~\eqref{eq:energ-cond-struwe}, the definition~\eqref{eq:def-vjm}, and point~\ref{it:cons-normap-anisot} of Lemma~\ref{lem:sym-anisot}, we deduce that
		\begin{align*}
			\abs*{\norma*{\nabla u_m}_{L^{p}(\R^n)}^p - \norma*{\sum_{i=1}^{k} \nabla v^i_m}_{L^{p}(\R^n)}^p} \leq C \,\abs*{\norma*{\nabla u_m}_{L^{p}(\R^n)} - \norma*{\sum_{i=1}^{k} \nabla v^i_m}_{L^{p}(\R^n)}},
		\end{align*}
		for some~$C>0$ independent of~$m$. Finally,~\eqref{eq:struwe-plap-2} implies that
		\begin{equation*}
			\abs*{\norma*{\nabla u_m}_{L^{p}(\R^n)} - \norma*{\sum_{i=1}^{k} \nabla v^i_m}_{L^{p}(\R^n)}} \to 0,
		\end{equation*}
		from which~\eqref{eq:ener-est-p-toprove} follows. Thus, Lemma~\ref{lem:xi-p} and~\eqref{eq:ij-non-equiv} yield
		\begin{equation*}
			\int_{\R^n} \,\abs*{\nabla u_m}^p \, dx = \sum_{q=1}^{h} \int_{\R^n} \,\abs*{\sum_{j \in I_q} \nabla v^j_m}^p \, dx +o(1) ,
		\end{equation*}
		from which, thanks to~\eqref{eq:ener-sum-I}, we infer
		\begin{equation}
			\label{eq:energ-tochain-2}
			\int_{\R^n} \,\abs*{\nabla u_m}^p \, dx = \sum_{q=1}^{h} \int_{\R^n} \,\abs*{\nabla v_{I_q}}^p \, dx +o(1).
		\end{equation}
		In addition, from~\eqref{eq:struwe-plap-2}, we also read off that
		\begin{equation}
			\label{eq:energ-tochain-3}
			\int_{\R^n} \,\abs*{\nabla u_m}^p \, dx = \sum_{j=1}^{k} \int_{\R^n} \,\abs*{\nabla v_j}^p \, dx +o(1).
		\end{equation}
		Combining~\eqref{eq:energ-tochain-1},~\eqref{eq:energ-tochain-2}, and~\eqref{eq:energ-tochain-3}, we deduce that
		\begin{align*}
			 \sum_{q=1}^{h} \int_{\R^n} \,\abs*{\nabla v_{I_q}}^p \, dx &\leq \sum_{q=1}^{h} \sum_{j \in I_q} \int_{\R^n} \,\abs*{\nabla v_j}^p \, dx = \sum_{j=1}^{k} \int_{\R^n} \,\abs*{\nabla v_j}^p \, dx \\
			 &= \int_{\R^n} \,\abs*{\nabla u_m}^p \, dx + o(1) = \sum_{q=1}^{h} \int_{\R^n} \,\abs*{\nabla v_{I_q}}^p \, dx.
		\end{align*}
		Therefore, equality holds in~\eqref{eq:energ-tochain-1} for each equivalence class, thereby proving the theorem.
	\end{proof}

	
	\section{Integral inequalities involving the $P$-function}
	\label{sec:ineq-Pfunct}
	
	
	\subsection{Introduction of the~$P$-function and of the relevant vector fields.}
	
	Since~$\kappa \in L^\infty(\R^n)$, it is well known, by the results of Peral~\cite{peral-ictp}, Serrin~\cite{serr}, DiBenedetto~\cite{diben}, and Tolksdorf~\cite{tolk} -- see also~\cite{caa-reg} for the local regularity result --, that any solution~$u \in \mathcal{D}^{1,p}(\R^n)$ of~\eqref{eq:maineq-bubb-anisot} actually satisfies
    \begin{equation}
    \label{eq:u-C1alpha}
   		u \in W^{1,\infty}(\R^n) \cap C^{1,\alpha}_{\loc}(\R^n) \quad\text{for some } \alpha \in (0,1).
    \end{equation}
	Moreover, by an adaptation of V\'etois' Lemma~2.2 in~\cite{vet}, we also have that~$u \in L^{p_\ast-1,\infty}(\R^n)$, where~$p_\ast\coloneqq\frac{p(n-1)}{n-p}$. Hence, by interpolation,
	\begin{equation*}
		\label{eq:integrabilita}
		u \in L^q(\R^n) \quad\text{for every } q \in \left(p_\ast-1,+\infty\right]\!.
	\end{equation*}
	Furthermore, an adaptation of Theorem~1.1 in~\cite{vet} provides a priori estimates for~$u$ and its gradient, namely 
	\begin{equation}
		\label{eq:est-vetois}
		u(x) \leq \frac{C_u}{1+\abs*{x}^\frac{n-p}{p-1}} \quad \text{and} \quad \abs*{\nabla u(x)} \leq \frac{C_u}{1+\abs*{x}^\frac{n-1}{p-1}} \quad \text{for every } x \in \R^n
	\end{equation}
	for some constant~$C_u>0$ depending on~$u$. See also Proposition~2.3 in~\cite{cfr}.
	
	Concerning the regularity of the solution~$u$ to~\eqref{eq:maineq-bubb-anisot}, further results have been established by Antonini, Ciraolo \& Farina~\cite{carlos} -- see also the references therein. Let us define the critical set of~$u$ as
	\begin{equation*}
		\mathcal{Z}_u \coloneqq \left\{x \in \R^n \,\lvert\, \nabla u(x)=0 \right\}\!,
	\end{equation*}
	and the \textit{stress field} associated with~$u$ by
	\begin{equation*}
		\stressu=\frac{1}{p}\nabla_\xi H^p(\nabla u) \,,
	\end{equation*} 
	which is extended to zero on~$\mathcal{Z}_u$. Then, the set~$\mathcal{Z}_u$ is negligible, i.e.,~$\abs*{\mathcal{Z}_u}=0$, and
	\begin{gather}
		\notag
		\stressu \in W^{1,2}_{\loc}(\R^n)\cap C^{0,\alpha}_{\loc}(\R^n), \\
		\label{eq:regu-anisot}
		u \in W^{2,2}_{\loc}(\R^n \setminus \mathcal{Z}_u) \quad \text{and} \quad \abs*{\nabla u}^{p-2} \,\nabla^2 u \in L^{2}_{\loc}(\R^n \setminus \mathcal{Z}_u) 
	\end{gather}

	We now claim that
    \begin{equation}
    \label{reg:c2beta}
        u\in C^{3,\beta}_{\loc}(\R^n\setminus \mathcal{Z}_u).
    \end{equation}
	Indeed, by the chain rule, in the open set~$\R^n\setminus \mathcal{Z}_u$,~\eqref{eq:critica-anisot} can be written as
    \begin{equation}
    \label{eq:trace-equ}
        \mathrm{tr}\!\left(\nabla^2_\xi H^p(\nabla u)\,D^2 u \right)+\kappa(x)u^{\past-1}=0\quad\text{in } \R^n\setminus \mathcal{Z}_u.
    \end{equation}
	Now, let~$x_0\in \Omega \setminus \mathcal Z_u$. By continuity, one can find a radius~$\varrho_{x_0}>0$ and a couple of constants~$c_{x_0},C_{x_0}>0$ such that~$B_{\varrho_{x_0}}(x_0)\subseteq \Omega \setminus \mathcal Z_u$, and~$c_{x_0}\leq H(\nabla u)\leq C_{x_0}$ in~$B_{\varrho_{x_0}}(x_0)$. Combining this piece of information with~\eqref{eq:nabla2Hp:twosides}, we have that \eqref{eq:trace-equ} is a uniformly elliptic equation in $B_{\varrho}(x_0)$. Given the regularity of~$H\in C^{3,\beta}(\R^n\setminus \{0\})$ and of~$u$ in~\eqref{eq:u-C1alpha}, standard regularity theory -- see~\cite[Theorem~9.19]{gt} -- yields that~$u\in C^{3,\beta}(B_{\varrho_{x_0}}(x_0))$. As~$x_0\in \Omega\setminus \mathcal{Z}_u$ was arbitrary, a standard covering argument implies~\eqref{reg:c2beta}. \newline
  
	Now, as~$u>0$, we introduce the auxiliary function~$v$ given by
	\begin{equation}
		\label{eq:defv-anisot}
		v \coloneqq u^{-\frac{p}{n-p}},
	\end{equation}
	and the so-called $P$-function as
	\begin{equation}
		\label{eq:defPfunct-anisot}
		P \coloneqq n \,\frac{p-1}{p} v^{-1} \, H^p \!\left(\nabla v\right) + \left(\frac{p}{n-p}\right)^{\! p-1} v^{-1}.
	\end{equation}
	Note that, by~\eqref{eq:defV-anisot}, we can rewrite
	\begin{equation*}
		P = \frac{n \left(p-1\right)}{v} \, V\!\left(\nabla v\right) + \left(\frac{p}{n-p}\right)^{\! p-1} \frac{1}{v}.
	\end{equation*}
	From this definition and~\eqref{eq:maineq-bubb-anisot}, it follows that~$v$ weakly solves
	\begin{equation}
		\label{eq:eqforv-anisot}
		\Delta_p v = P + R \quad \text{in } \R^n,
	\end{equation}
	where~$P$ is given by~\eqref{eq:defPfunct-anisot}, and
	\begin{equation}
		\label{eq:defRem-anisot}
		R \coloneqq \left(\frac{p}{n-p}\right)^{\! p-1} \left(\kappa-1\right) v^{-1}.
	\end{equation}
	
	Furthermore,~$v$ inherits some regularity properties from those of~$u$ listed in~\eqref{eq:u-C1alpha} and in~\eqref{eq:regu-anisot}--\eqref{reg:c2beta}, more precisely
	\begin{gather}
		\notag
		\mathcal{Z}_v = \mathcal{Z}_u, \\
		\label{eq:regv-anisot}
		v \in C^{1,\alpha}_{\loc}(\R^n)\cap C^{3,\beta}_{\loc}(\R^n\setminus \mathcal{Z}_v) \\
		\notag
        \stressv \in W^{1,2}_{\loc}(\R^n)\cap C^{0,\alpha}_{\loc}(\R^n).
	\end{gather}
	
	Additionally, thanks to the homogeneity of~$H$ and~\eqref{H0:nablH}, we have
	\begin{equation*}
		V(\nabla v)=H_0^{\frac{p}{p-1}}(a(\nabla v)),
	\end{equation*}
	so  the chain rule,~\eqref{eq:defPfunct-anisot}, and~\eqref{eq:regv-anisot} yield
	\begin{gather}
		\notag
		V \!\left(\nabla v\right) \in W^{1,2}_{\loc}(\R^n) \cap C^{0,\alpha}_{\loc}(\R^n), \\
		\label{eq:regP-anisot}
		P \in W^{1,2}_{\loc}(\R^n) \cap C^{0,\alpha}_{\loc}(\R^n).
	\end{gather}
	
	Finally, we introduce the centrally important vector field
	\begin{equation*}
		\label{eq:def-gradA}
		W \coloneqq \nabla \stressv.
	\end{equation*}
	We also define its traceless version as
	\begin{equation}
	\label{eq:tracless-W-anisot}
		\mathring{W} = \mathring{\nabla} \stressv \coloneqq \nabla \stressv - \frac{\tr \nabla \stressv}{n} \Id,
	\end{equation}
	where~$\Id$ is the~$n$-dimensional identity matrix. From~\eqref{eq:regv-anisot}, we deduce that
	\begin{equation*}
		W \in L^2_{\loc}(\R^n) \cap C^{1,\beta}_{\loc} \!\left(\R^n \setminus \mathcal{Z}_v\right)\!,
	\end{equation*}
	and by the chain rule and~\eqref{eq:regv-anisot}, we can also write
	\begin{equation*}
		\label{eq:defW-anisot}
		W(x)=
		\begin{dcases}
			A(x) D^2v(x)	& \quad \text{if } x \in \R^n \setminus \mathcal{Z}_v, \\
			0				& \quad \text{if } x \in \mathcal{Z}_v,
		\end{dcases}
	\end{equation*}
	where~$A$ is the matrix with components
	\begin{equation}
	\label{eq:def-A-anisot}
		\left[A(x)\right]_{ij} = \alpha_{ij}(x) \coloneqq \partial_{\xi_i\xi_j} V \!\left(\nabla v\right)\!(x) \quad\text{for } x \in \R^n \setminus \mathcal{Z}_v.
	\end{equation}
	In particular,~$A$ is well defined and of class~$C^{1}$ in~$\R^n \setminus \mathcal{Z}_v$, which has full measure, hence~$A$ is measurable and we may extend it identically equal to zero on~$\mathcal{Z}_v$. Moreover, by Lemma~3.2 in~\cite{carlos+big4} -- see also Proposition~3.1 in~\cite{cozzi-monot} -- and~\eqref{eq:equiv-norma}, we have
	\begin{equation}
	\label{eq:A-boundls}
    \left\langle A \eta, \eta\right\rangle \geq c \,\abs*{\nabla v}^{p-2}\,\abs*{\eta}^{2} \quad\text{and}\quad |A|\leq C \, \abs*{\nabla v}^{p-2} \quad\text{on } \R^n\setminus \mathcal{Z}_v,
	\end{equation}
	for all $\eta\in \R^n$ and some~$c,C>0$.
    
	With this notation at hand, we can rewrite~\eqref{eq:eqforv-anisot} as
	\begin{equation}
	\label{eq:trW-anisot}
		\tr W =\diver\left(a(\nabla v)\right)= \alpha_{ij} v_{ij} = P + R \quad \text{a.e.\ in } \R^n.
	\end{equation}
	
	We conclude by noticing that a direct computation from~\eqref{eq:defPfunct-anisot} reveals
	\begin{equation}
	\label{eq:grad-Puso}
    	\nabla P=\frac{n(p-1)}{v} \,\nabla V(\nabla v)-\frac{P}{v}\,\nabla v \quad \text{on } \R^n \setminus \mathcal{Z}_v,
	\end{equation}
	and we also have
	\begin{equation}
	\label{eq:gradP-anisot}
		\nabla P = \frac{n}{v} \left(W^\mathsf{T} - \frac{P}{n} \Id\right) \nabla v = \frac{n}{v}\left(\mathring{W}^\mathsf{T}+\frac{R}{n} \Id\right)\,\nabla v \quad \text{on } \R^n \setminus \mathcal{Z}_v,
	\end{equation}
    where~$W^\mathsf{T}$ denotes the transpose matrix of~$W$. Indeed, by  the chain rule and the homogeneity of~$H$, we deduce
	\begin{equation*}
    \begin{split}
        (p-1)\,\partial_{x_j}V(\nabla v)&=\frac{(p-1)}{p}\,\partial_{\xi_k} H^p(\nabla v)\,v_{jk}=\frac{1}{p}\,\partial_{\xi_k\xi_l}H^p(\nabla v)\,v_l v_{jk}
        \\
        &=\frac{1}{p} \,\partial_{x_j}\!\left(\partial_{\xi_l} H^p(\nabla v)\right) v_l=\left(\nabla a(\nabla v)^\mathsf{T}\,\nabla v\right)_{\! j} = \left(W^\mathsf{T}\nabla v\right)_{\! j}.
    \end{split}
	\end{equation*}
    This identity, coupled with~\eqref{eq:grad-Puso}, entails the first identity of~\eqref{eq:gradP-anisot}. The second one follows from~\eqref{eq:trW-anisot} and the definition of~$\mathring{W}$ in~\eqref{eq:tracless-W-anisot}.
    
	
	\subsection{Inequalities involving the $P$-function.}
	\label{subsec:ineqP}
	
	We start by recalling a differential identity, established in Proposition~3.1 of~\cite{cg-plap}, which holds for sufficiently smooth functions and will be useful later on.
	
	\begin{proposition}
		\label{prop:id-diff}
		Let~$\Omega\subseteq \R^n$ be an open set and~$w\in C^3(\Omega)$ be a positive function. Let~$\mathscr{V}:\R^n \to [0,+\infty)$ be a function of class $ C^3(\R^n\setminus\{0\})$, and define
		\begin{equation*}
			\mathsf{P} \coloneqq \frac{n \left(p-1\right)}{w} \, \mathscr{V}\!\left(\nabla w\right) + \left(\frac{p}{n-p}\right)^{\! p-1} \frac{1}{w}.
		\end{equation*}
		Then, by setting
		\begin{equation*}
			\mathsf{a}\!\left(\nabla w\right)\coloneqq \nabla_\xi \mathscr{V}\!\left(\nabla w\right)\!,\quad \left[\mathsf{A}\right]_{ij} =\tilde{\alpha}_{ij}\coloneqq\partial_{\xi_i\xi_j}\mathscr{V}\!\left(\nabla w\right)\!, \quad\text{and}\quad \mathscr{W} \coloneqq \nabla  \mathsf{a}\!\left(\nabla w\right)\!,
		\end{equation*}
		the following differential identity holds in~$\Omega\setminus\{x \in \Omega \,\lvert\, \nabla w(x)=0\}$
		\begin{multline}
			\label{eq:id-diff-anisot}
			\diver\!\left(w^{2-n}\mathsf{A}\nabla \mathsf{P} \right)=w^{1-n}\Bigl\{-n \left\langle \mathsf{A}\,\nabla\mathsf{P}, \nabla w\right\rangle-\mathsf{P} \tr\mathscr{W}
			\\
			+n\left(p-1\right)\left[\tr\mathscr{W}^2+ \left\langle \nabla\!\left(\tr\mathscr{W}\right), \mathsf{a}\!\left(\nabla w\right) \right\rangle \right]-\mathsf{P} w_j\partial_{\xi_i\xi_j\xi_k}\mathscr{V}\!\left(\nabla w\right) w_{ki}\Bigr\}.
		\end{multline}
	\end{proposition}
	
	Our next goal is to establish a fundamental inequality that will play a key role in proving the stability result. As a preliminary step, we first state a lemma on the regularity of the distributional divergence of a vector field.

	\begin{lemma}
	\label{lem:vect-field}
    Let~$Z\subseteq \R^n$ be such that~$\abs*{Z}=0$, and let~$X:\R^n\to \R^n$ be a vector field satisfying~$X\in L^2_{\loc}(\R^n)\cap C^1(\R^n\setminus Z)$, and 
    \begin{equation}
    \label{eq:uguale-g}
        \diver X = g \quad \text{on } \R^n\setminus Z,
    \end{equation}
    for some function~$g\in L^2_{\loc}(\R^n)$. Then,~$\diver X=g$ in the distributional sense and the divergence theorem holds, i.e.,
	\begin{equation}
	\label{eq:divergenza}
    	\int_{\R^n}\left(\diver X\right) \phi\,dx=-\int_{\R^n} \left\langle X , \nabla \phi \right\rangle dx,
	\end{equation}
	for all functions~$\phi\in W^{1,2}(\R^n)$ with compact support in~$\R^n$.
	\end{lemma}

	\begin{proof}
    The proof readily follows via a standard density argument. We report the full argument for completeness. Let~$\{\varphi_\delta\}_{\delta}$ be the family of standard, radially symmetric convolution kernels, and let~$X_\delta \coloneqq X\ast \varphi_\delta$,~$g_\delta \coloneqq g\ast\varphi_\delta$ be the convolution of~$X$ and~$g$, respectively. In particular, by  the properties of convolutions and~\eqref{eq:uguale-g},~$X_\delta$,~$g_\delta$ are smooth functions,~$X_\delta\to X$ and~$g_\delta\to g$ in~$L^2_{\loc}(\R^n)$, and we have~$\diver X_\delta=g_\delta$ in~$\R^n$. It follows that, for a given function~$\phi\in C^\infty_c(\R^n)$, the divergence theorem and the hypothesis on~$Z$ yield
    \begin{equation*}
        \int_{\R^n} g_\delta\,\phi\,dx=\int_{\R^n\setminus Z} \left(\mathrm{div}\,X_\delta\right)\phi\,dx=\int_{\R^n} \left(\mathrm{div}\,X_\delta\right)\phi\,dx=-\int_{\R^n} \left\langle X_\delta , \nabla \phi \right\rangle dx,
    \end{equation*}
    so, letting $\delta\to 0$, we find
	\begin{equation*}
    	\int_{\R^n} g\,\phi\,dx=-\int_{\R^n}  \left\langle X , \nabla \phi \right\rangle dx,
	\end{equation*}
 	for all~$\phi\in C^\infty_c(\R^n)$, that is~$\diver X=g$ in the distributional sense. Finally,~\eqref{eq:divergenza} follows from the above identity and a standard density argument on~$\phi$.
\end{proof}

	We are now ready to prove the desired integral identity. The key idea is to show that the vector field~$A\nabla P$ fulfills~$\diver (A\nabla P)\in L^2_{\loc}(\R^n)$, and then the desired identity will follow by combining Lemma~\ref{lem:vect-field} with Proposition~\ref{prop:id-diff}.

	\begin{proposition}
	\label{prop:fund-ident-anisot-tbp}
		Let~$v$,~$P$,~$R$, and~$A$ be given by~\eqref{eq:defv-anisot},~\eqref{eq:defPfunct-anisot},~\eqref{eq:defRem-anisot}, and~\eqref{eq:def-A-anisot}, respectively.
		Then, it follows that
        \begin{equation}
        \label{eq:L2-divergence}
            A\nabla P\in L^2_{\loc}(\R^n)\quad \text{and}\quad \diver \!\left(A\nabla P \right) \in L^2_{\loc}(\R^n),
        \end{equation}
        where the divergence is interpreted in the distributional sense.
        Moreover, for any~$\phi\in C^\infty_{c}(B_r)$ and any~$t \in \R$, the following identity holds
		\begin{align}
		\notag
			-\int_{B_r} &v^{2-n} \left\langle A\nabla P, \nabla (P^t \phi) \right\rangle dx = \int_{B_r}\diver\!\left(v^{2-n}A\nabla P \right) P^t \phi \, dx \\
		\label{eq:fund-ineq-anisot}
			&=\int_{B_r}v^{1-n} \,\Bigl\{-n \left\langle A\nabla P, \nabla v \right\rangle -(p-1)\,P\, \tr W \\
			\notag
			&\hspace{0.89cm}+ n \left(p-1\right) \left[\tr[W]^2+ \left\langle \nabla (P+R),  a(\nabla v) \right\rangle \right]\Bigr\}\, P^t \phi \, dx.
		\end{align}
	\end{proposition}

    \begin{proof}
        We start by proving~\eqref{eq:L2-divergence}. According to~\eqref{eq:grad-Puso} and the regularity properties in~\eqref{eq:regv-anisot} and~\eqref{eq:regP-anisot}, it suffices to prove that~$A\nabla v$,~$A \nabla V(\nabla v)$, and their divergences are equal to an~$L^2$-function on the set~$\R^n\setminus \mathcal{Z}_v$.
		We remark that, on such a set, all subsequent computations hold pointwise, by virtue of the regularity of~$H$ and of~$v$ in~\eqref{eq:regv-anisot}.
        
        By the homogeneity properties of~$H$, for all~$x\in \R^n\setminus \mathcal{Z}_v$, we have
        \begin{equation}
        \label{eq:uso-1}
            A\nabla v=\frac{1}{p} \,\nabla^2_\xi H^p(\nabla v)\,\nabla v=(p-1)\,\frac{1}{p} \,\nabla_\xi H^p(\nabla v)=(p-1)\, \stressv,
        \end{equation}
        and, for~$i=1,\dots,n$, we have
        \begin{equation}
        \label{eq:uso-2}
          \begin{split} 
          	\left(\nabla  V(\nabla v)\right)_{i} &=\frac{1}{p^2} \,\partial_{\xi_i\xi_j} H^p(\nabla v)\,\partial_{x_j} H^p(\nabla v) =\frac{1}{p^2} \, \partial_{\xi_i\xi_j} H^p(\nabla v)\,\partial_{\xi_k}H^p(\nabla v)\,v_{jk}	\\
          	&=\frac{1}{p^2}\,\partial_{x_k}\!\left(\partial_{\xi_i}H^p(\nabla v) \right) \partial_{\xi_k} H^p(\nabla v) =\left(W \stressv\right)_{i}.
          \end{split}
        \end{equation}
		Recalling that~$\abs*{\mathcal{Z}_v}=0$,~$W\in L^2_{\loc}(\R^n)$, and~$a(\nabla v)\in L^\infty_{\loc}(\R^n)$ due to~\eqref{eq:regv-anisot}, we thus have that~$A\nabla v$ and~$A\nabla V(\nabla v)$ are in~$L^2_{\loc}(\R^n)$, and so is~$A\nabla P$ by~\eqref{eq:grad-Puso} and the previous discussion.
		
		Let us now study their pointwise divergence on~$\R^n\setminus \mathcal{Z}_v$. Thanks to~\eqref{eq:uso-1} and the equation~\eqref{eq:eqforv-anisot}
		\begin{equation*}
	    	\diver\!\left(A\nabla v \right)=(p-1) \diver\!\left(\stressv\right)=(p-1)\left(P+R\right)
		\end{equation*}
    	on~$\R^n\setminus \mathcal{Z}_v$. Owing to~\eqref{eq:defRem-anisot},~\eqref{eq:regv-anisot}, and the regularity of~$\kappa$, the right hand-side in the above equation belongs to~$L^\infty_{\loc}(\R^n)$. Hence, from Lemma~\ref{lem:vect-field} we deduce that~$\diver\!\left(A\nabla v\right)\in L^2_{\loc}(\R^n)$ in the distributional sense as well.
    
		Next, on~$\R^n\setminus \mathcal{Z}_v$, for all~$k=1,\dots,n$, by Schwarz's theorem and equation~\eqref{eq:eqforv-anisot}, we have
		\begin{equation}
		\label{eq:uso-3}
    		\begin{split}
			\partial_{x_i} W_{ik} &= \frac{1}{p} \, \partial_{x_i}\!\left(\partial_{x_k}\left(\partial_{\xi_i} H^p(\nabla v) \right) \right) = \frac{1}{p} \, \partial_{x_k} \!\left(\partial_{x_i}\!\left(\partial_{\xi_i} H^p(\nabla v) \right) \right) \\
			&=\partial_{x_k}\diver\!\left(a(\nabla v)\right)=\partial_{x_k}\!\left( P+R\right) \in L^2_{\loc}(\R^n)
    	\end{split}
		\end{equation}
    	the latter inclusion stemming from~\eqref{eq:regv-anisot},~\eqref{eq:regP-anisot}, and the regularity of~$R$. As~$\stressv \in W^{1,2}_{\loc}(\R^n)\cap L^\infty_{\loc}(\R^n)$ thanks to~\eqref{eq:regv-anisot}, from \eqref{eq:uso-2}, \eqref{eq:uso-3} and, Lemma~\ref{lem:vect-field}, we deduce that~$\diver\!\left(A\nabla V\right)\in L^2_{\loc}(\R^n)$ in the distributional sense, and this establishes~\eqref{eq:L2-divergence} by the discussion at the beginning of the proof. \newline
    	
		Next, by the homogeneity of~$H$ and the chain rule for derivatives, we have
		\begin{equation}
		\label{eq:uso-4}
    		\begin{split}
        	\partial_{\xi_i\xi_j\xi_k} V(\nabla v)\,v_j v_{ki}=(p-2)\,\partial_{\xi_i\xi_k} V(\nabla v)\,v_{ki} = (p-2) \diver\!\left(\stressv\right) \!.
    		\end{split}
		\end{equation}
		Therefore, coupling~\eqref{eq:id-diff-anisot} with~\eqref{eq:trW-anisot} and~\eqref{eq:uso-4}, on~$\R^n\setminus \mathcal{Z}_v$, we have
		\begin{equation*}
			\begin{split}
    			\diver\!\left(v^{2-n}A\nabla P \right) &= v^{1-n} \,\Bigl\{-n \left\langle A\nabla P, \nabla v \right\rangle -(p-1)\,P\, \tr W \\
				&\quad+ n \left(p-1\right) \left[\tr[W]^2+ \left\langle \nabla (P+R),  a(\nabla v) \right\rangle \right]\Bigr\}
    		\end{split}
		\end{equation*}
    	Multiplying the above identity by~$P^t \phi$, recalling that $\abs*{\mathcal{Z}_v}=0$, and using the distributional divergence theorem finally yield~\eqref{eq:fund-ineq-anisot}.
    \end{proof}
    
	By means of the homogeneity of~$H$, and a few algebraic manipulations, we now deduce the following fundamental inequality.
    
	\begin{proposition}
	\label{prop:fund-ineq-anisot}
		Let~$v$ be a weak solution of~\eqref{eq:eqforv-anisot}. Then for every non-negative~$\phi \in C^\infty_c(\R^n)$ such that~$\mathrm{supp}\,\phi \subseteq \overline{B_{r}}$ and any~$t \in \R$, we have
		\begin{equation}
		\label{eq:fund-ineq-anisot-2}
			\begin{split}
				-\int_{B_{r}} &v^{2-n} \left\langle A\nabla P, \nabla (P^t \phi) \right\rangle dx \geq n \left(p-1\right) \left(1-c_{p,H}\right) \int_{B_{r}} v^{1-n} \,\abs{\mathring{W}}^2 \, P^t \phi \, dx \\
				&+ \left(p-1\right) \int_{B_{r}} v^{1-n} \,\bigl\{n \left\langle \nabla R, \stressv \right\rangle + R \left(P+R\right) \bigr\}\, P^t \phi \, dx,
			\end{split}
		\end{equation}
		where~$c_{p,H} \in [0,1)$ is defined as
		\begin{equation}
		\label{eq:cph}
			c_{p,H} \coloneqq \frac{\left(1-\varrho_{p,H}\right)^2}{1+\varrho_{p,H}^2} \quad\text{with}\quad \varrho_{p,H} \coloneqq \frac{\lambda_H}{\Lambda_H} \left(p-1\right)^{\mathrm{sgn}(2-p)}\!.
		\end{equation}
	\end{proposition}
	\begin{proof}
		We first observe that, for~$A$ given by~\eqref{eq:def-A-anisot}, an application of Lemma~3.2 in~\cite{carlos+big4} yields
		\begin{equation*}
    		\frac{\lambda_{\mathrm{min}}(A(x))}{\lambda_{\mathrm{max}}(A(x))} \geq \frac{\lambda_H}{\Lambda_H} \frac{\min\left\{1,p-1\right\}}{\max\left\{1,p-1\right\}}\eqqcolon \varrho_{p,H} \quad \text{on } \R^n\setminus \mathcal{Z}_v.
		\end{equation*}
		Therefore, applying Lemma~3.2 in~\cite{cg-plap} to~$W = A D^2v$, we obtain
    	\begin{equation}
		\label{eq:trW2-anisot}
			\tr[W]^2 \geq \abs*{W}^2 - c_{p,H} \,\abs{\mathring{W}}^2.
		\end{equation}
    where~$c_{p,H}$ is defined by~\eqref{eq:cph}. Now, using~\eqref{eq:tracless-W-anisot} and \eqref{eq:trW-anisot}, we notice that
		\begin{equation}
		\label{eq:norma-Wep-senzatr-anisot}
			\abs{\mathring{W}}^2 = \abs{W}^2 - \frac{1}{n} \left(\tr W\right)^2 = \abs{W}^2 - \frac{1}{n} \left(P+R\right)^2 \quad \text{on } \R^n\setminus \mathcal{Z}_v
		\end{equation}
		so combining~\eqref{eq:trW2-anisot} and~\eqref{eq:norma-Wep-senzatr-anisot}, we get
		\begin{equation}
			\label{eq:trW2-final-anisot}
			\tr[W]^2 \geq \left(1-c_{p,H}\right) \abs {\mathring{W}}^2 + \frac{1}{n} \left(P+R\right)^2 \quad \text{on $\R^n\setminus \mathcal{Z}_v$.}
		\end{equation}
		Then, by the homogeneity of~$H$, and recalling the definition of~$A$ in~\eqref{eq:def-A-anisot}, we deduce
    	\begin{equation}
    	\label{eq:uso-5}
    		\begin{split}
        		\left\langle A\nabla P,\nabla v\right\rangle &= \frac{1}{p} \, \partial_{\xi_i\xi_j} H^p(\nabla v)\,\partial_{x_j} P \,\partial_{x_i}v = \left(p-1\right)\frac{1}{p}\,\partial_{\xi_j} H^p(\nabla v)\,\partial_{x_j} P \\
        		&=(p-1)\left\langle \nabla P,a(\nabla v)\right\rangle\!.
        	\end{split}
    	\end{equation}
		Substituting~\eqref{eq:trW-anisot},~\eqref{eq:trW2-final-anisot}, and~\eqref{eq:uso-5} into~\eqref{eq:fund-ineq-anisot},  we obtain
		\begin{align*}
			-&\int_{B_{r}} v^{2-n} \left\langle A\nabla P, \nabla (P^t \phi) \right\rangle dx = \int_{B_r} v^{1-n} \,\Bigl\{-(p-1)\,P\,(P+R)+n(p-1)\,\mathrm{tr}[W]^2 \\
			&\hspace{50pt}\quad +n(p-1)\left\langle \nabla R,a(\nabla v)\right\rangle\Bigr\} P^t\phi\,dx
			\\
			&\hspace{50pt}=\int_{B_r} v^{1-n} \,\Bigl\{-(p-1)\,(P+R)^2+(p-1)\,R(R+P) +n(p-1)\,\mathrm{tr}[W]^2
			\\
			&\hspace{50pt}\quad
			+n(p-1)\left\langle \nabla R,a(\nabla v)\right\rangle\Bigr\} P^t\phi \, dx
			\\
			&\hspace{50pt}\geq\int_{B_r} v^{1-n} \,\Bigl\{n(p-1)(1-c_{p,H})\,|\mathring{W}|^2+(p-1)\,R(R+P)
			\\
			&\hspace{50pt}\quad +n(p-1)\left\langle \nabla R,a(\nabla v)\right\rangle\Bigr\} P^t \phi \, dx,
		\end{align*}
		which is our thesis.
	\end{proof}
	
	
	\section{Anisotropic classification result}
	\label{sec:anis-classif-energ}
	
	In this brief section, we exploit the integral estimate of Proposition~\ref{prop:fund-ineq-anisot} to prove the classification result of Theorem~\ref{th:calssif-anisot}. Although the result has already been shown in~\cite{cfr}, we decided to include a very short proof here to emphasize the strength of the~$P$-function approach.
	
	In the following calculations,~$C>0$ will denote a constant independent of~$r$, which may vary from line to line.
	
	As~$u>0$, we consider the function~$v$ given by~\eqref{eq:defv-anisot}, which is a weak solution of~\eqref{eq:eqforv-anisot}. We first observe that
	\begin{equation}
		\label{eq:D1p-anisot}
		u \in \mathcal{D}^{1,p}(\R^n) \quad\text{if and only if}\quad \int_{\R^n} v^{-n} \, dx +  \int_{\R^n} v^{-n} \,\abs*{\nabla v}^p \, dx < +\infty.
	\end{equation}
	Assume that~$\kappa \equiv 1$, so that the reminder term defined in~\eqref{eq:defRem-anisot} vanishes. Moreover, we set~$A_r \coloneqq B_{2r} \setminus B_r$ and let~$\phi \in C^\infty_c(\R^n)$ be a cut-off function such that~$0 \leq \phi \leq 1$, with~$\phi=1$ in~$B_{r}$,~$\phi=0$ in~$\R^n \setminus B_{2r}$, and~$\abs*{\nabla \phi} \leq C/r$ in~$A_{r}$.
	
	Choosing~$t=1$ and replacing~$\phi$ with~$\phi^2$ in~\eqref{eq:fund-ineq-anisot-2} yields
	\begin{equation}
	\label{eq:clas-1-rew}
		-\int_{B_{2r}} v^{2-n} \left\langle A\nabla P, \nabla \!\left(P \phi^2\right) \right\rangle dx \geq \mathsf{c}_\sharp \int_{B_{2r}} v^{1-n} \,\abs{\mathring{W}}^2 \, P \phi^2 \, dx,
	\end{equation}
	with~$\mathsf{c}_\sharp \coloneqq n \left(p-1\right) \left(1-c_{p,H}\right)$. In addition, we observe that by~\eqref{eq:A-boundls} and since~$\abs*{\mathcal{Z}_v}=0$, we have
	\begin{equation*}
		\left\langle A\nabla P, \nabla P \right\rangle \geq 0\quad \text{a.e. in } \R^n,
	\end{equation*}
	and therefore~\eqref{eq:clas-1-rew} simplifies to
	\begin{equation}
	\label{eq:clas-2-rew}
		\mathsf{c}_\sharp \int_{B_{2r}} v^{1-n} \,\abs{\mathring{W}}^2 P \phi^2 \, dx \leq - 2\int_{A_{r}} v^{2-n} \left\langle A\nabla P, \nabla \phi \right\rangle  P \phi \, dx.
	\end{equation}
	Now, exploiting~\eqref{eq:A-boundls},~\eqref{eq:gradP-anisot}, and Young's inequality, we get
	\begin{equation*}
		\abs*{2 v^{2-n} \left\langle A\nabla P, \nabla \phi \right\rangle  P \phi} \leq \frac{\mathsf{c}_\sharp}{2} \, v^{1-n} \,\abs{\mathring{W}}^2 \, P \phi^2 + C v^{1-n} \,\abs*{\nabla v}^{2(p-1)} \,\abs*{\nabla\phi}^2 P.
	\end{equation*}
	Piecing this with~\eqref{eq:clas-2-rew}, we finally obtain
	\begin{equation}\label{chiamo}
		\int_{B_{r}} v^{1-n} \,\abs{\mathring{W}}^2 P \, dx \leq C \int_{A_{r}} v^{1-n} \,\abs*{\nabla v}^{2(p-1)} \,\abs*{\nabla\phi}^2 P \, dx.
	\end{equation}
	
	Taking advantage of the gradient estimate~\eqref{eq:est-vetois} and of the definition of~$v$ in~\eqref{eq:defv-anisot}, we easily see that
	\begin{equation*}
		\abs*{\nabla v(x)} \leq C \left(1+\abs*{x}^\frac{1}{p-1}\right) \quad\text{for every } x \in \R^n,
	\end{equation*}
	and therefore, using~$|\nabla \phi|\leq C/r$,~\eqref{eq:equiv-norma}, and~\eqref{eq:defPfunct-anisot}, from~\eqref{chiamo} we deduce
	\begin{equation}
	\label{eq:limit-class}
		\int_{B_{r}} v^{1-n} \,\abs{\mathring{W}}^2 P \, dx \leq C \int_{A_{r}} v^{1-n} P \, dx \leq C \int_{A_{r}} v^{-n} + v^{-n} \,\abs*{\nabla v}^p \, dx
	\end{equation}
	for~$r>0$ sufficiently large. Hence, in view of~\eqref{eq:D1p-anisot}, letting~$r \to +\infty$ in~\eqref{eq:limit-class} yields
	\begin{equation*}
		\int_{\R^n} v^{1-n} \,\abs{\mathring{W}}^2 P \, dx = 0.
	\end{equation*}
	This implies that~$\mathring{W} = 0$ a.e.\ in~$\R^n$, so that from \eqref{eq:gradP-anisot} it follows that
	\begin{equation*}
		\nabla P = 0 \quad\text{and}\quad \nabla \stressv = \frac{P}{n} \Id \quad\text{a.e.\ in } \R^n.
	\end{equation*}
    By~\eqref{eq:regP-anisot}, this yields
	\begin{equation*}
		\stressv(x) = \frac{P}{n} \left(x-z\right) \quad\text{for some } z \in \R^n,
	\end{equation*}
	where~$P(x)=P$ is a constant. By definition of~$a(\xi)$, this entails~$v(x) = c_1 + c_2 H_0^{\frac{p}{p-1}}(x-z)$. Thus, from~\eqref{eq:defv-anisot}, and since~$u$ is a solution to~\eqref{eq:critica-anisot}, we conclude that~$u$ must be of the form~\eqref{eq:pH-bubb} for some~$\lambda>0$ and~$z \in \R^n$, which is our thesis.
	
	
	\section{Anisotropic stability result}
	\label{sec:anis-stab}
	
	This section is completely devoted to the proof of Theorem~\ref{th:main-th-anisot-stab}. Of course, it suffices to prove it under the assumption that
	\begin{equation}
	\label{eq:def-small-anisot}
		\defi(u,\kappa) \leq \gamma,
	\end{equation}
	for some~$\gamma \in (0,1)$, depending only on~$n$,~$p$,~$H$, and~$\norma*{\kappa}_{L^\infty(\R^n)}$. Indeed, if~$\defi(u,\kappa) > \gamma$, then it follows that
	\begin{equation*}
		\norma*{u-\mathcal{U}_0}_{\mathcal{D}^{1,p}(\R^n)} \leq \norma*{\nabla u}_{L^{p}(\R^n)} + \norma*{\nabla \mathcal{U}_0}_{L^{p}(\R^n)} \leq 4 c_H^{-1} S_p^{\frac{n}{p}} \leq \frac{4 S_p^{\frac{n}{p}}}{c_H \,\gamma} \defi(u,\kappa)
	\end{equation*}
	for any~$(p,H)$-bubble~$\mathcal{U}_0$, where we used~\eqref{eq:energ-bubb-anisot},~\eqref{eq:equiv-norma}, and~\eqref{eq:ipotesi-energ-anisot}. Therefore, in what follows, we will assume that~\eqref{eq:def-small-anisot} is in force. \newline
	
	Following~\cite{cg-plap}, we divide the proof into several steps.
	
	\renewcommand\thesubsection{\bfseries Step \arabic{subsection}}
	
	
	\subsection{Application of the Struwe-type result and reduction.}
	\label{step:struwe-reduction}
	
	For any~$\epsilon>0$, by the Struwe-type result of Theorem~\ref{th:struwe-def} and considering~\eqref{eq:ipotesi-energ-anisot}, there exits a~$\delta>0$, depending on~$\epsilon$, and a $(p,H)$-bubble~$U_\epsilon \coloneqq U_p[z_\epsilon,\lambda_\epsilon]$, for a couple of parameters~$z_\epsilon \in \R^n$ and~$\lambda_\epsilon>0$, such that if
	\begin{equation}
		\label{eq:delta-to-fix-anisot}
		\defi(u,\kappa) \leq \delta,
	\end{equation}
	then
	\begin{equation}
		\label{eq:struwe-0-anisot}
		\norma*{\nabla \!\left(u-U_\epsilon\right)}_{\mathcal{D}^{1,p}(\R^n)} \leq \epsilon.
	\end{equation}
	We now scale and translate our functions in order to lead~$U_\epsilon$ to be the~$(p,H)$-bubble with center the origin and unit scale factor. Taking into account point~\ref{it:transf-unitbub-anisot} of Lemma~\ref{lem:sym-anisot}, we set
	\begin{equation*}
		U \coloneqq U_p[0,1]=T_{-z_\epsilon/\lambda_\epsilon,\lambda_\epsilon} \!\left(U_\epsilon\right)
	\end{equation*}
	and define accordingly
	\begin{equation*}
		u_\epsilon \coloneqq T_{-z_\epsilon/\lambda_\epsilon,\lambda_\epsilon} \!\left(u\right)\!.
	\end{equation*}
	Therefore, by point~\ref{it:cons-normap-anisot} of Lemma~\ref{lem:sym-anisot} and~\eqref{eq:struwe-0-anisot}, it follows that
	\begin{equation*}
	\label{eq:struwe-anisot}
		\norma*{\nabla \!\left(u_\epsilon-U\right)}_{\mathcal{D}^{1,p}(\R^n)} \leq \epsilon.
	\end{equation*}
	As noticed above,~$u_\epsilon$ is a weak solution to~\eqref{eq:maineq-bubb-anisot} for a different~$\kappa$ satisfying~\eqref{eq:inv-norma-anisot}--\eqref{eq:inv-k-anisot} and~$u_\epsilon$ enjoys the regularity properties listed in~\eqref{eq:regu-anisot}.
	
	Since, according to points~\ref{it:cons-normapas-anisot}--\ref{it:cons-normap-anisot} of Lemma~\ref{lem:sym-anisot} and~\eqref{eq:inv-norma-anisot}--\eqref{eq:inv-k-anisot}, all the relevant quantities are invariant by scaling, we will omit the subscript in what follows -- writing therefore~$u$ instead of~$u_\epsilon$. The subscript will be restored when needed in~\ref{step:conclusion-anisot}.
	
	
	\subsection{Decay and other preliminary estimates.}
	\label{step:decay-anisot}

	We claim that, for~$\epsilon \in (0,1)$, depending only on~$n$,~$p$,~$H$, and~$\norma*{\kappa}_{L^\infty(\R^n)}$, there exists a point~$x_0 \in \R^n$ where~$u$ attains its maximum and
	\begin{equation}
		\label{eq:bound-Linf-univ}
		\norma*{u}_{L^\infty(\R^n)} = u(x_0) \leq \mathscr{M}
	\end{equation}
	for some~$\mathscr{M} \geq 1$ depending only on~$n$,~$p$,~$H$, and~$\norma*{\kappa}_{L^\infty(\R^n)}$. Moreover, there exist two constants~$c_0,C_0>0$, depending on the same quantities, such that
	\begin{equation}
		\label{eq:bounds-u-anisot}
		\frac{c_0}{1+\abs*{x}^{\frac{n-p}{p-1}}} \leq u(x) \leq \frac{C_0}{1+\abs*{x}^\frac{n-p}{p-1}} \quad\text{for every } x \in \R^n.
	\end{equation}
	Finally, there exists a constant~$C_1 \geq 1$, depending only on~$n$,~$p$,~$H$, and~$\norma*{\kappa}_{L^\infty(\R^n)}$, such that
	\begin{equation}
		\label{eq:bound-gradu-anisot}
		\abs*{\nabla u(x)} \leq \frac{C_1}{1+\abs*{x}^\frac{n-1}{p-1}} \quad\text{for every } x \in \R^n.
	\end{equation}
	The proofs of~\eqref{eq:bound-Linf-univ}--\eqref{eq:bound-gradu-anisot} go as in the corresponding step of the proof of Theorem~1.1 in~\cite{cg-plap}, with minor modification due to presence of $H$, and is therefore omitted.
	
	We also note that, by virtue of~\eqref{eq:equiv-norma}, we can replace the Euclidean norm with~$H$ in~\eqref{eq:bounds-u-anisot}--\eqref{eq:bound-gradu-anisot}, obtaining the validity of these estimates for possibly different constants, depending only on~$n$,~$p$,~$H$, and~$\norma*{\kappa}_{L^\infty(\R^n)}$.
	
	As~$u>0$, we now consider the function~$v$ given by~\eqref{eq:defv-anisot}, which is a weak solution of~\eqref{eq:eqforv-anisot}. From the decay estimates~\eqref{eq:bounds-u-anisot}, we deduce
	\begin{equation}
		\label{eq:bounds-v-anisot}
		\hat{c}_0 \left(1+\abs*{x}^\frac{p}{p-1}\right) \leq v(x) \leq \widehat{C}_0 \left(1+\abs*{x}^\frac{p}{p-1}\right) \quad\text{for every } x \in \R^n
	\end{equation}
	for a couple of constants~$\hat{c}_0,\widehat{C}_0>0$ depending only on~$n$,~$p$,~$H$, and~$\norma*{\kappa}_{L^\infty(\R^n)}$. Moreover, by~\eqref{eq:bound-gradu-anisot} and~\eqref{eq:bounds-v-anisot}, we get
	\begin{equation}
		\label{eq:bound-gradv-anisot}
		\abs*{\nabla v(x)} \leq \widehat{C}_1 \left(1+\abs*{x}^\frac{1}{p-1}\right) \quad\text{for every } x \in \R^n,
	\end{equation}
	for some~$\widehat{C}_1 \geq 1$ depending only on~$n$,~$p$,~$H$, and~$\norma*{\kappa}_{L^\infty(\R^n)}$. Taking advantage of of~\eqref{eq:equiv-norma},~\eqref{eq:bounds-v-anisot}, and~\eqref{eq:bound-gradv-anisot}, it follows that~$v \geq \hat{c}_0$ and
	\begin{equation*}
		\label{eq:bound-fgrad}
		v^{-1} H^p(\nabla v) \in L^\infty(\R^n) \quad\text{with}\quad v^{-1} H^p(\nabla v) \leq C_H^p \, \widehat{C}_1^p \,\hat{c}_0^{-1},
	\end{equation*}
	thus
	\begin{gather}
	\label{eq:P-Linf-anisot}
		P \in L^\infty(\R^n) \quad\text{with}\quad \norma*{P}_{L^\infty(\R^n)} \leq n \,\frac{p-1}{p} \,  C_H^p \, \widehat{C}_1^p \,\hat{c}_0^{-1} + \left(\frac{p}{n-p}\right)^{\! p-1} \hat{c}_0^{-1}, \\
	\label{eq:R-Linf-anisot}
		R \in L^\infty(\R^n) \quad\text{with}\quad \norma*{R}_{L^\infty(\R^n)} \leq \left(\frac{p}{n-p}\right)^{\! p-1} \hat{c}_0^{-1} \,\norma*{\kappa-1}_{L^\infty(\R^n)}.
	\end{gather}
	
	
	\subsection{Quantitative integral estimate.}
	\label{step:quant-est-anisot}
	
	In this step, we shall derive a quantitative estimate for an integral involving~$\abs{\mathring{W}}$, the traceless tensor defined in~\eqref{eq:tracless-W-anisot}, by means of Proposition~\ref{prop:fund-ineq-anisot}. Particularly, for any~$t \geq 1$, we aim to show the following estimate:
	\begin{equation}
	\label{eq:fund-est-anisot-def}
		\int_{\R^n} v^{1-n} \,\abs{\mathring{W}}^2 P^t \, dx + \int_{\R^n} v^{2-n}\,\abs*{\nabla v}^{p-2} P^{t-1} \,\abs*{\nabla P}^{2} \, dx \leq C_\sharp \defi(u,\kappa),
	\end{equation}
	for some constant~$C_\sharp>0$ depending only on~$n$,~$p$,~$H$,~$\norma*{\kappa}_{L^\infty(\R^n)}$, and~$t$.
	
	Let~$\phi \in C^\infty_c(\R^n)$ be a cut-off function as in Section~\ref{sec:anis-classif-energ}. By using~$\phi^2$ in~\eqref{eq:fund-ineq-anisot} as a test function, we obtain
	\begin{multline}
	\label{eq:I-toest}
		n \left(p-1\right) \left(1-c_{p,H}\right) \int_{B_{2r}} v^{1-n} \,\abs{\mathring{W}}^2 \, P^t \phi \, dx + t \int_{B_{2r}} v^{2-n} P^{t-1} \left\langle A\nabla P, \nabla P \right\rangle \phi^2 \, dx \\
		\leq \mathsf{I}^r_1 + \mathsf{I}^r_2 + \mathsf{I}^r_3,
	\end{multline}
	where $A_r=B_{2r}\setminus B_r$, and
	\begin{align*}
		\mathsf{I}_1^r &\coloneqq - 2\int_{A_{r}} v^{2-n} \left\langle A\nabla P, \nabla \phi \right\rangle  P^t \phi \, dx, \\
		\mathsf{I}_2^r &\coloneqq \left(1-p\right) \int_{B_{2r}} v^{1-n} \left(R^2+PR\right) P^t \phi^2 \, dx, \\
		\mathsf{I}_3^r &\coloneqq n \left(1-p\right) \int_{B_{2r}} v^{1-n}  \left\langle \nabla R, \stressv \right\rangle P^t \phi^2 \, dx.
	\end{align*}
	We proceed with the estimate each of these integrals. In the following calculations~$C$ and~$C'$ will be positive constants, possibly differing from line to line, but independent of~$r$. Additionally, the pointwise estimates will be performed on the set~$\R^n\setminus \mathcal{Z}_v$, which we recall that it is a set of full measure.
    
	From~\eqref{eq:gradP-anisot}, we have
	\begin{equation*}
	\label{eq:piec1}
		\abs*{\nabla P} \leq C \left(v^{-1} \,\abs{\mathring{W}} \abs*{\nabla v} + v^{-1} \,\abs*{\nabla v} \abs{R} \right)\!.
	\end{equation*}
	This inequality coupled with \eqref{eq:A-boundls} and Young’s inequality, yields
	\begin{align}
	\notag
		\abs*{\mathsf{I}_1^r} &\leq C \int_{A_{r}} v^{2-n} \,\abs*{\nabla v}^{p-2} \,\abs*{\nabla P} \abs*{\nabla \phi} P^t \phi \, dx \\
	\notag
		&\leq C \int_{A_r} v^{1-n} \,\abs*{\nabla v}^{p-1} \,\abs{\mathring{W}} \,\abs*{\nabla\phi} P^t  \phi \, dx +  C' \int_{A_r} v^{1-n} \,\abs*{\nabla v}^{p-1} \,\abs*{\nabla\phi} \abs{R} P^t \phi \, dx \\
	\notag
		&\leq \frac{n}{2} \left(p-1\right) \left(1-c_{p,H}\right) \int_{B_{2r}} v^{1-n} \,\abs{\mathring{W}}^2 P^t  \phi^2 \, dx +  C \int_{A_r} v^{1-n} \,\abs*{\nabla v}^{2(p-1)} \,\abs*{\nabla\phi}^2 P^t \, dx \\
	\label{eq:est1-anisot}
		&\quad+ C' \int_{A_r} v^{1-n} \,\abs*{\nabla v}^{p-1} \,\abs*{\nabla\phi} \abs{R} P^t \phi \, dx.
	\end{align}
	Finally, computing as~\cite{cg-plap} p.~27, from~\eqref{eq:defPfunct-anisot},~\eqref{eq:defRem-anisot},~\eqref{eq:D1p-anisot},~\eqref{eq:P-Linf-anisot}--\eqref{eq:R-Linf-anisot}, and the fact that~$t \geq 1$, we infer that
	\begin{equation}
	\label{eq:est11-anisot}
		\lim_{r \to +\infty} \left\{ \int_{A_r} v^{1-n} \,\abs*{\nabla v}^{2(p-1)} \,\abs*{\nabla\phi}^2 P^t  \, dx + \int_{A_r} v^{1-n} \,\abs*{\nabla v}^{p-1} \,\abs*{\nabla\phi} \abs{R} P^t \phi \, dx \right\} = 0.
	\end{equation}
	
	For the second integral we argue once more as on p.~27 of~\cite{cg-plap} to deduce its quantitative smallness, i.e.,
	\begin{equation}
	\label{eq:est2-anisot}
		\abs*{\mathsf{I}_2^r} \leq C \defi(u,\kappa).		
	\end{equation}
	
	For~$\mathsf{I}_3^r$, in light of~\eqref{eq:defRem-anisot}--\eqref{eq:regP-anisot}, and the regularity of~$\kappa$, we perform an integration by parts to obtain
	\begin{equation}
	\label{eq:est3-anisot}
		\mathsf{I}_3^r = n \left(1-p\right) \left(\mathsf{J}_1^r+\mathsf{J}_2^r+\mathsf{J}_3^r+\mathsf{J}_4^r\right) \!,
	\end{equation}
	where
	\begin{align*}
		\mathsf{J}_1^r &\coloneqq \left(1-n\right) \int_{B_{2r}} v^{-n}  \left\langle \nabla v,  \stressv \right\rangle  RP^t \phi^2 \, dx, \\
		\mathsf{J}_2^r &\coloneqq t \int_{B_{2r}} v^{1-n}  \left\langle \nabla P, \stressv \right\rangle R P^{t-1} \phi^2 \, dx, \\
		\mathsf{J}_3^r &\coloneqq  2 \int_{A_{r}} v^{1-n}  \left\langle \nabla \phi, \stressv \right\rangle  RP^t \phi \, dx, \\
		\mathsf{J}_4^r &\coloneqq \int_{B_{2r}} v^{1-n} \left[\tr W \right] RP^t \phi^2 \, dx = \int_{B_{2r}} v^{1-n} \left(P+R\right) RP^t \phi^2 \, dx.
	\end{align*}
	
	Using also~\eqref{eq:def-a-2} and~\eqref{eq:equiv-norma}, as in~\cite{cg-plap} p.~28, we deduce that
	\begin{equation}
	\label{eq:est4-anisot}
		\abs*{\mathsf{J}_1^r} + \abs*{\mathsf{J}_4^r} \leq C \defi(u,\kappa) \quad\text{and}\quad \lim_{r \to +\infty} \mathsf{J}_3^r = 0.
	\end{equation}
	
	Before estimating~$\mathsf{J}_2^r$, we go back to the left-hand side of~\eqref{eq:I-toest} and observe that by \eqref{eq:A-boundls}, we have
	\begin{equation*}
		c\, H^{p-2}(\nabla v) \,\abs*{\nabla P}^2 \leq \left\langle A\nabla P, \nabla P \right\rangle\!.
	\end{equation*}
	On top of the last estimate and reasoning as in~\cite{cg-plap} p.~28, we get
	\begin{equation}
	\label{eq:est5-anisot}
		\abs{\mathsf{J}_2^r} \leq C_H\,\frac{t}{2}\,\int_{B_{2r}} v^{2-n} H^{p-2}(\nabla v) P^{t-1} \,\abs*{\nabla P}^2 \, dx + C \defi(u,\kappa).
	\end{equation}
	
	Finally,~\eqref{eq:fund-est-anisot-def} follows by piecing~\eqref{eq:I-toest} together with~\eqref{eq:est1-anisot}--\eqref{eq:est5-anisot}, using~\eqref{eq:equiv-norma} and letting~$r \to +\infty$.
	
	
	\subsection{Integral estimate for $W$.}
	\label{step:est-W-anisot}
	
	In this step, we will derive an estimate for~$W$ in an appropriate weighted Lebesgue space.
	
	Let~$r>1$ and~$\mathsf{t} \in (0,1)$ two fixed quantities  to be determined by the end of the proof in dependence of~$\defi(u,\kappa)$.
	
	We start by recalling that in~\ref{step:decay-anisot} we have identified a point~$x_0 \in \R^n$ where~$u$ attains its maximum, so~$v$ attains its minimum at~$x_0$. By taking advantage of~\eqref{eq:bounds-v-anisot}, we can quantitatively locate this point; indeed we have
	\begin{equation}
	\label{eq:bound-vx0-anisot}
		v(x_0) \leq v(0) \leq \widehat{C}_0.
	\end{equation}
	Hence, by setting
	\begin{equation}
	\label{eq:defR-rad-anisot}
		\mathscr{R} \coloneqq \left(\frac{\widehat{C}_0-\hat{c}_0}{\hat{c}_0}\right)^{\!\!\frac{p-1}{p}}\!,
	\end{equation}
	which depends only on~$n$,~$p$,~$H$, and~$\norma*{\kappa}_{L^\infty(\R^n)}$, we necessarily have that~$x_0 \in B_{\mathscr{R}}$ due to~\eqref{eq:bounds-v-anisot}. From this, we have~$B_\mathsf{t}(x_0) \subseteq B_r$ provided that
	\begin{equation}
		\label{eq:cond-r-R}
		r \geq \mathscr{R}+2.
	\end{equation}
	
	We now claim that
	\begin{equation}
	\label{eq:intW-to-prove}
		\int_{B_{r}} v^{-n} \,\abs*{\nabla \stressv - \frac{\overline{P}}{n} \Id}^q \, dx \leq C_\flat \,\mathscr{F}_{\! q} \quad \mbox{for every } q \in [1,2],
	\end{equation}
	for some~$C_\flat>0$ depending only on~$n$,~$p$,~$H$, and~$\norma*{\kappa}_{L^\infty(\R^n)}$, some function of the deficit ~$\mathscr{F}_{\! q}$ to be determined later, where we set
	\begin{equation*}
		\overline{P} \coloneqq \dashint_{B_\mathsf{t}(x_0)} P \, dx.
	\end{equation*}
	To prove~\eqref{eq:intW-to-prove}, we begin by noticing that, thanks to the definition of the~$P$-function in~\eqref{eq:defPfunct-anisot} and \eqref{eq:bound-vx0-anisot}, estimate~\eqref{eq:fund-est-anisot-def} with~$t=1$ immediately yields
	\begin{equation*}
		\label{eq:stima-def}
		\int_{\R^n} v^{-n} \,\abs{\mathring{W}}^2 \, dx \leq C_2 \defi(u,\kappa),
	\end{equation*}
	for some~$C_2>0$ depending only on~$n$,~$p$,~$H$, and~$\norma*{\kappa}_{L^\infty(\R^n)}$. Therefore, using H\"older inequality, together with the anisotropic Sobolev inequality and the energy bound~\eqref{eq:ipotesi-energ-anisot}, for~$1\leq q\leq2$ we get
	\begin{equation}
	\label{eq:est-W-dot-anisot}
		\int_{\R^n} v^{-n}\,\abs{\mathring{W}}^q \, dx \leq \left(\int_{\R^n} v^{-n}\,\abs{\mathring{W}}^2 \, dx \right)^{\!\!\frac{q}{2}} \left(\int_{\R^n} v^{-n} \, dx \right)^{\!\!\frac{2-q}{2}} \leq C_3 \defi(u,\kappa)^{\frac{q}{2}},
	\end{equation}
	for some~$C_3>0$ depending only  on~$n$,~$p$,~$H$,~$\norma*{\kappa}_{L^\infty(\R^n)}$, and~$q$.
	
	Now we note that by exploiting~\eqref{eq:trW-anisot}, it follows
	\begin{multline}
		\label{eq:norma-L1-anisot}
		\int_{B_{r}} v^{-n}\,\abs*{W - \frac{\overline{P}}{n} \Id}^q \, dx \\
		\leq 4^{q-1} \int_{B_{r}} v^{-n}\,\abs{\mathring{W}}^q \, dx + 4^{q-1} \int_{B_{r}} v^{-n}\,\abs{P-\overline{P}}^q \, dx 
		+ 4^{q-1} \int_{B_{r}} v^{-n}\,\abs*{R}^q \, dx \\ 
		\eqqcolon 4^{q-1} \left\{\int_{B_{r}} v^{-n}\,\abs{\mathring{W}}^q \, dx + \mathcal{I}_1 + \mathcal{I}_2 \right\}\!.
	\end{multline}
	We shall proceed with the estimate of both~$\mathcal{I}_1$ and~$\mathcal{I}_2$. In the subsequent calculations~$C$ will denote a positive constant, possibly differing from line to line, which depends only on~$n$,~$p$,~$H$,~$\norma*{\kappa}_{L^\infty(\R^n)}$, and~$q$.
	
	By~\eqref{eq:regP-anisot}, the Poincar\'e inequality of Equation~(7.45) in~\cite{gt}, and~\eqref{eq:bounds-v-anisot} -- which also implies that~$v$ is bounded from below --, we get 
	\begin{equation}
	\label{eq:est-mI1}
		\mathcal{I}_1 \leq C \, r^{\frac{np}{p-1}} \, \mathcal{C}_P \!\left(r,\mathsf{t}\right)^q \int_{B_{r}} v^{-n}\,\abs{\nabla P}^q \, dx,
	\end{equation}
	where
	\begin{equation}
	\label{eq:def-CP-anisot}
		\mathcal{C}_P \!\left(r,\mathsf{t}\right) \coloneqq 2^n r^n \, \mathsf{t}^{1-n}.
	\end{equation}
	From~\eqref{eq:gradP-anisot}, and arguing as in~\cite{cg-plap} p.~30, we deduce that
	\begin{align}
		\notag
		\int_{B_{r}} v^{-n}\,\abs{\nabla P}^q \, dx &\leq C \int_{B_{r}} v^{-n-q} \,\abs{\mathring{W}}^q\abs{\nabla v}^q \, dx + C \int_{B_{r}} v^{-n-q} \,\abs{\nabla v}^q \abs{R}^q \, dx \\
		\label{eq:est-mI11}
		&\leq C \left(\defi(u,\kappa) + \defi(u,\kappa)^{\frac{q}{2}}\right)\!.
	\end{align}
	For the second integral, by~\eqref{eq:defRem-anisot} and since~$v$ is bounded below, we have
	\begin{equation}
	\label{eq:est-mI2}
		\mathcal{I}_2 \leq C \int_{B_r} v^{-n-q} \,\abs*{\kappa-1}^q \, dx \leq C \int_{B_r} v^{-n} \,\abs*{\kappa-1} \, dx \leq C \defi(u,\kappa).
	\end{equation}
	
	Since~\eqref{eq:def-small-anisot} is in force and~$r>1$, combining~\eqref{eq:norma-L1-anisot} with~\eqref{eq:est-W-dot-anisot},~\eqref{eq:est-mI1}, and~\eqref{eq:est-mI11}--\eqref{eq:est-mI2}, we infer the validity of~\eqref{eq:intW-to-prove} with
	\begin{equation*}
		\mathscr{F}_{\! q} \coloneqq r^{\frac{np}{p-1}} \left[1+ \mathcal{C}_P \!\left(r,\mathsf{t}\right)^q\right] \defi(u,\kappa)^{\frac{q}{2}},
	\end{equation*}
	where~$\mathcal{C}_P \!\left(r,\mathsf{t}\right)$ is given by~\eqref{eq:def-CP-anisot}.
	
	
	\subsection{Construction of the approximating functions.}
	\label{step:approx-func-anisot}
	
	We proceed now via a two-step approximation. First, we show that, as a consequence of~\eqref{eq:intW-to-prove}, it is possible to construct a function~$\mathsf{Q}$ that approximates~$v$ in~$L^p_{\omega}(B_{r})$, with~$\omega \coloneqq v^{-n-p+1}$, and whose gradient approximates~$\nabla v$ in~$L^p_{\omega}(B_{r})$, where~$w \coloneqq v^{-n}$. Secondly, we define a function~$\mathcal{Q}$ which, when inverted through the transformation~\eqref{eq:defv-anisot}, yields a~$(p,H)$-bubble. Both approximations are obtained using the technique developed in~\cite{cg-plap}, which also applies to the anisotropic case.
	
	We begin by defining the first approximating function by
	\begin{equation*}
		\mathsf{Q}(x) \coloneqq v(x_0) + \frac{p-1}{p} \left(\frac{\overline{P}}{n}\right)^{\!\frac{1}{p-1}} H_0^{\frac{p}{p-1}}(x-x_0),
	\end{equation*}
	which is of class~$C^\infty(\R^n \setminus \left\{x_0\right\})$ with~$\mathcal{Z}_{\mathsf{Q}} =\{x\in \R^n \mid \nabla \mathsf{Q}(x)=0\}= \left\{x_0\right\}$. Moreover, up to uniquely extending the stress field~$a(\xi)$ on~$\mathcal{Z}_{\mathsf{Q}}$, we notice that
	\begin{gather}
		\label{eq:stressQ-anisot}
		a \!\left(\nabla \mathsf{Q}\right)\!(x) = \frac{\overline{P}}{n} \left(x-x_0\right) \quad\mbox{for all } x \in \R^n, \\
		\notag
		\nabla a \!\left(\nabla \mathsf{Q}\right)\!(x) = \frac{\overline{P}}{n} \Id \quad\mbox{for all } x \in \R^n.
	\end{gather}
	Observe, in particular, that~\eqref{eq:stressQ-anisot} can be verified by using the homogeneity of~$H$ and Lemma~3.1 in~\cite{cianchi-sal}. Therefore, we have~$a \!\left(\nabla \mathsf{Q}\right) \in C^\infty(\R^n)$ and, from~\eqref{eq:intW-to-prove} and \eqref{eq:stressQ-anisot}, we deduce that
	\begin{equation}
	\label{eq:to-rewrite-anisot}
		\int_{B_{r}} v^{-n} \,\abs*{\nabla \stressv - \nabla a \!\left(\nabla \mathsf{Q}\right)}^q \, dx \leq C_\flat \,\mathscr{F}_{\! q} \quad \mbox{for every } q \in [1,2].
	\end{equation}
	For notational simplicity, we set
	\begin{equation*}
		\zeta \coloneqq \stressv - a \!\left(\nabla \mathsf{Q}\right)\!,
	\end{equation*}
	which is well defined in~$\R^n$, since we extended~$\stressv$ to zero on~$\mathcal{Z}_v$, and note that we can rewrite~\eqref{eq:to-rewrite-anisot} in the form
	\begin{equation}
	\label{eq:stimagrad-anisot}
		\int_{B_{r}} v^{-n}\,\abs*{\nabla \zeta}^q \, dx \leq C_\flat \,\mathscr{F}_{\! q} \quad \mbox{for every } q \in [1,2],
	\end{equation}
	with~$\zeta \in W^{1,2}_{\loc}(\R^n)$ by~\eqref{eq:regv-anisot}.
	
	We now introduce the small parameter~$\tau \in (0,1)$, to be chosen later, and define, for each component of~$\zeta$, the average
	\begin{equation*}
		\overline{\zeta_{i}} \coloneqq \dashint_{B_\tau(x_0)} \zeta_{i} \, dx = \dashint_{B_\tau(x_0)} a_i \!\left(\nabla v\right) dx,
	\end{equation*}
	where we used that, by~\eqref{eq:stressQ-anisot}, each component of~$a \!\left(\nabla \mathsf{Q}\right)$ is harmonic in~$\R^n$. Therefore, we get
	\begin{equation}
		\label{eq:est1-media-anisot}
		\abs{\overline{\zeta_{i}}} \leq \dashint_{B_\tau(x_0)} \,\abs*{a_i \!\left(\nabla v\right)} \, dx .
	\end{equation}
	Furthermore, using also~$\abs*{a(\xi)}\leq C\,\abs*{\xi}^{p-1}$ and~\eqref{eq:regv-anisot}, we have
	\begin{equation}
	\label{eq:est2-media-anisot}
		\begin{split}
			\norma*{\stressv}_{L^1(B_{\tau}(x_0))} &\leq C \norma*{\nabla v}_{L^{p-1}(B_{\tau}(x_0) )}^{p-1} = \norma*{\nabla v - \nabla v(x_0)}_{L^{p-1}(B_{\tau}(x_0))}^{p-1} \\
			&\leq \|\nabla v\|_{C^{0,\alpha}(B_{1+\mathscr{R}})} \tau^{n+\alpha\left(p-1\right)} \leq C\,\tau^{n+\alpha\left(p-1\right)},
		\end{split}
	\end{equation}
	for some~$C>0$ and~$\alpha\in(0,1)$ depending only on~$n$,~$p$,~$H$, and~$\norma*{\kappa}_{L^\infty(\R^n)}$. Here we used that~$B_\tau(x_0)\subseteq B_1(x_0)\subseteq B_{1+\mathscr{R}}$, with~$\mathscr{R}$ given by~\eqref{eq:defR-rad-anisot}. Thus,~\eqref{eq:est1-media-anisot} simplifies to
	\begin{equation}
		\label{eq:L1-toest-1-anisot}
		\abs{\overline{\zeta_{i}}} \leq C \tau^{\alpha\left(p-1\right)},
	\end{equation}
	for some~$C>0$ depending only on~$n$,~$p$,~$H$, and~$\norma*{\kappa}_{L^\infty(\R^n)}$.
	
	Observe that~\eqref{eq:cond-r-R} yields that~$B_\tau(x_0) \subseteq B_r$, hence, by applying once more the Poincar\'e inequality of Equation~(7.45) in~\cite{gt} and using also~\eqref{eq:bounds-v-anisot}, we get
	\begin{equation}
	\label{eq:L1-toest-2-anisot}
		\norma{\zeta_{i} - \overline{\zeta_{i}}}_{L^q_{w}(B_{r})}^q \leq C \, r^{\frac{np}{p-1}} \, \mathcal{C}_P \!\left(r,\tau\right)^q \int_{B_{r}} v^{-n}\,\abs*{\nabla \zeta}^q \, dx \leq C \, r^{\frac{np}{p-1}} \, \mathcal{C}_P \!\left(r,\tau\right)^q \mathscr{F}_{\! q},
	\end{equation}
	where~$\mathcal{C}_P \!\left(r,\tau\right)$ is given by~\eqref{eq:def-CP-anisot} and we exploited~\eqref{eq:stimagrad-anisot}. Finally, since for each component of~$\zeta$ we have
	\begin{equation*}
		\norma{\zeta_{i}}_{L^q_{w}(B_{r})}^q \leq 2^{q-1}\norma{\zeta_{i} - \overline{\zeta_{i}}}_{L^q_{w}(B_{r})}^q + 2^{q-1} \abs{\overline{\zeta_{i}}}^q \,\norma{v^{-n}}_{L^1(\R^n)},
	\end{equation*}
	from~\eqref{eq:L1-toest-1-anisot} and~\eqref{eq:L1-toest-2-anisot}, we conclude that
	\begin{equation}
	\label{eq:L1-stress-anisot}
		\norma{\zeta}_{L^q_{w}(B_{r})}^q \leq C_4 \left[r^{\frac{np}{p-1}} \, \mathcal{C}_P \!\left(r,\tau\right)^q \mathscr{F}_{\! q}+ \tau^{\alpha q\left(p-1\right)} \right] \eqqcolon C_4 \,\mathscr{G}_{q},
	\end{equation}
	for some~$C_4>0$ depending only on~$n$,~$p$,~$H$,~$\norma*{\kappa}_{L^\infty(\R^n)}$, and~$q$.
	
	Once~\eqref{eq:L1-stress-anisot} has been established, arguing as in~\cite{cg-plap} p.~32, we deduce that
	\begin{equation}
	\label{eq:norma-p-grad-anisot}
		\int_{B_{r}} v^{-n}\,\abs*{\nabla \!\left(v-\mathsf{Q}\right)}^{p} \, dx \leq C_5 \, r^{\frac{p\left(2-p\right)_+}{p-1}} \, \mathscr{G}_{\min\left\{p,\frac{p}{p-1}\right\}},
	\end{equation}
	for some~$C_5>0$ depending only on~$n$,~$p$,~$H$, and~$\norma*{\kappa}_{L^\infty(\R^n)}$. From this, arguing once more as in~\cite{cg-plap} pp.~32-33, we also have
	\begin{equation}
		\label{eq:norma-p-func-anisot}
		\int_{B_{r}} v^{-n-p+1}\,\abs*{\left(v-\mathsf{Q}\right)}^{p} \, dx \leq C_6 \, r^{\frac{p\left(2-p\right)_+}{p-1}} \, \mathscr{G}_{\min\left\{p,\frac{p}{p-1}\right\}},
	\end{equation}
	for some~$C_6>0$ depending only on~$n$,~$p$,~$H$, and~$\norma*{\kappa}_{L^\infty(\R^n)}$.
	
	As anticipated earlier, we now define the second approximating function given by
	\begin{equation}\label{second:approx}
		\mathcal{Q}(x) \coloneqq \frac{\lambda^\frac{p}{p-1}+H_0^\frac{p}{p-1}(x-x_0)}{\lambda^\frac{1}{p-1} \, n^\frac{1}{p} \left(\frac{n-p}{p-1}\right)^{\!\!\frac{p-1}{p}}} \quad\mbox{with}\quad \lambda = \frac{1}{\overline{P}} \left(\frac{p}{p-1}\right)^{\! p-1} n^\frac{1}{p} \left(\frac{n-p}{p-1}\right)^{\!\!- \frac{(p-1)^2}{p}}\!.
	\end{equation}
	Note, in particular, that when inverted through the transformation~\eqref{eq:defv-anisot},~$\mathcal{Q}$ yields a~$(p,H)$-bubble of the form~\eqref{eq:pH-bubb}. A direct computation reveals that
	\begin{equation*}
		\label{eq:diffQ-anisot}
		\nabla \mathcal{Q} = \nabla \mathsf{Q} \quad\mbox{in } \R^n,
	\end{equation*}
	therefore~\eqref{eq:norma-p-grad-anisot} reads as
	\begin{equation}
	\label{eq:norma-p-grad-2-anisot}
		\int_{B_{r}} v^{-n}\,\abs*{\nabla \!\left(v-\mathcal{Q}\right) }^{p} \, dx \leq C_5 \, r^{\frac{p\left(2-p\right)_+}{p-1}} \, \mathscr{G}_{\min\left\{p,\frac{p}{p-1}\right\}}.
	\end{equation}
	Moreover, we have
	\begin{equation}
		\label{eq:diff-per-Linf-anisot}
		\mathsf{Q}(x) - \mathcal{Q}(x) = v(x_0) - \frac{1}{\overline{P} }\left(\frac{p}{n-p}\right)^{\! p-1} \quad\mbox{for every } x \in \R^n.
	\end{equation}
	We now aim to prove that the right-hand side of~\eqref{eq:diff-per-Linf-anisot} is~$L^\infty$-small and then upgrade this estimate to a suitable Lebesgue norm. To this end, we observe that
	\begin{align*}
		v(x_0) \overline{P} &- \left(\frac{p}{n-p}\right)^{\! p-1} \\
		&= \dashint_{B_\mathsf{t}(x_0)} \frac{n(p-1)}{p} \frac{v(x_0)}{v} \,
		H^p \!\left(\nabla v-\nabla v(x_0)\right) + \left(\frac{p}{n-p}\right)^{\! p-1} \frac{v(x_0)-v}{v} \, dx.
	\end{align*}
	where we used that~$\nabla v(x_0)=0$ by minimality. Thus, by exploiting~\eqref{eq:equiv-norma},~\eqref{eq:regv-anisot}, as well as the fact that~$v$ is bounded below, and arguing as in the deduction of~\eqref{eq:est2-media-anisot}, we infer that
	\begin{equation}
		\label{eq:to-mult}
		\abs*{v(x_0) \overline{P} - \left(\frac{p}{n-p}\right)^{\! p-1}} \leq C \,\mathsf{t}^{\alpha p},
	\end{equation}
	for some~$C>0$ depending only on~$n$,~$p$,~$H$, and~$\norma*{\kappa}_{L^\infty(\R^n)}$. Moreover, since we have already proven that~$x_0 \in B_{\mathscr{R}}$, with~$\mathscr{R}$ defined in~\eqref{eq:defR-rad-anisot}, the bounds~\eqref{eq:bound-gradv-anisot} and~\eqref{eq:bound-vx0-anisot} give that~$v \leq C$ in~$B_\mathsf{t}(x_0)$, for some~$C>0$ depending only on~$n$,~$p$,~$H$, and~$\norma*{\kappa}_{L^\infty(\R^n)}$. As a consequence,
	\begin{equation}
		\label{eq:bound-mean-below}
		\overline{P} \geq \left(\frac{p}{n-p}\right)^{\! p-1} \dashint_{B_\mathsf{t}(x_0)}  v^{-1} \, dx \geq C,
	\end{equation}
	for some~$C>0$ depending only on~$n$,~$p$,~$H$, and~$\norma*{\kappa}_{L^\infty(\R^n)}$. This, together with~\eqref{eq:P-Linf-anisot}, also implies that $\lambda$ in \eqref{second:approx} satisfies
	\begin{equation}
	\label{eq:bounds-lambda-anisot}
		\Lambda_1 \leq \lambda \leq \Lambda_2,
	\end{equation}
	for some~$\Lambda_1,\Lambda_2>0$ depending only on~$n$,~$p$, and~$\norma*{\kappa}_{L^\infty(\R^n)}$. In addition, dividing both sides of~\eqref{eq:to-mult} by~$\overline{P}$, and exploiting~\eqref{eq:bound-mean-below}, we conclude that
	\begin{equation*}
		\abs*{v(x_0) - \frac{1}{\overline{P}} \left(\frac{p}{n-p}\right)^{\! p-1}} \leq C \,\mathsf{t}^{\alpha p},
	\end{equation*}
	for some~$C>0$ depending only on~$n$,~$p$,~$H$, and~$\norma*{\kappa}_{L^\infty(\R^n)}$. From the latter and~\eqref{eq:diff-per-Linf-anisot}, we read off
	\begin{equation*}
		\norma*{\mathsf{Q} - \mathcal{Q}}_{L^\infty(\R^n)} \leq C \,\mathsf{t}^{\alpha p},
	\end{equation*}
	which, using also the boundedness~$v$ from below,~\eqref{eq:ipotesi-energ-anisot}, and the anisotropic Sobolev inequality, implies
	\begin{equation}
		\label{eq:closQ-anisot}
		\norma*{\mathsf{Q} - \mathcal{Q}}_{L^p_\omega(\R^n)}^p \leq \norma*{\mathsf{Q} - \mathcal{Q}}_{L^\infty(\R^n)}^p \int_{\R^n} v^{-n-p+1} \, dx \leq C_8 \,\mathsf{t}^{\alpha p^2},
	\end{equation}
	for some~$C_8>0$ depending only on~$n$,~$p$,~$H$, and~$\norma*{\kappa}_{L^\infty(\R^n)}$.
	
	Taking advantage of~\eqref{eq:norma-p-func-anisot} and~\eqref{eq:closQ-anisot}, using the triangle inequality, we infer
	\begin{align}
		\notag
		\int_{B_{r}} v^{-n-p+1} \,\abs*{v-\mathcal{Q}}^{p} \, dx
		&\leq 2^{p-1}  C_6 \, r^{\frac{p\left(2-p\right)_+}{2\left(p-1\right)}} \, \mathscr{G}_{\min\left\{p,\frac{p}{p-1}\right\}} + 2^{p-1} C_8 \,\mathsf{t}^{\alpha p^2} \\
		\label{eq:diff-vQ-anisot}
		&\leq C_{9} \left[ r^{\frac{p\left(2-p\right)_+}{p-1}} \, \mathscr{G}_{\min\left\{p,\frac{p}{p-1}\right\}} + \mathsf{t}^{\alpha p^2} \right] \!,
	\end{align}
	for some~$C_{9}>0$ depending only on~$n$,~$p$,~$H$, and~$\norma*{\kappa}_{L^\infty(\R^n)}$.
	
	
	\subsection{Going back to~$u$.}
	\label{step:backtou-anisot}
	
	We are now ready to return to the level of~$u$ by defining the final approximation as
	\begin{equation*}
		\mathcal{U} \coloneqq \mathcal{Q}^{-\frac{n-p}{p}} = U_p[x_0,\lambda],
	\end{equation*}
	obtained by inverting~$\mathcal{Q}$ through~\eqref{eq:defv-anisot}. The goal is to establish the~$\mathcal{D}^{1,p}$-closeness of~$\mathcal{U}$ to~$u$, making use of~\eqref{eq:norma-p-grad-2-anisot} and~\eqref{eq:diff-vQ-anisot}.
	
	We begin by noting that
	\begin{equation*}
		\nabla u = \frac{p-n}{p} \, v^{-\frac{n}{p}} \,\nabla v \quad\mbox{and}\quad \nabla \mathcal{U} = \frac{p-n}{p} \, \mathcal{Q}^{-\frac{n}{p}} \,\nabla \mathcal{Q} \quad \mbox{in } \R^n,
	\end{equation*}
	thus we have
	\begin{multline}
		\label{eq:normp-toest-anisot-1}
		\norma*{\nabla \!\left(u-\mathcal{U}\right) }_{L^p(B_{r})}^p = \left(\frac{n-p}{p}\right)^{\! p} \int_{B_{r}} \,\abs*{v^{-\frac{n}{p}} \,\nabla v-\mathcal{Q}^{-\frac{n}{p}} \,\nabla \mathcal{Q}} \, dx \\
		\leq 2^{p-1} \left(\frac{n-p}{p}\right)^{\! p} \left\{\int_{B_{r}} v^{-n}  \,\abs*{\nabla \!\left(v-\mathcal{Q}\right)}^p \, dx + \int_{B_{r}} \,\abs*{v^{-\frac{n}{p}}-\mathcal{Q}^{-\frac{n}{p}}}^p \,\abs*{\nabla \mathcal{Q}}^p \, dx\right\} \!.
	\end{multline}
	Arguing once more as in~\cite{cg-plap} p.~35, we deduce that
	\begin{equation}
	\label{eq:normp-toest-anisot-2}
		\int_{B_{r}} \,\abs*{v^{-\frac{n}{p}}-\mathcal{Q}^{-\frac{n}{p}}}^p \,\abs*{\nabla \mathcal{Q}}^p \, dx \leq C \int_{B_{r}} v^{-n-p+1} \,\abs*{v-\mathcal{Q}}^{p} \, dx,
	\end{equation}
	for some constant~$C>0$ depending only on~$n$,~$p$,~$H$, and~$\norma*{\kappa}_{L^\infty(\R^n)}$.
	
	Combining~\eqref{eq:normp-toest-anisot-1} with~\eqref{eq:norma-p-grad-2-anisot},~\eqref{eq:diff-vQ-anisot}, and~\eqref{eq:normp-toest-anisot-2}, we conclude that
	\begin{equation}
		\label{eq:D1p-int-anisot}
		\norma*{\nabla \!\left(u-\mathcal{U}\right) }_{L^p(B_{r})}^p \leq C_{10} \left[ r^{\frac{p\left(2-p\right)_+}{p-1}} \, \mathscr{G}_{\min\left\{p,\frac{p}{p-1}\right\}} + \mathsf{t}^{\alpha p^2} \right]\!,
	\end{equation} 
	for some~$C_{10}>0$ depending only on~$n$,~$p$,~$H$, and~$\norma*{\kappa}_{L^\infty(\R^n)}$, with
	\begin{equation*}
		\mathscr{G}_{q} = r^{\frac{2np}{p-1}} \, \mathcal{C}_P \!\left(r,\tau\right)^q \left[1+ \mathcal{C}_P \!\left(r,\mathsf{t}\right)^q\right] \defi(u,\kappa)^{\frac{q}{2}}+ \tau^{\alpha q\left(p-1\right)}
	\end{equation*}
	and where~$\mathcal{C}_P \!\left(r,\cdot\right)$ is defined in~\eqref{eq:def-CP-anisot}.
	
	Moreover, by~\eqref{eq:cond-r-R}, it holds that~$B_1(x_0) \subseteq B_{r}$, therefore, exploiting also~\eqref{eq:equiv-H0} and~\eqref{eq:bounds-lambda-anisot}, we have
	\begin{equation*}
		\abs*{\nabla \mathcal{U}(x)} \leq C \,\abs*{x}^{\frac{1-n}{p-1}} \quad \mbox{for } x \in \R^n \setminus B_{r},
	\end{equation*}
	for some~$C>0$ depending only on~$n$,~$p$,~$H$, and~$\norma*{\kappa}_{L^\infty(\R^n)}$. The latter and~\eqref{eq:bound-gradu-anisot} imply
	\begin{equation}
		\label{eq:D1p-ext-anisot}
		\norma*{\nabla \!\left(u-\mathcal{U}\right) }_{L^p(\R^n \setminus B_{r})}^p \leq C \int_{\R^n \setminus B_{r}} \,\abs*{x}^{-\frac{p}{p-1} (n-1)} \, dx \leq C \, r^{-\frac{n-p}{p-1}},
	\end{equation}
	for some~$C>0$ depending only on~$n$,~$p$,~$H$, and~$\norma*{\kappa}_{L^\infty(\R^n)}$.
	
	Choosing the parameters
	\begin{align*}
		r&\coloneqq \defi(u,\kappa)^{-\min\left\{\frac{q}{4\mathfrak{p}},\frac{\alpha q\left(p-1\right)^2}{16\left(n-1\right)p\left(2-p\right)_+}\right\}} = \defi(u,\kappa)^{-\mathsf{m}}, \\
		\mathsf{t}=\tau&\coloneqq \defi(u,\kappa)^{\frac{1}{8(n-1)}},
	\end{align*}
	where
	\begin{equation*}
		q \coloneqq \min\left\{p,\frac{p}{p-1}\right\} \quad\mbox{and}\quad \mathfrak{p} \coloneqq 2nq+\frac{p}{p-1}\left[2n+\left(2-p\right)_+\right]\!,
	\end{equation*}
	and summing~\eqref{eq:D1p-int-anisot} with~\eqref{eq:D1p-ext-anisot}, we finally conclude that
	\begin{equation}
	\label{eq:D1p-final}
		\norma*{\nabla \!\left(u-\mathcal{U}\right) }_{L^p(\R^n)} \leq \mathscr{C} \defi(u,\kappa)^{\vartheta},
	\end{equation}
	for some~$\mathscr{C}>0$ and~$\vartheta \in (0,1)$, depending only on~$n$,~$p$,~$H$, and~$\norma*{\kappa}_{L^\infty(\R^n)}$. In addition, we observe that the above requirements on the parameters, in particular~\eqref{eq:cond-r-R}, are verified provided that
	\begin{equation*}
		\gamma \leq \min\left\{\frac{1}{2}, \delta, \left(\mathscr{R}+2\right)^{\!-\frac{1}{\mathsf{m}}} \right\},
	\end{equation*}
	where~$\gamma$ and~$\delta$ verifies~\eqref{eq:def-small-anisot} and~\eqref{eq:delta-to-fix-anisot}, respectively.
	
	
	\subsection{Conclusion.}
	\label{step:conclusion-anisot} 
	
	As indicated at the end of~\ref{step:struwe-reduction}, we resume using the notation~$u_\epsilon$ in place of~$u$. Hence, from~\eqref{eq:D1p-final}, it follows that
	\begin{equation}
	\label{eq:conclusion-anisot}
		\norma*{\nabla \!\left(u_\epsilon-\mathcal{U}\right) }_{L^p(\R^n)} \leq \mathscr{C} \defi(u,\kappa)^{\vartheta}.
	\end{equation}
	
	Recall that we set~$u_\epsilon = T_{-z_\epsilon/\lambda_\epsilon,\lambda_\epsilon} \!\left(u\right)$. Therefore, point~\ref{it:inver-anisot} of Lemma~\ref{lem:sym-anisot} implies that~$u = T_{z_\epsilon,\lambda_\epsilon} \!\left(u_\epsilon\right)$. Defining~$\mathcal{U}_0 \coloneqq T_{z_\epsilon,\lambda_\epsilon} \!\left(\mathcal{U}\right)$, which is a~$(p,H)$-bubble by point~\ref{it:stillbub-anisot} of Lemma~\ref{lem:sym-anisot},~\eqref{eq:conclusion-anisot} reads as
	\begin{equation*}
		\norma*{\nabla \!\left(u-\mathcal{U}_0\right) }_{L^p(\R^n)} \leq \mathscr{C} \defi(u,\kappa)^{\vartheta},
	\end{equation*}
	thereby proving the result.
	
	\appendix
	
	\section{Proof of Lemma~\ref{lem:xi-p}}
	\label{app:lem}
	
	\noindent
		Of course, for~$k=1$ there is nothing to prove so we may assume~$k\geq2$.
		
		To prove~\eqref{eq:xi-p-stat}, we first claim that there exists a constant~$\mathsf{C}_{p,k}>0$ such that
		\begin{equation}
			\label{eq:xi-p}
			\abs*{\sum_{i=1}^{k} x_i}^p \leq \sum_{i=1}^{k} \,\abs*{x_i}^p + \mathsf{C}_{p,k} \left\{\sum_{i=1}^{k} \sum_{\substack{j=1 \\ j \neq i}}^{k} \,\abs*{x_i}^{p-1} \abs*{x_j} \right\}\!.
		\end{equation}
		By finite induction, it suffices to prove~\eqref{eq:xi-p} for~$k=2$. Thus, we will show that for~$p>1$, there exists a constant~$\mathsf{C}_{p}>0$ such that
		\begin{equation}
			\label{eq:toprove-k=2}
			\abs*{x+y}^p \leq \abs*{x}^p + \abs*{y}^p + \mathsf{C}_{p} \left\{\abs*{x}^{p-1} \abs*{y} + \abs*{y}^{p-1} \abs*{x}\right\} \quad\text{for all } x,y \in \R^n.
		\end{equation}
		Indeed, fix a~$\mathsf{k} >2$ and suppose that~\eqref{eq:xi-p} holds for all~$2 \leq k \leq \mathsf{k}$. For~$x_1,\dots,x_{\mathsf{k}+1} \in \R^n$, set
		\begin{equation}
			\label{eq:def-X-red}
			X \coloneqq \sum_{i=1}^{\mathsf{k}} x_i.
		\end{equation}
		Applying~\eqref{eq:xi-p} to the couple~$X$ and~$x_{\mathsf{k}+1}$, we deduce that
		\begin{equation}
			\label{eq:est-xk+1}
			\begin{split}
				\abs*{\sum_{i=1}^{\mathsf{k}+1} x_i}^p &= \abs*{X+x_{\mathsf{k}+1}}^p \\
				&\leq \abs*{X}^p + \abs*{x_{\mathsf{k}+1}}^p + \mathsf{C}_{p,2} \left\{\abs*{X}^{p-1} \abs*{x_{\mathsf{k}+1}} + \abs*{x_{\mathsf{k}+1}}^{p-1} \abs*{X}\right\}\!.
			\end{split}
		\end{equation}
		Moreover, by discrete H\"older inequality, we have
		\begin{equation}
			\label{eq:chain-1}
			\abs*{X}^{p-1} \leq \left(\sum_{i=1}^{\mathsf{k}} \,\abs*{x_i}\right)^{\! p-1} \leq \mathsf{k}^{(p-2)_+} \sum_{i=1}^{\mathsf{k}} \,\abs*{x_i}^{p-1}
		\end{equation}
		where~$(p-2)_+=\max\{p-2,0\}$, and using~\eqref{eq:xi-p} with~$k=\mathsf{k}$ we also get 
		\begin{equation}
			\label{eq:chain-2}
			\abs*{X}^p \leq \sum_{i=1}^{\mathsf{k}} \,\abs*{x_i}^p + \mathsf{C}_{p,\mathsf{k}} \left\{\sum_{i=1}^{\mathsf{k}} \sum_{\substack{j=1 \\ j \neq i}}^{\mathsf{k}} \,\abs*{x_i}^{p-1} \abs*{x_j} \right\}\!.
		\end{equation}
		Chaining~\eqref{eq:chain-1}--\eqref{eq:chain-2} with~\eqref{eq:est-xk+1} and exploiting the triangle inequality, we infer
		\begin{align*}
			\abs*{\sum_{i=1}^{\mathsf{k}+1} x_i}^p &\leq \sum_{i=1}^{\mathsf{k}+1} \,\abs*{x_i}^p + \mathsf{C}_{p,\mathsf{k}} \left\{\sum_{i=1}^{\mathsf{k}} \sum_{\substack{j=1 \\ j \neq i}}^{\mathsf{k}} \,\abs*{x_i}^{p-1} \abs*{x_j} \right\} \\
			&\quad+ \mathsf{C}_{p,2} \left\{\mathsf{k}^{(p-2)_+} \sum_{i=1}^{\mathsf{k}} \,\abs*{x_i}^{p-1}\abs{x_{\mathsf{k}+1}} + \sum_{j=1}^{\mathsf{k}} \,\abs*{x_{\mathsf{k}+1}}^{p-1}\abs*{x_j} \right\} \\
			&= \sum_{i=1}^{\mathsf{k}+1} \,\abs*{x_i}^p + \left[\mathsf{C}_{p,\mathsf{k}} + \left(\mathsf{k}^{(p-2)_+}+1\right) \mathsf{C}_{p,2} \right] \left\{\sum_{i=1}^{\mathsf{k+1}} \sum_{\substack{j=1 \\ j \neq i}}^{\mathsf{k+1}} \,\abs*{x_i}^{p-1} \abs*{x_j} \right\}\!,
		\end{align*}
		which proves~\eqref{eq:xi-p} for~$k=\mathsf{k}+1$.
		
		We are left to verify~\eqref{eq:toprove-k=2}. By the triangle inequality, it suffices to prove that for~$p>1$, there exists a constant~$\mathsf{C}_{p}>0$ such that
		\begin{equation}
			\label{eq:est-ab-toprove}
			(a+b)^p \leq a^p + b^p + \mathsf{C}_{p} \left\{a^{p-1} b + b^{p-1} a\right\} \quad\text{for all } a,b \geq 0.
		\end{equation}
		
		We now observe that if either~$a=0$ or~$b=0$, then~\eqref{eq:est-ab-toprove} is trivial. Moreover since~\eqref{eq:est-ab-toprove} is symmetric in~$a$ and~$b$, we may assume without loss of generality that~$a \geq b>0$, and set~$t \coloneqq b/a \in (0,1]$. As a result,~\eqref{eq:est-ab-toprove} is equivalent to
		\begin{equation}
			\label{eq:est-t-toprove}
			(1+t)^p \leq 1 + t^p + \mathsf{C}_{p} \left(t^{p-1} + t\right) \quad\text{for all } t \in (0,1].
		\end{equation}
		We now define
		\begin{equation*}
			f_p(t) \coloneqq 1 + t^p - (1+t)^p + \mathsf{C}_{p} \left(t^{p-1} + t\right)\!,
		\end{equation*}
		so that~$f_p(0)=0$ and
		\begin{equation}
			\label{eq:fp-der}
			f_p'(t) = p \left[t^{p-1} - \left(1+t\right)^{p-1}\right] + \mathsf{C}_{p} \left[\left(p-1\right)t^{p-2} + 1\right]\!.
		\end{equation}
		Setting~$g_p(t) \coloneqq t^{p-1} - \left(1+t\right)^{p-1}$, it follows that
		\begin{equation*}
			g_p'(t) = (p-1) \left[t^{p-2} - \left(1+t\right)^{p-2}\right] \neq 0 \quad\text{for all } t \in (0,1).
		\end{equation*}
		Therefore, we have
		\begin{equation*}
			g_p(t) \geq \min\left\{g_p(0) , g_p(1) \right\} = 1-2^{\max\left\{1,p-1\right\}} \eqqcolon \mathsf{c}_p <0 \quad\text{for all } t \in [0,1].
		\end{equation*}
		Coupling with~\eqref{eq:fp-der} and taking~$\mathsf{C}_{p} = 1-p \,\mathsf{c}_p>0$ , we obtain
		\begin{equation*}
			f_p'(t) \geq 1+\left(p-1\right) \mathsf{C}_{p} t^{p-2} >0 \quad\text{for all } t \in (0,1] \quad\text{and}\quad \lim_{t\to0^+} f_p'(t) >0.
		\end{equation*}
		As a consequence, we conclude that
		\begin{equation*}
			f_p(t) \geq f_p(0) = 0 \quad\text{for all } t \in [0,1],
		\end{equation*}
		which, in turn, implies the validity of~\eqref{eq:est-t-toprove}.
		
		We now claim that, possibly for a larger~$\mathsf{C}_{p,k}>0$, we also have
		\begin{equation}
			\label{eq:xi-p-2}
			\sum_{i=1}^{k} \,\abs*{x_i}^p \leq \abs*{\sum_{i=1}^{k} x_i}^p + \mathsf{C}_{p,k} \left\{\sum_{i=1}^{k} \sum_{\substack{j=1 \\ j \neq i}}^{k} \,\abs*{x_i}^{p-1} \abs*{x_j} \right\}\!.
		\end{equation}
		Clearly, the latter, together with~\eqref{eq:xi-p}, implies the validity of~\eqref{eq:xi-p-stat}. By finite induction, it suffices to prove~\eqref{eq:xi-p-2} for~$k=2$. Thus, we will show that for~$p>1$, there exists a constant~$\mathsf{C}_{p}>0$ such that
		\begin{equation}
			\label{eq:toprove-k=2-2}
			\abs*{x}^p + \abs*{y}^p \leq \abs*{x+y}^p + \mathsf{C}_{p} \left\{\abs*{x}^{p-1} \abs*{y} + \abs*{y}^{p-1} \abs*{x}\right\} \quad\text{for all } x,y \in \R^n.
		\end{equation}
		Indeed, fix a~$\mathsf{k} >2$ and suppose that~\eqref{eq:xi-p-2} holds for all~$2 \leq k \leq \mathsf{k}$. Let~$X$ be defined by~\eqref{eq:def-X-red}, then~\eqref{eq:xi-p-2} with~$k=\mathsf{k}$ gives
		\begin{equation}
			\label{eq:chain-3}
			\sum_{i=1}^{\mathsf{k}+1} \,\abs*{x_i}^p = \sum_{i=1}^{\mathsf{k}} \,\abs*{x_i}^p + \abs*{x_{\mathsf{k}+1}}^p \leq \abs*{X}^p + \abs*{x_{\mathsf{k}+1}}^p + \mathsf{C}_{p,\mathsf{k}} \left\{\sum_{i=1}^{\mathsf{k}} \sum_{\substack{j=1 \\ j \neq i}}^{\mathsf{k}} \,\abs*{x_i}^{p-1} \abs*{x_j} \right\}\!.
		\end{equation}
		Applying~\eqref{eq:xi-p-2} to the couple~$X$ and~$x_{\mathsf{k}+1}$, we deduce that
		\begin{equation}
			\label{eq:chain-4}
			\abs*{X}^p + \abs*{x_{\mathsf{k}+1}}^p \leq \abs*{\sum_{i=1}^{\mathsf{k}+1} x_i}^p + \mathsf{C}_{p,2} \left\{\abs*{X}^{p-1} \abs*{x_{\mathsf{k}+1}} + \abs*{x_{\mathsf{k}+1}}^{p-1} \abs*{X}\right\}\!.
		\end{equation}
		Chaining~\eqref{eq:chain-3}--\eqref{eq:chain-4} with \eqref{eq:chain-1} and the triangle inequality, we get
		\begin{align*}
			\sum_{i=1}^{\mathsf{k}+1} \,\abs*{x_i}^p &\leq \abs*{\sum_{i=1}^{\mathsf{k}+1} x_i}^p + \mathsf{C}_{p,\mathsf{k}} \left\{\sum_{i=1}^{\mathsf{k}} \sum_{\substack{j=1 \\ j \neq i}}^{\mathsf{k}} \,\abs*{x_i}^{p-1} \abs*{x_j} \right\} \\
			&\quad+ \mathsf{C}_{p,2} \left\{\mathsf{k}^{(p-2)_+} \sum_{i=1}^{\mathsf{k}} \,\abs*{x_i}^{p-1}\abs{x_{\mathsf{k}+1}} + \sum_{j=1}^{\mathsf{k}} \,\abs*{x_{\mathsf{k}+1}}^{p-1}\abs*{x_j} \right\} \\
			&= \abs*{\sum_{i=1}^{\mathsf{k}+1} x_i}^p + \left[\mathsf{C}_{p,\mathsf{k}} + \left(\mathsf{k}^{(p-2)_+}+1\right) \mathsf{C}_{p,2} \right] \left\{\sum_{i=1}^{\mathsf{k+1}} \sum_{\substack{j=1 \\ j \neq i}}^{\mathsf{k+1}} \,\abs*{x_i}^{p-1} \abs*{x_j} \right\}\!,
		\end{align*}
		which proves~\eqref{eq:xi-p-2} for~$k=\mathsf{k}+1$.
		
		We are left to verify~\eqref{eq:toprove-k=2-2}. We start by observing that if either~$x=0$ or~$y=0$, we have nothing to prove. Moreover,~\eqref{eq:toprove-k=2-2} is invariant under the change of~$x$ with~$y$, therefore we may assume that~$0<\abs*{x} \leq \abs*{y}$. We now define the vectors
		\begin{equation*}
			e \coloneqq \frac{y}{\abs*{y}} \in \sfera^n \quad\text{and}\quad \omega \coloneqq \frac{x}{\abs*{y}} \in \overline{B_1} \setminus \{0\}.
		\end{equation*}
		Dividing both sides by~$\abs*{y}^p$,~\eqref{eq:toprove-k=2-2} is equivalent to
		\begin{equation}
			\label{eq:red-1}
			1 + \abs*{\omega}^p \leq \abs*{\omega+e}^p + \mathsf{C}_{p} \left\{\abs*{\omega}^{p-1} + \abs*{\omega}\right\} \!.
		\end{equation}
		Since~$\abs*{e}=1$, we have~$\abs*{\omega+e}^2 = 1+\abs*{\omega}^2+2\left\langle \omega,e\right\rangle \geq 1+\abs*{\omega}^2-2\abs*{\omega} = \left(1-\abs*{\omega}\right)^2$. As a consequence, inequality~\eqref{eq:red-1} follows if we can prove that
		\begin{equation*}
			1 + \abs*{\omega}^p \leq \left(1-\abs*{\omega}\right)^p + \mathsf{C}_{p} \left\{\abs*{\omega}^{p-1} + \abs*{\omega}\right\} \!.
		\end{equation*}
		Setting~$t \coloneqq \abs*{\omega} \in (0,1]$, the latter amounts to
		\begin{equation}
			\label{eq:red-2}
			1 + t^p \leq \left(1-t\right)^p + \mathsf{C}_{p} \left(t^{p-1} + t \right) \quad\text{for all } t \in (0,1].
		\end{equation}
		We now define
		\begin{equation*}
			h_p(t) \coloneqq 1 + t^p - \left(1-t\right)^p - \mathsf{C}_{p} \left(t^{p-1} + t \right)\!,
		\end{equation*}
		so that~$h_p(0)=0$ and
		\begin{equation*}
			h_p'(t) = p \left[t^{p-1} + \left(1-t\right)^{p-1}\right] - \mathsf{C}_{p} \left[\left(p-1\right)t^{p-2} + 1\right]\!.
		\end{equation*}
		Taking~$\mathsf{C}_{p} = 2p+1>0$ , we obtain
		\begin{equation*}
			h_p'(t) < 0 \quad\text{for all } t \in (0,1] \quad\text{and}\quad \lim_{t\to0^+} h_p'(t) <0.
		\end{equation*}
		As a consequence, we conclude that
		\begin{equation*}
			h_p(t) \leq h_p(0) = 0 \quad\text{for all } t \in [0,1],
		\end{equation*}
		which, in turn, implies the validity of~\eqref{eq:red-2}, thereby concluding the proof.

	
	\section*{Acknowledgments} 
	\noindent C.A.A., G.C., and M.G.\ are members of the “Gruppo Nazionale per l'Analisi Matematica, la Probabilità e le loro Applicazioni” (GNAMPA) of the “Istituto Nazionale di Alta Matematica” (INdAM, Italy). C.G.\ and M.G.\ have been partially supported by the “INdAM - GNAMPA Project”, CUP \#E5324001950001\#. C.G.\ has been also partially supported by the by the Research Project of the Italian Ministry of University and Research (MUR) PRIN 2022 “Partial differential equations and related geometric-functional inequalities”, grant number 20229M52AS\_004.
	
	C.A.A. is a postdoctoral fellow of the “Istituto Nazionale di Alta Matematica” (INdAM, Italy) at the University of Florence.
	
	The authors kindly thank Alberto Farina for the discussions they have had over the years on the topic of this manuscript and, in particular, for sharing with them an unpublished draft, coauthored with C.A.A.\ and G.C., containing a preliminary version of the main results of Subsection~\ref{subsec:ineqP}. \newline
	
	\textbf{Data availability statement.} Data sharing not applicable to this article as no datasets were
	generated or analyzed during the current study. \newline
	
	\textbf{Conflict of Interest.} The authors declare that there is no conflict of interest.
	

\end{document}